\documentclass{imsart}

\RequirePackage[OT1]{fontenc} 
\RequirePackage{amsthm,amsmath}
\RequirePackage[numbers]{natbib}
\RequirePackage[colorlinks,citecolor=blue,urlcolor=blue]{hyperref}

\numberwithin{equation}{section}
\usepackage[utf8]{inputenc}
\usepackage{enumerate}

\usepackage{graphicx}
\usepackage{psfrag}

\usepackage{amsfonts}
\usepackage{latexsym,amssymb,epsfig,psfrag}
\usepackage{bm}
\usepackage{stackrel}

\newcommand{\eps}{\varepsilon}

\newcommand{\PP}{\mathbb{P}}
\newcommand{\EE}{\mathbb{E}}

\newcommand{\RR}{\mathbb{R}}

\newcommand{\NN}{\mathbb{N}}
\newcommand{\NNbar}{\overline{\mathbb{N}}}

\newcommand{\MG}{\mathbb{MG}}

\newcommand{\UU}{\mathbb{U}}

\newcommand{\calG}{\mathcal{G}}
\newcommand{\calA}{\mathcal{A}}
\newcommand{\calM}{\mathcal{M}}
\newcommand{\calB}{\mathcal{B}}
\newcommand{\calF}{\mathcal{F}}
\newcommand{\calC}{\mathcal{C}}
\newcommand{\calD}{\mathcal{D}}
\newcommand{\calN}{\mathcal{N}}

\newcommand{\calR}{\mathcal{R}}
\newcommand{\calT}{\mathcal{T}}
\newcommand{\calS}{\mathcal{S}}
\newcommand{\calP}{\mathcal{P}}
\newcommand{\calW}{\mathcal{W}}

\newcommand{\calE}{\mathcal{E}}

\newcommand{\calQ}{\mathcal{Q}}

\DeclareMathOperator{\Frag}{Frag}
\DeclareMathOperator{\CoalFrag}{CoalFrag}
\DeclareMathOperator{\sizes}{masses}
\DeclareMathOperator{\comp}{comp}
\DeclareMathOperator{\dis}{dis}

\DeclareMathOperator{\Coal}{Coal}
\DeclareMathOperator{\diam}{diam}
\DeclareMathOperator{\supdiam}{supdiam}
\DeclareMathOperator{\suplength}{suplength}
\DeclareMathOperator{\surplus}{surplus}
\DeclareMathOperator{\core}{core}
\DeclareMathOperator{\Leb}{Leb}
\DeclareMathOperator{\leb}{leb}
\DeclareMathOperator{\leaves}{leaves}

\newcommand{\II}{\bm{1}}

\newcommand{\Glambda}{\mathcal{G}_\lambda}
\newcommand{\bfX}{{\bm{X}}}
\newcommand{\bfY}{{\bm{Y}}}
\newcommand{\bfR}{{\bm{R}}}
\newcommand{\bfA}{{\bm{A}}}
\newcommand{\bfG}{{\bm{G}}}
\newcommand{\bfS}{{\bm{S}}}
\newcommand{\bfT}{{\bm{T}}}

\newtheorem{defi}{Definition}[section]
\newtheorem{lemm}[defi]{Lemma}
\newtheorem{prop}[defi]{Proposition}
\newtheorem{coro}[defi]{Corollary}
\newtheorem{theo}[defi]{Theorem}

\newtheorem{rem}[defi]{Remark}

\newtheorem{hyp}[defi]{Hypothesis}

\newenvironment{dem}{\vskip 2mm\noindent {\it Proof:} }
                    {\hfill $\square$ \vskip 2mm \noindent} 

\newcommand{\emphdef}[1]{{\bf #1}}

\endlocaldefs

\begin{document}

\begin{frontmatter}

\title{Scaling limit of dynamical percolation\\ on critical Erdös-Rényi random graphs.}
\runtitle{Dynamical percolation on critical random graphs}

\begin{aug}
\author{\fnms{Raphaël} \snm{Rossignol}\ead[label=e1]{raphael.rossignol@univ-grenoble-alpes.fr}}
\runauthor{R. Rossignol}

\affiliation{Univ. Grenoble Alpes, CNRS, Institut Fourier, \\ F-38000 Grenoble, France}

\address{Raphaël Rossignol\\
Université Grenoble Alpes\\
Institut Fourier\\
CS 40700\\
38058 Grenoble cedex 9\\
France\\
\printead{e1}}
\end{aug}

\begin{abstract}
Consider a critical Erd\H{o}s-Rényi random graph: $n$ is the number of vertices, each one of the $\binom{n}{2}$ possible edges is kept in the graph independently from the others with probability $n^{-1}+\lambda n^{-4/3}$, $\lambda$ being a fixed real number. When $n$ goes to infinity, Addario-Berry, Broutin and Goldschmidt \cite{AdBrGolimit} have shown that the collection of connected components, viewed as suitably normalized measured compact metric spaces, converges in distribution to a continuous limit $\Glambda$ made of random real graphs. In this paper, we consider notably the dynamical percolation on  critical Erd\H{o}s-Rényi random graphs. To each pair of vertices is attached a Poisson process of intensity $n^{-1/3}$, and every time it rings, one resamples the corresponding edge. Under this process, the collection of connected components undergoes coalescence and fragmentation. We prove that this process converges in distribution, as $n$ goes to infinity, towards a fragmentation-coalescence process on the continuous limit $\Glambda$. We also prove convergence of discrete coalescence and fragmentation processes and provide general Feller-type properties associated to fragmentation and coalescence.
\end{abstract}



\end{frontmatter}

\tableofcontents
\noindent\hrulefill

\section{Introduction}
Starting with the complete graph with $n$ vertices, $K_n$, the Erd\H{o}s-Rényi random graph $\calG(n,p)$ is the graph obtained from $K_n$ by deleting its edges independently with probability $p$. A well-known phase transition occurs around $p_c=\frac{1}{n}$: when $p=\frac{c}{n}$ with $c<1$, the largest connected component is of order $\log n$, as $n$ goes to infinity, while if  $p=\frac{c}{n}$ with $c>1$, the largest component is of order $n$. This is the so-called {\it appearance of the giant component}, cf. \cite[section 6]{Bollobas}. A even more precise \emph{critical window} was discovered, as early as in the seminal work of Erd\H{o}s and Rényi \cite{ErdosRenyi}: when $p(\lambda,n):=n^{-1}+\lambda n^{-4/3}$, the largest components are of order $n^{2/3}$, and their diameter is of order $n^{1/3}$. Then, when $\lambda$ goes to infinity with $n$, a component starts to dominate the others, and swallows them step by step. Inside this scaling window, i.e for fixed $\lambda$ and large $n$, there is a clean procedure, due to \cite{AdBrGolimit}, to capture the metric structure of those components: if one assigns mass $n^{-2/3}$ to each vertex  and length $n^{-1/3}$ to each edge, the largest components, seen as measured metric spaces, converge to a collection of random $\RR$-graphs $\Glambda$ (see Theorem~\ref{theo:cvGnp} below, or \cite[Theorem~24]{AdBrGolimit} for a more precise statement). Let us mention that we are particularly interested in the large components because in some sense, they contain all the complexity of the graph: small components are with high probability either trees or unicyclic components, cf. \cite[sections 4--6]{Bollobas}. Subsequently, the last decade has seen similar results for critical percolation on other random graphs of mean-field type (See notably \cite{BhamidiBudhirajaWang2014,BhamidiHofstadSen2018,DharaHofstadLeeuwaardenSen2016arXiv,DharaHofstadLeeuwaardenSen2017} and section~\ref{sec:perspectives} for further references). Despite being interesting on its own, this phase transition is also related to the study of the minimal spanning tree on the graph on which percolation is performed, cf. \cite{Braunsteinetal2007,AdBrGoMi2017}. 

In this article, we shall be interested in dynamical versions of the scaling limits just described. To be more precise, put the following dynamic on $\calG(n,p(\lambda,n))$: each pair of vertices is equipped with an independent Poisson process with rate $\gamma_n$ and every time it rings, one \emph{refreshes} the corresponding edge, meaning that one replaces its state by a new independent state: present with probability $p(\lambda,n)$, absent with probability $1-p(\lambda,n)$. This procedure corresponds to \emph{dynamical percolation} on the complete graph with $n$ vertices, at rate $\gamma_n$. A natural question now is ``At which rate should we refresh the edges in order to see a non trivial process in the large $n$ limit ?'' In this question,  it is understood that one remains interested in the same scaling as before concerning masses and lengths: each vertex is assigned mass $n^{-2/3}$ and each edge is assigned  length $n^{-1/3}$.

A moment of thought suggests that a good choice should be $\gamma_n=n^{-1/3}$. Indeed, since large components are of size $\Theta(n^{2/3})$, in a pair of components there are $\Theta(n^{4/3})$ pairs of vertices which after refreshment lead to $\Theta(n^{4/3}p(\lambda,n))=\Theta(n^{1/3})$ edges added. Thus, choosing $\gamma_n=\Theta(n^{-1/3})$ will lead large components to coalesce at rate $\Theta(1)$. Furthermore, the main result of \cite{AdBrGolimit} implies that those large components can be seen, once the graph metric has been divided by $n^{1/3}$, as compact continuous trees with a finite number of additional cycles. Thus, typical distances are of order $\Theta(n^{1/3})$. An edge will be destroyed at rate $\gamma_n(1-p(\lambda,n))$ so the geometry inside such a component will be affected at rate $\Theta(n^{1/3}\gamma_n(1-p(\lambda,n)))$, which is again of order $\Theta(1)$ when $\gamma_n=\Theta(n^{-1/3})$. Of course, instead of refreshing the edges, one may decide to only add edges, or to only destroy edges. In the first case, one will observe coalescence of components and in the second case,  fragmentation. Once again, one may ask the same question as before: what is the right rate in order to obtain a non-trivial process in the large $n$ limit, and what is this limit process ? One of the main purposes of this article is to give an answer to these questions for the three cases that we just defined informally: dynamical percolation, coalescence and fragmentation. The limit processes will be dynamical percolation, coalescence and fragmentation processes acting on the limit $\Glambda$ obtained in \cite{AdBrGolimit}. Furthermore, we will show that coalescence is the time-reversal of fragmentation on this limit. Our approach is to provide Feller-type properties for coalescence and fragmentation, which we hope will be useful in the future to study scaling limits of similar dynamics on other critical random graphs. Notice that the study of coalescence of graphs is a central tool in the work of \cite{BhamidiBroutinSenWang2014arXiv} to show convergence of a number of critical random graphs to $\Glambda$ (configuration models, inhomogeneous random graphs etc.).

Since a large amount of notation is needed in order to make such statements precise we will switch to the presentation of notation in section~\ref{sec:notations} and then announce the main results and outline the plan of the rest of the article in section~\ref{sec:mainresults}. We finish this section by describing informally some works related to coalescence or dynamical percolation.

{\bf Background.} The most important work for the present study is that of Aldous~\cite{Aldous97Gnp}. First, note that there is a natural coupling of  $\{\calG(n,p(\lambda,n)),\;\lambda\in\RR\}$ obtained by assigning i.i.d random variables $(U_e)$ to the edges of the complete graph on $n$ vertices and putting edge $e$ in $\calG(n,p(\lambda,n))$ if and only if $U_e\leq p(\lambda,n)$. Studying the coalescence of $\calG(n,p(\lambda,n))$ at rate $n^{-1/3}$ or the coupled collection $(\calG(n,p(\lambda,n)))_{\lambda \geq 0}$ is essentially equivalent for our purpose. To describe the first major contribution of \cite{Aldous97Gnp}, let us introduce the surplus of a connected graph, which is the minimal number of edges which need to be deleted in order to get a tree. When $\lambda$ is fixed, Aldous proved convergence in distribution of the rescaled masses of the largest components of $\calG(n,p(\lambda,n))$, jointly with their surplus. Aldous' approach  consists in studying the Lukasiewicz walk associated to an exploration process of the graph -- a sequential revealment of the states of the edges following the graph structure -- and to show convergence in distribution of this walk to a Brownian motion with parabolic drift. The second main contribution of \cite{Aldous97Gnp} deals with the dynamics when $\lambda$ increases. Noting that connected components merge at rate proportional to the product of their (rescaled) sizes when $\lambda$ grows, Aldous defined an abstract version of this process, which he called the \emph{multiplicative coalescent}: weighted points merge two by two at a rate proportional to the product of their weights. He managed to define this process when the weights are in $\ell^2$, and proved that the process then satisfies the Feller property for the $\ell^2$-topology. Together with his first main contribution, this implies the convergence of the finite dimensional marginals of the process of rescaled sizes of $\calG(n,p(\lambda,n))$, when $\lambda$ increases. Aldous' work has a large prosperity. We shall only mention two works which go in the direction of tracking the dynamic of the structure of the graph during coalescence, in that they track the dynamic of the surplus (but not of the whole graph structure). Indeed, between the two contributions described above, Aldous has lost the dynamic of the surplus. In \cite{BhamidiBudhirajaWang2014}, the authors enrich Aldous' multiplicative coalescent by taking into account the dynamic of the surplus of the connected components. This leads to what they call the \emph{augmented coalescent}, which they prove to satisfy the Feller property  for a topology that we will not use in this paper (see however section~\ref{sec:perspectives} for more comments). This allows to show convergence of the finite dimensional distributions of the processes of rescaled sizes and surplus of $\{\calG(n,p(\lambda,n)),\;\lambda\geq 0\}$, and even of other random graphs, namely Achlioptas processes with a bounded-size rule. Finally, in \cite{BroutinMarckert2016}, the authors manage to prove the convergence of a sequence of two-parameter processes, where the first parameter is the exploration parameter of $\calG(n,p(\lambda,n))$ and the second is $\lambda$. This allows them to obtain the convergence of the process (in $\lambda$) of rescaled sizes of $\calG(n,p(\lambda,n))$ together with their surplus, this time in the Skorokhod sense (not only in the sense of finite dimensional marginals). In order to get this convergence, they define the exploration process using the \emph{Prim order} on vertices of the complete graph. This order is consistant in $\lambda$, in the sense that connected components are always intervals of the Prim order, and when $\lambda$ increases, only adjacent intervals coalesce. Unfortunately, the Prim order seems to be inconsistent with the internal structure of the graph (see  \cite[section~4.1]{BroutinMarckert2016} and notably the remarks after Theorem~4). This approach is therefore not adapted for the purpose of the present article.

Let us finish this short review of related works by focusing on dynamical percolation. This theme was introduced in \cite{HaggstromPeresSteif97}, and studied in a number of subsequent works by various authors. In the context of \cite{HaggstromPeresSteif97}, only the edges of some fixed infinite graph are resampled while in the definition above, we resample the edges of a finite complete graph. The scaling limit of dynamical percolation for critical percolation on the two dimensional triangular lattice was obtained in \cite{GarbanPeteSchramm2018}, with techniques quite different from the ones used in the present paper. More related to the present paper is the work \cite{RobertsSengul2018}, where dynamical percolation on critical Erd\H{o}s-Rényi random graphs, as introduced above, is studied notably at rate $1$. The authors show that the size of the largest connected component that appears during the time interval $[0,1]$ is of order $n^{2/3}\log^{1/3}n$ with probability tending to one as $n$ goes to infinity. They also study ``quantitative noise-sensitivity'' of the event $A_n$ that the largest component of $\calG(n,p(\lambda,n))$ is of size at least $an^{2/3}$ for some fixed $a>0$ (see \cite[Proposition~2.2]{RobertsSengul2018}). The results in the present paper can be used to find the precise scaling of quantitative noise-sensitivity for events concerning the sizes of the largest components (like $A_n$ for instance). However, we leave this question and precise statements for future work.

\section{Notation and Background}
\label{sec:notations}
\subsection{General notation}
\label{subsec:gennot}
If $(X,\tau)$ is a topological space, we denote by $\calB(X)$ the Borel $\sigma$-field on $X$.

If $\psi$ is a measurable map between $(E,\calE)$ and $(F,\calF)$, and $\mu$ is a measure on $(E,\calE)$, then we denote by $\psi\sharp\mu$ the push-forward of $\mu$ by $\psi$: $\psi\sharp\mu(A)=\mu(\psi^{-1}(A))$ for any $A\in \calF$.

We shall frequently use Poisson processes. Let $(E,\calE,\mu)$ be a measurable set with $\mu$ a $\sigma$-finite measure. Denote by $\Leb(\RR^+)$ the Lebesgue $\sigma$-field on $\RR^+$, $\leb_{\RR^+}$ the Lebesgue measure on $\RR^+$ and let $\gamma\geq 0$. If $\calP$ is a Poisson random set with intensity $\gamma$ on $(E\times\RR^+,\calE\times \Leb(\RR^+),\mu\times \leb_{\RR^+})$ (that is with intensity measure $\gamma\mu\otimes \leb_{\RR^+}$)  we shall denote by $\calP_t$ the multiset containing the points of $\calP$ with birthtime at most $t$, with multiplicity the number of times they appear before $t$.
$$\calP_t:=\{\{x\in E:\exists s\leq t,\;(x,s)\in \calP\}\}\;.$$
Notice that $\calP_t$ is in general a multiset, not a set, but is a set if $\mu$ is diffuse. One can equivalently see $\calP_t$ as a counting measure on $E$. The disjoint union of two multisets $A$ and $B$ will be denoted by $A\sqcup B$.

When $(M,d)$ is a Polish space, let $\calF([0,\infty),M)$ (resp. $\calD([0,\infty),M)$) be the set of functions (resp. càdlàg functions) from $\RR^+$ to $M$. We shall use two topologies on $\calF([0,\infty),M)$ and $\calD([0,\infty),M)$: the Skorokhod topology and the \emph{topology of compact convergence} (also known as topology of uniform convergence on compact sets)  which is finer than Skorokhod's topology. Although it is not crucial to use this topology, it turns out that it is more natural, in our setting, to make approximations with this topology, which has furthermore the advantage that the limit of càdlàg functions are càdlàg. Convergence in distributions for random processes will however be obtained for the Skorokhod topology. Everything (very little in fact) needed for these topologies is gathered in Appendix~\ref{sec:compactcv}. Furthermore, we shall always suppose that our processes are defined on a complete probability space $(\Omega,\calF,\PP)$ (completing the original space if necessary). We shall occasionnally need $\PP^*$,  the outer measure associated to $\PP$ defined on $\calP(\Omega)$ by
\[\PP^*(A):=\inf\{\PP(B)\,:\,A\subset B\text{ and }B\in \calF\}\;.\]
We shall use the notation $\NN$ for the natural numbers (including $0$), $\NN^*$ for the positive natural numbers and $\NNbar:=\NN\cup\{+\infty\}$. We shall also adopt the convention that the infimum over an empty set equals $+\infty$.

Finally, let us define, for $p\geq 1$:
$$\ell^p_+:=\left\{x\in(\RR^+)^{\NN^*}:\sum_{i\geq 1}|x_i|^p<\infty\right\}\;,$$
and
$$\ell^p_\searrow:=\{x\in\ell^p_+:x_1\geq x_2\geq \ldots\}\;.$$

\subsection{Discrete graphs and dynamical percolation}
\label{subsec:discretedynperco}
We will talk of a \emph{discrete graph} to mean the usual graph-theoretic notion of an unoriented graph, that is a pair $G=(V,E)$ with $V$ a finite set and $E$ a subset of $\binom{V}{2}:=\{ \{u,v\}:u\not=v\in V\}$. Often, $E$ is seen as a point in $\{0,1\}^{\binom{V}{2}}$, where $0$ codes for the absence of the corresponding edge and $1$ for its membership to $E$. 

For a positive integer $n$ and $p\in[0,1]$, the \emph{Erd\H{o}s-Rényi random graph} (or Gilbert random graph) $\calG(n,p)$ is the random graph with vertices $[n]:=\{1,\ldots,n\}$ such that each edge is present with probability $p$, independently from the others. Alternatively, one may see it as a Bernoulli bond percolation with parameter $p$ on the complete graph with $n$ vertices $K_n=([n],\binom{[n]}{2})$. 

 Let $\gamma_+$ and $\gamma_-$ be non-negative real numbers. If $G=([n],E)$ is a discrete graph on $n$ vertices, define a random process $N_{\gamma_+,\gamma_-}(G,t)=(G,E_t)$, $t\geq 0$ as follows on the set of subgraphs of the complete graph $K_n$. To each pair $e\in\binom{[n]}{2}$, we attach two Poisson processes on  $\RR^+$: $\calP^{+}_e$ of intensity $\gamma_+$ and $\calP^{-}_e$ of intensity $\gamma_-$. We suppose that all the $2\binom{[n]}{2}$ Poisson processes are independent. Each time  $\calP^{+}_e$ rings, we replace $E_{t^-}$ by $E_{t^-}\cup\{e\}$ (nothing changes if $e$ already belongs to $E_{t^-}$), and each time $\calP^{-}_e$ rings, we replace $E_{t^-}$ by $E_{t^-}\setminus\{e\}$ (nothing changes if $e$ doe not belong to $E_{t^-}$). The letter $N$ is reminiscent of ``noise''. If one wants to insist on the Poisson processes, we shall write $N(G,(\calP^+,\calP^-)_t)$ instead of $N_{\gamma_+,\gamma_-}(G,t)$, with an implicit definition for the map $N$.

One may take only $\calP^+$ or only $\calP^-$ into account: write $N^+(G,\calP^+_t)$ for $N(G,(\calP^+,\emptyset)_t)$ and  $N^-(G,\calP^+_t)$ for $N(G,(\emptyset,\calP^-)_t)$. Then, $N^+(G,\calP^+_t)$ will be referred as \emph{the discrete coalescent process  of intensity $\gamma^+$ started at $G$} and $N^-(G,\calP^+_t)$ as \emph{the discrete fragmentation process  of intensity $\gamma^+$ started at $G$}. 

Now, \emph{dynamical percolation of parameter $p$ and intensity $\gamma$}, as described in the introduction, corresponds to the process  $N_{\gamma p,\gamma (1-p)}$, and is in its stationary state when started with $\calG(n,p)$ (independently of the Poisson processes used to define the dynamical percolation). 

All these processes will have continuous couterparts, which will be defined in sections~\ref{subsec:gluing} and \ref{subsec:frag}.

\subsection{Measured semi-metric spaces}
The main characters in this article are the connected components of Erd\H{o}s-Rényi random graphs and their continuum limit, each one undergoing the updates due to dynamical percolation. One task is therefore to define a proper space where those characters can live, and first to state precisely what we mean by ``the connected components of a graph'' seen as a single object. One option is to order the components by decreasing order of size\footnote{It requires some device to break ties, but those disappear in the continuum limit, for the Erd\H{o}s-Rényi random graphs at least.}, as in \cite{AdBrGolimit}, or in a size-biased way, as in \cite{Aldous97Gnp}, and thus see the collected components of a graph as a sequence of graphs. However, this order is not preserved under the process of dynamical percolation. Also, looking only at the mass to impose which graphs are pairwise compared between two collections of graphs might lead to a larger distance than what one would expect. Indeed, suppose that $(G_1,G_2)$ and $(G'_1,G'_2)$ are two pairs of graphs, with $G_1$ (resp. $G'_1$) having slightly larger mass than $G_2$ (resp. $G'_2$). One might have $G_1$ close to $G'_2$ and $G_2$ close to $G'_1$ in some topology (the Gromov-Hausdorff-Prokhorov topology to be defined later), but $G_1$ far from $G'_1$ in this topology. For all these reasons, I found it somewhat uncomfortable to work with such a topology in the dynamical context. The topology we will use will be defined in section~\ref{subsec:GHP}, and the story begins with the definition of a semi-metric space.

One way to present the connected components of a graph is to consider the graph as a metric space using the usual graph distance,  allowing the metric to take the value $+\infty$ between points which are not in the same connected component, as in \cite{BuragoBuragoIvanov}, page~1. In addition, the main difficulty in defining dynamical percolation on the continuum limit will be in defining coalescence. In this process some points will be identified, and one clear way to present this is to modify the metric and allow it to be equal to zero between different points rather than performing the corresponding quotient operation. This type of space is called a semi-metric space in \cite{BuragoBuragoIvanov}, Definition 1.1.4, and we shall stick to this terminology.

\begin{defi}
A \emphdef{semi-metric space} is a pair $(X,d)$ where $X$ is a non-empty set and $d$ is a function from $X\times X$ to $\RR^+\cup\{+\infty\}$ such that for all $x$, $y$ and $z$ in $X$:
\begin{itemize}
\item $d(x,z)\leq d(x,y)+d(y,z)$,
\item $d(x,x)=0$,
\item $d(x,y)=d(y,x)$.
\end{itemize}
A semi-metric space $(X,d)$ is a \emphdef{metric space} if in addition
\begin{itemize}
\item $d(x,y)=0\Rightarrow x=y$. 
\end{itemize}
A metric or semi-metric space $(X,d)$ is said to be \emphdef{finite} if $d$ is finite.
\end{defi}
Of course, when thinking about a semi-metric space $(X,d)$, one may visualize the \emph{quotient metric space} $(X/d,d)$ where points at  null distance are identified. $X$ and $X/d$ are at zero Gromov-Hausdorff distance (we shall use a version of Gromov-Hausdorff distance extended to semi-metric spaces, defined in section~\ref{subsec:GHP} below). Notice that $(X,d)$ is not necessarily a Hausdorff space (different points cannot always be separated by disjoint neighborhoods), but $(X/d,d)$ always is. Furthermore, $(X,d)$ is separable if and only if $(X/d,d)$ is separable.

\begin{defi}
If $(X,d)$ is a semi-metric space, the relation $\mathcal{R}$ defined by:
$$x\mathcal{R} y\Leftrightarrow d(x,y)<\infty$$
is an equivalence relation. Each equivalence class is called a \emphdef{component} of $(X,d)$ and $\comp(X,d)$ denotes the set of components. We denote by $\diam(X)$ the diameter of $(X,d)$:
$$\diam(X)=\sup_{x,y\in X} d(x,y)$$
and by $\supdiam (X)$ the supremum of the diameters of its components:
$$\supdiam(X)=\sup_{m\in\comp(X,d)}\diam(m)\;.$$
\end{defi}

\begin{defi}
A \emphdef{measured semi-metric space} (m.s-m.s) is a triple $\bfX=(X,d,\mu)$ where $(X,d)$ is a semi-metric space and $\mu$ is a measure on $X$ defined on a $\sigma$-field containing the Borel $\sigma$-field for the topology induced by $d$. 

An m.s-m.s $(X,d,\mu)$ is said to be \emphdef{finite} if $(X,d)$ is a finite totally bounded semi-metric space and  $\mu$ is a finite measure. 

Finally, we define $\comp(\bfX):=\comp(X,d)$ and 
$$\sizes(\bfX):=(\mu(m))_{m\in\comp(\bfX)}\;.$$
\end{defi}
Notice that a finite m.s-m.s has only one component.
\begin{rem}
  \begin{enumerate}[(i)]
    \item The reason why we allow $\mu$ to be defined on a larger field than the Borel $\sigma$-field is the following. We want to keep $X$ and $\mu$ unchanged along coalescence, only the semi-metric will change, say to a new semi-metric $d'$, by performing various identifications. When one performs identifications, the topology shrinks: there are less and less open sets. Thus, the Borel $\sigma$-field shrinks too, and the original measure $\mu$, defined on the original Borel $\sigma$-field assocated to $d$, is now defined on a $\sigma$-field which is larger than the Borel $\sigma$-field associated to $d'$.
    \item One might feel more comfortable after realizing the following. Let $\pi$ denote the projection from $(X,d)$ to $X':=X/d$, $\calB'$ the Borel $\sigma$-field on $X'$ and $\calB$ the Borel $\sigma$-field on $X$. Then, $\pi^{-1}(\calB')=\calB$ and the image measure $\pi\sharp\mu$ on $X'$ is a Borel measure.
      \end{enumerate}
\end{rem}

\subsection{The Gromov-Hausdorff-Prokhorov distance}
\label{subsec:GHP}
In the introduction, we mentioned that $\calG(n,p(\lambda,n))$ converges in distribution, but we did not mention precisely the underlying topology. The main topological ingredient in \cite{AdBrGoMi2017} is the Gromov-Hausdorff-Prokhorov distance between two components of the graph, and we shall use this repeatedly. To define it, we need to recall some definitions from \cite{AdBrGoMi2017}.

If $\bfX=(X,d,\mu)$ and $\bfX'=(X',d',\mu')$ are two measured semi-metric spaces
 a \emph{correspondence} $\calR$ between $X$ and $X'$ is a measurable subset of $X\times X'$ such that:
$$\forall x\in X,\;\exists x'\in X':(x,x')\in\calR$$
and
$$\forall x'\in X',\;\exists x\in X:(x,x')\in\calR\;.$$
We let $C(X,X')$ denote the set of correspondences  between $X$ and $X'$. The \emph{distortion} of a correspondence $\calR$ is defined as
$$\dis(\calR):=\inf\left\{\eps >0:\forall (x,x'),(y,y')\in\calR,\; \left(\begin{array}{c}d(x,y)\leq d'(x',y')+\eps \\ \text{ and }\\d'(x',y')\leq d(x,y)+\eps\end{array}\right)\right\}$$

The Gromov-Hausdorff distance between two semi-metric spaces $(X,d)$ and $(X',d')$ is defined as:
$$d_{GH}((X,d),(X',d')):=\inf_{\calR\in C(X,X')}\frac{1}{2}\dis(\calR)\;.$$

We denote by $M(X,X')$ the set of finite Borel measures on $X\times X'$. For $\pi$ in $M(X,X')$, we denote by $\pi_1$ (resp. $\pi_2$) the first (resp. the second) marginal of $\pi$. For any $\pi\in M(X,X')$, and any finite measures $\mu$ on $X$ and $\mu'$ on $X'$ one defines:
$$D(\pi;\mu,\mu')=\|\pi_1-\mu\|+\|\pi_2-\mu'\|$$
where $\|\nu\|$ is the total variation of a signed measure $\nu$.

The \emph{Gromov-Hausdorff-Prokhorov distance}  is defined as follows in \cite{AdBrGoMi2017}.
\begin{defi}
\label{defi:GHPMiermont}
If $\bfX=(X,d,\mu)$ and $\bfX'=(X',d',\mu')$ are two m.s-m.s, the Gromov-Hausdorff-Prokhorov distance between them is defined as:
$$d_{GHP}(\bfX,\bfX')=\inf_{\substack{\pi\in M(X,X')\\ \calR\in C(X,X')}}\{D(\pi;\mu,\mu')\lor \frac{1}{2}\dis(\calR)\lor \pi(\calR^c)\}\;.$$
\end{defi}
It is not difficult to show that $d_{GHP}$ satisfies the axioms of a semi-metric. Let us give a bit more intuition to what the Gromov-Hausdorff-Prokhorov distance measures. On a semi-metric space $(X,d)$, let us denote by $\delta_H$ the Hausdorff distance and by $\delta_{LP}$ the Lévy-Prokhorov distance. Let us recall their definition. For $B\subset X$ and $\eps>0$, let
$$B^\eps:=\{x\in X:\exists y\in B,\;d(x,y)<\eps\}\;.$$
Now, for $A$ and $B$ subsets of $X$,
$$\delta_H(A,B):=\inf\{\eps>0: A\subset B^\eps\text{ and }B\subset A^\eps\}$$
and for finite measures $\mu$ and $\nu$ on $X$,
\begin{equation}
  \label{eq:defLP}
  \delta_{LP}(\mu,\nu):=\inf\left\{\eps>0:\forall B\in\calB(X),\;\left(\begin{array}{c}\mu(B)\leq \nu(B^\eps)+\eps\\\text{ and }\\\nu(B)\leq \mu(B^\eps)+\eps\end{array}\right)\right\}\;.
  \end{equation}

The following lemma shows that the Gromov-Hausdorff-Prokhorov distance measures how well two measured semi-metric spaces can be put in the same ambient space so that simultaneously their measures are close in Prokhorov distance and their geometries are close in Hausdorff distance. It shows that the definitions of \cite{AdBrGoMi2017} and \cite{AbrahamDelmasHoscheit2013} are equivalent. Its proof is a small variation on the proof of \cite[Proposition~6]{MiermontTessellations09}, where only probability measures were considered, so we leave it to the reader.
\begin{lemm}
\label{lemm:MiermontHoscheit}
If $\bfX=(X,d,\mu)$ and $\bfX'=(X',d',\mu')$ be two measured separable semi-metric spaces, let
$$\tilde d_{GHP}(\bfX,\bfX'):=\inf_{d''}\{\delta_H(X,X')\lor \delta_{LP}(\mu,\mu')\}$$
where the infimum is over all semi-metric $d''$ on the disjoint union $X \cup X'$ extending $d$ and $d'$. Then,
$$\frac{1}{2}\tilde d_{GHP}(\bfX,\bfX') \leq d_{GHP}(\bfX,\bfX') \leq \tilde d_{GHP}(\bfX,\bfX')\;.$$
\end{lemm}

It is easy to see that two m.s-m.s $\bfX$ and $\bfX'$  are at zero $d_{GHP}$-distance if and only if there are two distance and measure preserving maps $\phi$ and $\phi'$ such that $\phi$ is a map from $\bfX$ to $\bfX'$ and $\phi'$ a map from $\bfX'$ to $\bfX$. Let $\calC$ denote the class of finite measured semi-metric spaces and $\calR$ the equivalence relation on $\calC$ defined by $\bfX\calR \bfX'\Leftrightarrow d_{GHP}(\bfX,\bfX')=0$. The quotient $\calC / \calR$ can be seen as a set (cf. appendix \ref{sec:calM}), and we denote this set by the letter $\calM$, which stands thus for \emph{the  set of isometry classes of finite measured semi-metric spaces}. The following is shown in \cite{AbrahamDelmasHoscheit2013}.
\begin{theo}
$(\calM,d_{GHP})$ is a complete separable metric space.
\end{theo}

Now, the Gromov-Hausdorff-Prokhorov distance in Definition \ref{defi:GHPMiermont} is too strong for our purposes when applied to m.s-m.s which have an infinite number of components: it essentially amounts to a uniform control of the $d_{GHP}$-distance between paired components. We are interested in a weaker distance which localizes around the largest components. We shall restrict to countable unions of finite semi-metric spaces with the additional property that for any $\eps>0$, there are only a finite number of components whose size exceeds $\eps$. To formulate the distance, it will be convenient to view those semi-metric spaces as a set of counting measures on $\calM$.

\begin{defi}
\label{defi:calN}For any $\eps > 0$, let 
$$\calM_{>\eps}=\{[(X,d,\mu)]\in\calM\mbox{ s.t. }\mu(X)> \eps\}\;.$$
For any  counting measure $\nu$ on $\calM$, denote by $\nu_{>\eps}$ the restriction of $\nu$ to $\calM_\eps$. Denote by $\calN$ the set of counting measures $\nu$ on $\calM$ such that for any $\eps>0$, $\nu_{>\eps}$ is a finite measure
and such that $\nu$ does not have atoms of mass $0$, that is:
$$\nu(\{[(X,d,\mu)]\in\calM\mbox{ s.t. }\mu(X)=0\})=0$$
When $\bfX$ is a measured semi-metric space whose components are finite, we denote by $\nu_{\bfX}$ the counting measure on $\calM$ defined by
$$\nu_{\bfX}:=\sum_{m\in\comp(\bfX)}\delta_{[m]}$$
Abusing notation, we shall say that $\bfX\in\calN$ if $\nu_{\bfX}$ belongs to $\calN$, and we shall denote by $\bfX_{>\eps}$  the disjoint union of components of $\bfX$ whose masses are larger than $\eps$.
\end{defi}
Notice that $\bfX$ is in $\calN$ if and only if it has an at most countable number of components, each one of its components has positive mass and for any $\eps>0$, $\bfX_{>\eps}$ is the disjoint union of a finite number of components, each one being totally bounded and equipped with a finite measure. 

Essentially, we want to define a metric on $\calN$ such that a sequence $\nu_n$ converges to $\nu$ if and only if for any $\eps>0$, the components of $\nu$ of mass larger than $\eps$ are close (for $d_{GHP}$) to the components of $\nu_n$ of mass larger than $\eps$, for $n$ large enough. The idea is similar to the metrization of vague convergence for measures on $\RR^d$, with components of mass zero playing the role of infinity. In this vein, there is an  abstract notion of locally finite measure in \cite{KallenbergRandomMeasures}, but we shall note use it.

First, let $\delta_{LP}$ be the Lévy-Prokhorov distance on the set of finite measures on the metric space $(\calM,d_{GHP})$. Recall the definition of this distance in \eqref{eq:defLP}. Now, for $\bfX=[(X,d,\mu)]\in\calM$ and $k\geq 1$, define a function $f_k$ by

$$f_k(\bfX):=\left\lbrace\begin{array}{ll}
1&\text{ if }\mu(X)\geq \frac{1}{k}\\
k(k+1)\left(\mu(X)-\frac{1}{k+1}\right)&\text{ if }\mu(X)\in \left[\frac{1}{k+1},\frac{1}{k}\right[\\
0&\text{ if }\mu(X)<\frac{1}{k+1}
\end{array}\right.
$$

The following distance is an analogue of the distance in \cite[Lemma~4.6]{KallenbergRandomMeasures} where it is used to metrize the vague topology on locally finite measures.
\begin{defi}
  \label{defi:dGHPcountableproduct}
  If $\nu$ and $\nu'$ are counting measures on $\calM$, then we define:
  $$L_{GHP}(\nu,\nu'):=\sum_{k\geq 1}2^{-k}\{1\land \delta_{LP}(f_k\nu,f_k\nu')\}\;,$$
 where $f_k\nu$ is defined as follows for $\nu=\sum_{i\in I}\delta_{x_i}$:
  $$f_k\nu:=\sum_{i\in I}f_k(x_i)\delta_{x_i}\;.$$
\end{defi}
We shall prove later, in Proposition~\ref{prop:Npolish} that $(\calN,L_{GHP})$ is a complete separable metric space. Notice that any m.s-m.s of $\calN$ is at zero $L_{GHP}$-distance from a m.s-m.s whose components are compact metric spaces. In this article, we really are interested in equivalence classes of m.s-m.s for the equivalence relation ``being at zero $L_{GHP}$-distance'', although in order to define random processes such as coalescence and fragmentation, it will be convenient  to have in mind a particular representative of such a class.

\subsection{Gluing and coalescence}
\label{subsec:gluing}
\subsubsection{Gluing and $\delta$-gluing}
Gluing corresponds to identification of points which can belong to the same semi-metric space or to different semi-metric spaces. A formal definition is as follows, for a single semi-metric space (see also \cite[pp. 62--64]{BuragoBuragoIvanov}). 
\begin{defi}
Let $(X,d)$ be a semi-metric space and $\calR$ be an equivalence relation on $X$. The \emphdef{gluing of $(X,d)$ along $\calR$}, is the semi-metric space $(X,d_{\calR})$ with semi-metric defined on $X^2$ by
$$d_{\calR}(x,y) := \inf\{\sum_{i=0}^k d(p_i,q_i) \;:\; p_0 = x, q_k = y, k \in\NN\}$$
where the infimum is taken over all choices of $\{p_i\}_{0\leq i\leq k}$ and $\{q_i\}_{0\leq i\leq k}$ such that $(q_i,p_{i+1})\in\calR$ for all $i= 0,\ldots,k-1$.
\end{defi}

When performing dynamical percolation on a dicrete graph, edges appear, and these are not of length zero. Thus one needs a definition of gluing which leaves the possibility to add those edges. We shall define the $\delta$-gluing of a semi-metric space $X$ along a multiset $\tilde\calR$ with elements in $X^2$ as the operation of joining every pair $(x,x')\in \tilde\calR$ by an isometric copy of the interval $[0,\delta]$. If $(X,d)$ and $(X',d')$ are two semi-metric spaces, let us denote by $d_{X\sqcup X'}$ the \emph{disjoint union semi-metric} on the disjoint union $X\sqcup X'$, which is the semi-metric equal to $d$ on $X\times X$, to $d'$ on $X'\times X'$ and to $+\infty$ on $(X\times X')\cup (X'\times X)$. This notion extends trivially to the disjoint union of a collection of metric spaces.
\begin{defi}
  \label{defi:deltagluing}
  Let $(X,d)$ be a semi-metric space, $\tilde\calR$ be a multiset of elements in $X^2$ and $\delta\geq 0$. 
\begin{itemize}
\item If $\delta=0$, let $\calR$ denote the equivalence relation generated by $\tilde\calR$ and let $X_{\tilde\calR,\delta}=X$ and $d'=d$.
\item If $\delta>0$, for every pair $(x,x')\in\tilde\calR$, let $I_{x,x'}=[a_{x,x'},b_{x,x'}]$ be an isometric copy of $[0,\delta]$. Let us denote by $X_{\tilde\calR,\delta}$ the disjoint union of $X$ and all the $I_{x,x'}$, for $(x,x')$ in $\tilde\calR$ and let $d'$ denote the disjoint union distance on $X_{\tilde\calR,\delta}$. Let $\calR$ denote the equivalence relation on $X_{\tilde\calR,\delta}$ generated by
  \[\bigcup_{(x,x')\in\tilde\calR}\{(x,a_{x,x'}),(x',b_{x,x'})\}\;.\]
\end{itemize}
Then, the $\delta$-gluing of $(X,d)$ is the metric space $(X_{\tilde\calR,\delta},d_{\tilde\calR,\delta})$ which is the gluing of $(X_{\tilde\calR,\delta},d')$ along $\calR$.

When $\delta>0$, let us denote by $\calF$ the following $\sigma$-field on $X_{\tilde\calR,\delta}$:
\[\calF:=\{A\cup B\,:\,A\in \calB(X),\,B\subset \bigcup_{(x,x')\in\tilde\calR}I_{x,x'}\}\]
We can lift trivally any measure $\mu$ from $\calB(X)$ on $\calF$ as a measure $\tilde\mu$ on $X$ by:
\[\tilde\mu(C):=\mu(C\cap X)\]
When $\bfX=(X,d,\mu)$ is a measured semi-metric space, we equip the $\delta$-gluing of $(X,d)$ along $\tilde\calR$ with this lifted measure, but we shall still denote this measure by $\mu$ and denote the resulting semi-metric space by $\Coal_\delta(\bfX,\tilde\calR)=(X_{\tilde\calR,\delta},d_{\tilde\calR,\delta},\mu)$.
\end{defi}
\begin{rem}
\label{rem:deltagluing}
\begin{enumerate}[(i)]
\item For any $(x,y)\in X^2$, if $\delta>0$, one may see that
$$d_{\tilde\calR,\delta}(x,y) = \inf\{k\delta+\sum_{i=0}^k d(p_i,q_i) \;:\; p_0 = x, q_k = y, k \in\NN\}\;.$$
where the infimum is taken over all choices of $\{p_i\}_{0\leq i\leq k}$ and $\{q_i\}_{0\leq i\leq k}$ such that $(q_i,p_{i+1})\in\tilde\calR$ for all $i= 0,\ldots,k-1$. Furthermore, the same is true for $\delta=0$ if $\tilde\calR$ is already an equivalent relation.
\item The previous expressions allow to prove that for $A,B\subset X^2$,
  \[\Coal_{0}(\bfX,A\cup B)=\Coal_{0}(\Coal_{0}(\bfX,A),B)\;.\]
  Also, for any multisets $A$, $B$ and for any $\delta>0$,
  \[\Coal_{\delta}(\bfX,A\sqcup B)=\Coal_{\delta}(\Coal_{\delta}(\bfX,A),B)\;.\] See appendix~\ref{sec:commutcoalfrag}.
\item If $\delta>0$, one may like to consider the space $X$ with (the restriction of) the metric $\delta_{\tilde\calR,\delta}$: it can be seen as forgetting the interiors of the intervals $I_{x,x'}$ that have been added in $X_{\tilde \calR,\delta}$ (while keeping the same metric). If $\bfX=(X,d,\mu)\in\calN$, 
  \begin{equation}
  \label{eq:forgetinterior}L_{GHP}((X,d_{\tilde\calR,\delta},\mu),\Coal_\delta(\bfX,\tilde\calR))\leq \delta\;.
  \end{equation}
  Indeed, using the notations of Definition~\ref{defi:deltagluing}, one may use the following correspondance on $X\times X_{\tilde\calR,\delta}$:
  $$C=X^2\cup\bigcup_{(x,x')\in\tilde\calR}\{(x,y)\,:\,y\in I_{x,x'}\}$$
  which has distortion $\delta$ when $X$ is equipped with $d_{\tilde\calR,\delta}$, and then use $\pi$ as the trivial coupling of $\mu$ on $X$ and $\mu$ on $X_{\tilde\calR,\delta}$, which satisfies
  \[D(\pi;\mu,\mu)=\pi(C^c)=0\;.\]
  This can be done on each component of $(X,d_{\tilde\calR,\delta})$ separately to show that~\eqref{eq:forgetinterior} holds.
\item It is easy to see that for any $\bfX=(X,d,\mu)$ and $\tilde R\subset X^2$,
  \[\Coal_0(\bfX/d,\tilde R/d)/d_{\tilde\calR,0}=\Coal_0(\bfX,\tilde R)/d_{\tilde\calR,0}\]
  thus one may without harm identify the semi-metric spaces with their metric quotient before or after coalescence (in fact the collection of the metric quotient of its components). 
\end{enumerate}
\end{rem}

\subsubsection{The coalescence processes}
\label{subsec:coal}
When $(X,d,\mu)$ is a measured semi-metric space, there is a natural coalescence process (of mean-field type) which draws pairs of points $(x,y)$ with intensity $\mu(dx)\mu(dy)$ (and unit intensity in time) and identifies points $x$ and $y$, changing the metric accordingly. To describe the process of addition of edges during the dynamical percolation on the Erd\H{o}s-Rényi random graph, one needs to replace the identification of $x$ and $y$ by the fact that the distance between $x$ and $y$ drops to $n^{-1/3}$ (if $x\neq y$). This leads to the following definition.
\begin{defi}
\label{defi:coal}
Let $\bfX=(X,d,\mu)$ be an m.s-m.s with $\mu$ sigma-finite and $\delta\geq 0$. Let $\calP^+$ be a Poisson random set on $X^2\times \RR^+$ of intensity measure $\frac{1}{2}\mu^{2}\times \leb_{\RR^+}$. The \emphdef{coalescence process with edge-lengths $\delta$ started from $\bfX$,} denoted by $(\Coal_{\delta}(\bfX,t))_{t\geq 0}$, is the random process of m.s-m.s $(\Coal_{\delta}(\bfX,\calP^+_t))_{t\geq 0}$.
\end{defi}
Notice that this process inherits the strong Markov property from the strong Markov property of the Poisson process, and the fact that for multisets $A$ $B$ of elements of $X^2$, $\Coal_{\delta}(\bfX,A\sqcup B)=\Coal_{\delta}(\Coal_{\delta}(\bfX,A),B)$, cf. Remark~\ref{rem:deltagluing} (ii).

When $\delta>0$, if one wants to keep the space fixed and change only the metric, Remark~\ref{rem:deltagluing} (iii) shows that one can do so at the price of an $L_{GHP}$-distance at most $\delta$. In this paper, one wants typically to understand scaling limits of $N^+(G_n,\calP^+_t)$ as defined in section~\ref{subsec:discretedynperco} with $\calP^+$ of intensity $\gamma_n$ and $G_n$ a discrete graph equipped with the distance $d_n$ which is the graph distance multiplied by some $\delta_n>0$ going to zero as $n$ goes to infinity. See for instance Theorem \ref{theo:cvcoalGnp} below. If one equips $G_n$ and $N^+(G_n,\calP_t^+)$ with their counting measures multiplied by $\sqrt{\gamma_n}$, $N^+_{\gamma_n}(G_n,t)$ is at $L_{GHP}$-distance at most $\delta_n$ from $(\Coal_{\delta_n}((G_n,d_n,\mu_n),t))_{t\geq 0}$ (under a natural coupling), so the scaling limits will be the same. We shall want to identify the limit itself as $(\Coal_0(\Glambda,t))_{t\geq 0}$, and part of our work will consist in showing that it is a nicely behaved process. In order to accomplish this task, we need to define some subsets of $\calN$.
\begin{defi}
For $p>0$, we define $\calN_p$ to be the set of elements $\nu$ of $\calN$ such that if $\nu$ is written as $\sum_{m\in I}\delta_m$ for some countable index set $I$,
$$\sum_{m\in I}\mu(m)^p<\infty\;.$$
For $\nu$ in $\calN_p$, we let $\sizes(\nu)$ to be the sequence in $\ell^p_\searrow$ of masses $\mu(m)$ listed in decreasing order and define, for $\nu$ and $\nu'$ in $\calN_p$.
$$L_{p,GHP}(\nu,\nu')=L_{GHP}(\nu,\nu')\lor \|\sizes(\nu)-\sizes(\nu')\|_p\;.$$
\end{defi}
Again we shall abuse language, saying that $\bfX=(X,d,\mu)$ is in $\calN_p$ when $\nu_{\bfX}\in\calN_p$ and write $L_{p,GHP}(\bfX,\bfX')$ for $L_{p,GHP}(\nu_\bfX,\nu_{\bfX'})$. We let the reader check that $(\calN_p,L_{p,GHP})$ is a complete separable metric  space. 

It is easy to see that if $\bfX=(X,d,\mu)$ belongs to $\calN_1$, then almost surely, for every $t\geq 0$, $\Coal_{t,\delta}(X,d,\mu)$ is in $\calN_1$. We even have the Feller property on $\calN_1$, which will be proved in section~\ref{sec:FellerN1}. A consequence of the Feller property of the multiplicative coalescent in $\ell^2$  is that if $\bfX=(X,d,\mu)$ belongs to $\calN_2$, then almost surely for every $t\geq 0$ $\sum_{m\in\comp(\Coal_{\delta}(\bfX,t))}\mu(m)^2<\infty$. However, one cannot guarantee that components stay totally bounded, and thus  that  $\Coal_{\delta}(\bfX,t)$ or even $\Coal_{0}(\bfX,t)$ belongs to $\calN_2$. One will thus have to restrict to a subclass of $\calN_2$, which will fortunately contain $\Glambda$ with probability one.
\begin{defi}
\label{defi:calS}
We define $\calS$ to be the class of m.s-m.s $\bfX=(X,d,\mu)$ in $\calN_2$ such that
\begin{equation}
\label{eq:supdiam} \forall t\geq 0,\;\supdiam(\Coal_0(\bfX_{\leq \eta},t))\xrightarrow[\eta\rightarrow 0]{\PP}0
\end{equation}
\end{defi}
It will be shown in Lemmas~\ref{lemm:calSdansN} and \ref{lemm:calSstable} that if $\bfX\in\calS$, then almost surely, for any $t\geq 0$, $\Coal_0(\bfX,t)\in\calS$. Of course, I suspect that $\calS$ has a more intrinsic definition, at least when the components are $\RR$-graphs, and that there is a convenient topology which turns it into a Polish space, however I could not prove this for the moment. Let us mention that there are elements in $\calN_2\setminus \calS$, see Remark~\ref{rem:N2moinsS}.

Let us finish this section by a description of the coalescence at the level of components. When $\bfX$ and $\calP^+$ are as in Definition~\ref{defi:coal}, one may associate to them a process of multigraphs with vertices $\comp(X)$ which we denote by $\MG(X,t)$. It is defined as follows: there is an edge in $\MG(X,t)$ between $m$ and $m'$ if there is a point $(a,b,s)$ of $\calP$ with $s\leq t$, $a\in m$ and $b\in m'$. If $x:=\sizes(\bfX)$, this multigraph is of course closely related to the multigraph $\MG(x,t)$ defined in section~\ref{subsec:multiplicative}.

When $A$ is a measurable subset of $X$, let $\bfA:=(A,d|_{A\times A},\mu|_A)$. There is an obvious coupling between $(\Coal_\delta(\bfX,t))_{t\geq 0}$ and $(\Coal_\delta(\bfA,t))_{t\geq 0}$: just take the restriction of the poisson random set $\calP$ to $A^2\times \RR$. We shall call it the \emph{obvious coupling}. We shall use several times  the following easy fact.
\begin{lemm}
  \label{lemm:obviouscouplingforest} Suppose that $A$ is a union of components of $X$. Under the obvious coupling, if  $\MG(X,t)$ is a forest, then for every $s\leq t$ and every $x,y$ in $A$, one has the following:
  \begin{center}
    if the distance between $x$ and $y$ in $(\Coal_\delta(\bfA,s))$ is finite, \\
    then it is equal to the distance between $x$ and $y$ in $(\Coal_\delta(\bfX,s))$.\end{center}
\end{lemm}

\subsection{Length spaces and $\RR$-graphs}
\label{subsec:Rgraphs}
We refer to \cite{AdBrGoMi2017} for background on the definitions and statements in this section. Let us however recall the definition of a \emph{length space} and a \emph{geodesic space} for semi-metric spaces (cf. \cite{BuragoBuragoIvanov} for general background on length spaces).

$(X,d)$ is a length space if and only if for any $x$ and $y$ in $X$, the distance between $x$ and $y$ is the infimum of the lengths of paths $\gamma$ between $x$ and $y$. The \emph{length} of a path $\gamma$ from $[a,b]$ to $X$ being defined as:
\[\sup \sum_{i=0}^{r-1}d(\gamma(t_i),\gamma(t_{i+1}))\]
where the supremum is over all $a\leq t_0<t_1\leq \ldots \leq t_r=b$. We say that a length space $(X,d)$ is \emph{geodesic} if for any $x$ and $y$, there is a path from $x$ to $y$ with length $d(x,y)$. 
\begin{defi}
\label{defi:Rgraph}  An $\RR$-tree is a totally bounded geodesic and acyclic finite metric space. An $\RR$-graph is a totally bounded geodesic finite metric space $(G,d)$ such that there exists $R>0$ such that for any $x\in G$, $(B_R(x),d|_{B_R(x)})$ is an $\RR$-tree, where $B_R(x)$ is the  ball of radius $R$ and center $x$.
  
  For a semi-metric space $(X,d)$, we shall say that it is a \emph{semi-metric $\RR$-graph} if the quotient metric space $(X/d,d)$ is an $\RR$-graph.
\end{defi}
\begin{rem}
\begin{enumerate}[(i)]
\item  The definition above differs slightly from \cite[Definition~2.2]{AdBrGoMi2017}, where an $\RR$-graph $(X,d)$ is defined as a compact geodesic metric space such that for any $x\in G$, there exists $\eps=\eps(x)>0$ such that $(B_\eps(x),d|_{B_\eps(x)})$ is an $\RR$-tree, where $B_\eps(x)$ is the  ball of radius $\eps$ and center $x$. When $(X,d)$ is compact, the two definitions agree: one direction is obvious, whereas the other follows from the arguments at the beginning of \cite[section~6.1]{AdBrGoMi2017}. One advantage of working with precompact spaces instead of compact ones is that one may avoid having to take the completion in order to recover an $\RR$-graph after fragmentation (notice that after fragmentation, the space is not complete anymore).
  \item A semi-metric space $(X,d)$ is a length space (resp. a geodesic space) if and only if its quotient metric space $(X/d,d)$ is a length space (resp. a geodesic space). Thus, a semi-metric $\RR$-graph is notably a geodesic semi-metric space and thus a length semi-metric space.
\end{enumerate}
\end{rem}

The \emph{degree} $d_G(x)$ of a point $x$ in an $\RR$-graph $(G,d)$ is the number of connected components of $B_{R}(x)\setminus\{x\}$, where $R$ is any positive real number such that $(B_R(x),d|_{B_R(x)})$ is an $\RR$-tree. A \emph{branchpoint} $x$ is a point with $deg_G(x)\geq 3$. A \emph{leaf} $x$ is a point with degree one. We denote by $\leaves(G)$ the set of leaves of $G$. An $\RR$-tree or an $\RR$-graph is said to be \emph{finite} if it is compact and has a finite number of leaves.

An $\RR$-graph $(G,d)$, and more generally a length space, is naturally equipped with a \emph{length measure}, which assigns notably its length to the image of a simple path. Formally, it is defined as the $1$-dimensional Hausdorff measure on $(G,d)$ (see \cite[sections 1.7 and 2.6]{BuragoBuragoIvanov}). We shall denote it by $\ell_G$, and when $G$ is an $\RR$-graph, it is a $\sigma$-finite diffuse measure. If $(X,d)$ is a semi-metric length space, the length-measure of $X/d$ can be naturally carried over $X$, since Borel sets on $X/d$ are in bijections with Borel sets on $X$. We shall call this measure the length-measure of $X$ and denote it by $\ell_X$. For a semi-metric space $(X,d)$, we shall say that a Borel measure $\ell$ is \emph{diffuse} if for any $x\in X$, $\ell(\{y\in X: d(x,y)=0\})=0$ (or equivalently that its image on $X/d$ is diffuse). When $X$ is a semi-metric length space, $\ell_X$ is thus a diffuse measure. 

The structure of an $\RR$-graph is explained thoroughly\footnote{Although our definition differs slightly, the proof of \cite[Proposition~6.2]{AdBrGoMi2017} can be adapted straightforwardly.} in \cite[section~6.2]{AdBrGoMi2017}. The \emph{core} of $(G,d)$, denoted by $\core(G)$ is the union of all simple arcs with both endpoints in embedded cycles of $G$. It is also the maximal compact subset of $G$ having only points of degree at least $2$ (cf. \cite[Corollary~2.5]{AdBrGoMi2017}, where one needs to replace ``closed'' by ``compact'' in our precompact setting). The core of a tree is empty, that of a unicyclic graph is a cycle. When $G$ is neither a tree nor unicyclic, there is a finite connected multigraph $\ker(G)=(k(G),e(G))$ called the \emph{kernel} of $G$ such that the core of $G$ may be obtained from $\ker(G)$ by gluing along each edge $e$ an isometric copy of the interval $[0,l(e)]$, for some $l(e)>0$. The \emph{surplus} of $G$ is defined as $0$ when $G$ is a tree, $1$ when $G$ is unicyclic, and otherwise as:
$$\surplus(G)=|e(G)|-|k(G)|+1\;,$$
which is then at least two. If $(X,d)$ is a semi-metric $\RR$-graph, we shall define its surplus as the surplus of $(X/d,d)$.

Using the existence of the core, one gets the following equivalent definition of an $\RR$-graph, where an $\RR$-graph is obtained as a ``tree with shortcuts'', to employ the expression of \cite{BhamidiBroutinSenWang2014arXiv}. A sketch of proof is given in appendix~\ref{sec:shortcut}.
\begin{lemm}
\label{lemm:shortcut}
A metric space $(X,d)$ is an $\RR$-graph if and only if there exists an $\RR$-tree $(T,d)$ and a finite set $A\subset T^2$ such that $(X,d)$ is isomorphic to the quotient metric space obtained from $\Coal_0((T,d),A)$.
\end{lemm}
Now, let us introduce a quantity that will be useful to control the diameters of components during dynamical percolation, notably because it is monotone under fragmentation (contrarily to the diameter). In a semi-metric space $(X,d)$, we say that a path is injective if its projection on $X/d$ is injective. Then, for a semi-metric length space $X$ with length measure $\ell$, define
$$\suplength(X):=\sup\{\ell(\gamma):\gamma \text{ is a rectifiable injective path in }X\}\;.$$
Notice that 
$$\suplength(X)=\sup_{m\in\comp(X)}\suplength(m)\;.$$
Now, we can define the various spaces on which we shall study fragmentation and dynamical percolation.
\begin{defi}
\label{defi:Ngraph} 
Define $\calS^{length}$ as the class of length semi-metric spaces $\bfX$ in $\calN_2$  whose length measure is $\sigma$-finite and such that:
\begin{equation}
\label{eq:suplength} \forall t\geq 0,\;\suplength(\Coal_0(\bfX_{\leq \eta},t))\xrightarrow[\eta\rightarrow 0]{\PP}0\;.
\end{equation}

Let $\calM^{graph}$ denote the set of equivalence classes of $\RR$-graphs under $d_{GHP}$. Define $\calN^{graph}$ (resp. $\calN_p^{graph}$) from $\calM^{graph}$ in the same way that $\calN$ (resp. $\calN_p$) was defined from $\calM$.

Define $\calS^{graph}$ as the class of semi-metric spaces $\bfX$ in $\calN_2^{graph}$  such that \eqref{eq:suplength} holds.

If $\bfX=(X,d,\mu)$ and $\bfX'=(X',d',\mu')$ are measured semi-metric $\RR$-graphs, define 
$$d_{GHP}^{surplus}(\bfX,\bfX'):=d_{GHP}(\bfX,\bfX')\lor |\surplus(X)-\surplus(X')|\;.$$
Finally, define $L_{GHP}^{surplus}$ and $L_{p,GHP}^{surplus}$ in the same way that $L_{GHP}$ and $L_{p,GHP}$ were defined, but replacing $d_{GHP}$ by $d_{GHP}^{surplus}$.
\end{defi}
Notice that $\calS^{graph}$ is a subclass of $\calS^{length}$. A rigorous definition of $\calM^{graph}$ as a set is given in Appendix~\ref{sec:calM}. Thanks to Lemma~\ref{lemm:shortcut}, it is clear that if $\bfX\in\calN^{graph}$ and $\calP\subset X^2$ is finite, then for any $\delta\geq 0$, $\Coal_{\delta}(\bfX,\calP)$ still belongs to $\calN^{graph}$.


Additional notation concerning $\RR$-graphs will be introduced when needed, in section~\ref{subsec:notationsgraphs}.

\subsection{Cutting, fragmentation and dynamical percolation}
\label{subsec:frag}
\begin{defi}
\label{defi:cut}
Suppose that $\bfX=(X,d,\mu)$ is an m.s-m.s  which is a length space and $\calP^-$ is a subset of $X$. Let $\calP^{-,d}$ denote
\[\{x\in X\,:\,\exists y\in \calP^-,\;d(x,y)=0\}\;.\]
Then, the \emph{cut of $\bfX$ along $\calP^-$}, denoted by $\Frag(\bfX,\calP^-)$ is the m.s-m.s $(X\setminus \calP^{-,d},d^{\Frag}_{\calP^-}, \mu|_{X\setminus \calP^{-,d}})$ where
$$d^{\Frag}_{\calP^-}(x,y):=\inf_\gamma\{\ell_X(\gamma)\}$$
and the infimum is over all paths $\gamma$ from $x$ to $y$ disjoint from $\calP^{-,d}$.
\end{defi}
\begin{rem}
\label{rem:frag}
\begin{enumerate}[(i)]
\item $\Frag(\bfX,\emptyset)$ is the same as $\bfX$ precisely because the components of  $(X,d)$ are length spaces. Furthermore, $\Frag(\bfX,\calP^-)$ is still a length space, cf. Lemma~\ref{lemm:commutFrag}.
\item Notice that if $(X,d)$ is complete, $\Frag((X,d,\mu),\calP^-)$ is generally not complete anymore, but it is at zero $L_{GHP}$-distance from its completion.
\item It is easy to see that
  \[\Frag(X/d,\calP^-/d)/d^{\Frag}_{\calP^-}=\Frag(X,\calP^-)/d^{\Frag}_{\calP^-}\]
  Thus, one may without harm identify a semi-metric space with its metric quotient before or after fragmentation (in fact the collection of the metric quotient of its components).
\item It is easy to see that if $(X,d)$ is a length semi-metric space, $\delta\geq 0$ and $\tilde\calR$ is a multiset of elements of $X^2$, then $\Coal_\delta(X,\tilde \calR)$ is still a length semi-metric space, see for instance \cite[p. 62--63]{BuragoBuragoIvanov}. Thus, cutting $\Coal_\delta(X,\tilde \calR)$ through Definition~\ref{defi:cut} is well defined.
\item One could have defined the cut a bit differently in order to keep the base set unchanged: one could have defined $\calP^{-,d}$ to be a new component of $\bfX$, defining distance on it via an intrinsic formula. In the sequel, $\calP^{-,d}$ will have $\mu$-measure zero, so these two definitions lead to measured semi-metric spaces which are at $L_{GHP}$-distance zero, and when coalescence subsequently occurs, it ignores $\calP^{-,d}$.
\end{enumerate}
\end{rem}

\begin{defi}
\label{defi:frag}
Let $\bfX=(X,d,\mu)$ be an m.s-m.s which is a length space. Let $\ell$ be a diffuse $\sigma$-finite Borel measure on $X$. Let $\calP^-$ be a Poisson random set on $X\times \RR^+$ of intensity measure $\ell\otimes\leb_{\RR^+}$. The \emphdef{fragmentation process started from $\bfX$}, denoted by $(\Frag(\bfX,t))_{t\geq 0}$, is the random process of m.s-m.s $(\Frag(\bfX,\calP_t^-))_{t\geq 0}$. When $\bfX\in\calN^{graph}$, we shall always take $\ell$ to be $\ell_X$, the length-measure on $\bfX$. 
\end{defi}
\begin{rem}
  \label{rem:frag2}
\begin{enumerate}[(i)]
\item A similar fragmentation on the CRT is considered in \cite{AldousPitman98a}.
\item Since $\ell$ is a diffuse measure, almost surely $\mu(\calP^-_t)=0$ for any $t\geq 0$. Thus we shall abuse notation and consider that $\Frag(\bfX,\calP_t^-)$ is still equipped with $\mu$, instead of $\mu|_{X\setminus \calP^-_t}$.
\item Notice that this process inherits the strong Markov property from the strong Markov property of the Poisson process, and the fact that for $A,B\subset X$, $\Frag(\bfX,A\cup B)=\Frag(\Frag(\bfX,A),B)$, cf. Lemma~\ref{lemm:commutFrag}.
\end{enumerate}
\end{rem}

Now, one wants to define dynamical percolation on measured length spaces by performing independently and simultaneously coalescence and fragmentation. One needs to be a bit careful here: when $(X,d)$ is a geodesic space, $A\subset X^2$ and $B\subset X$, even if $B\cap\{x\in X:\exists y\in X,\; (x,y)\text{ or }(y,x)\in A\}=\emptyset$, one cannot guarantee that 
$\Coal_0(\Frag(X,B),A)$ is the same as $\Frag(\Coal_0(X,A),B)$. Indeed, let $X=[0,1]$ with the usual metric, let $B=\{\frac{3}{2^n},\; n\geq 2\}$ and $A=\{(\frac{1}{2^{n+1}},\frac{1}{2^{n}}),\;n\geq 0\}$. Then, there are two components in $\Coal_0(\Frag(X,B),A)$: $\{0\}$ and $]0,1]\setminus B$, whereas there is only one component in $\Frag(\Coal_0(X,A),B)$: $[0,1]\setminus B$. However, it will be shown in Lemma~\ref{lemm:calSdansN} that if  $\bfX\in\calS^{graph}$, $\calP^+$ is as in Definition~\ref{defi:coal},  $\calP^-$  as in Definition~\ref{defi:frag} then almost surely, 
\begin{equation}
\label{eq:commutation}
\forall t\geq 0,\;\Coal_0(\Frag(\bfX,\calP_t^-),\calP_t^+)=\Frag(\Coal_0(\bfX,\calP_t^+),\calP_t^-)\;.
\end{equation}
This will rely on Lemma~\ref{lemm:commutdeter}, proved in Appendix~\ref{sec:commutcoalfrag}. Now, let us define dynamical percolation.
\begin{defi}
\label{defi:dynperc}
Let $\bfX=(X,d,\mu)$ be an m.s-m.s which is a length space. Let $\ell$ be a diffuse $\sigma$-finite Borel measure on $X$. Let $\calP^-$ be a Poisson random set on $X\times \RR^+$ of intensity measure $\ell\otimes\leb_{\RR^+}$ and $\calP^+$ be a Poisson random set on $X^2\times \RR^+$ of intensity measure $\frac{1}{2}\mu^{2}\times \leb_{\RR^+}$. The \emphdef{dynamical percolation process started from $\bfX$}, denoted by $(\CoalFrag(\bfX,t))_{t\geq 0}$, is the stochastic process $(\Coal_0(\Frag(\bfX,\calP_t^-),\calP_t^+))_{t\geq 0}$.
\end{defi}
Property~\eqref{eq:commutation} (when it holds !) shows that $(\CoalFrag(\bfX,t))_{t\geq 0}$ inherits the strong Markov Property from that of the Poisson process.

\subsection{The scaling limit of critical Erd\H{o}s-Rényi random graphs}
The scaling limit of critical Erd\H{o}s-Rényi random graphs was obtained in \cite[Theorem~24]{AdBrGolimit}, for the Gromov-Hausdorff topology, and the result is extended to Gromov-Hausdorff-Prokhorov topology in \cite[Theorem~4.1]{AdBrGoMi2017}.  Let $\overline{\calG}_{n,\lambda}$ denote the element of $\calN_2^{graph}$ obtained by replacing each edge of  $\calG(n,p(\lambda,n))$ by an isometric copy of a segment of length $n^{-1/3}$ (notably, the distance is the graph distance divided by $n^{1/3}$) and choosing as measure the counting measure on vertices divided by $n^{2/3}$.  \cite[Theorem~4.1]{AdBrGoMi2017} and \cite[Corollary~2]{Aldous97Gnp} easily imply the following.
\begin{theo}[\cite{AdBrGolimit},\cite{AdBrGoMi2017}]
\label{theo:cvGnp}
Let $\lambda\in\RR$ and $p(\lambda,n)=\frac{1}{n}+\frac{\lambda}{n^{4/3}}$. There is a random element $\Glambda$ of $\calN_2^{graph}$ such that
$$\overline{\calG}_{n,\lambda}\xrightarrow[n\rightarrow \infty]{(d)}\Glambda\;,$$
where the convergence in distribution is with respect to the $L_{2,GHP}^{surplus}$-topology.
\end{theo}
We refer to \cite{AdBrGolimit} for the precise definition of the limit $\Glambda$ and to \cite{AdBrGoprop} for various properties of $\Glambda$.

\section{Main results}
\label{sec:mainresults}
We shall distinguish between two types of results: general ones, and those emphasized in the introduction, which are applications (usually not trivial ones) of the general results to Erd\H{o}s-Rényi random graphs.

\subsection{General results: Feller and almost Feller properties}

In the course of proving the results for Erd\H{o}s-Rényi random graphs, I tried to obtain more general results, such that one could apply the same technology to other sequences of random graphs, for instance those belonging to the basin of attraction of $\Glambda$ (see \cite{BhamidiBroutinSenWang2014arXiv} for this notion and section~\ref{sec:perspectives} in the present paper for a more detailed discussion). This is reflected in what I call below Feller or almost Feller properties for coalescence, fragmentation, and dynamical percolation. Recall that one says that a Markov process has \emph{the Feller property}\footnote{Not to be confused with being a Feller process !} if the distribution at a fixed time $t>0$ is a continuous function of the distribution at time $0$, where continuity is with respect to weak convergence of probability measures. The almost Feller properties below are variations on the Feller property, some of them weaker than a true Feller property in the sense that I need to add some condition in order to ensure convergence, but also a bit stronger in the sense that I added to the results the convergence of the whole process in the sense of the Skorokhod topology.

\begin{theo}[Almost Feller property for coalescence]
\label{theo:AlmostFellercoal}
Let $\bfX^{n}=(X^{n},d^{n},\mu^{n})$, $n\geq 0$ be a sequence of random variables in $\calS$ and $(\delta^{n})_{n\geq 0}$ a sequence of non-negative real numbers. Suppose that: 
\begin{enumerate}[(a)]
\item $(\bfX^{n})$ converges in distribution (for $L_{2,GHP}$) to $\bfX^{\infty}=(X^{\infty},d^{\infty},\mu^{\infty})$ as $n$ goes to infinity
\item $\delta^n\xrightarrow[n\rightarrow\infty]{}0$
\item For any $\alpha>0$ and any $T>0$,
  \begin{equation}
    \label{eq:supdiamunif}
    \limsup_{n\in\NN} \PP(\supdiam(\Coal_{\delta^{n}}(\bfX^{n}_{\leq \eps},T))>\alpha)\xrightarrow[\eps\rightarrow 0]{} 0
  \end{equation}
\end{enumerate}
Then, 
\begin{enumerate}[(i)]
  \item  $(\Coal_{0}(\bfX^{\infty},t))_{t\geq 0}$ is strong Markov with càdlàg trajectories in $\calS$,
  \item $(\Coal_{\delta^{n}}(\bfX^{n},t))_{t\geq 0}$ converges in distribution to $(\Coal_{0}(\bfX^{\infty},t))_{t\geq 0}$ (for $L_{2,GHP}$), 
  \item if $t^n\xrightarrow[n\rightarrow\infty]{}t$, $\Coal_{\delta^{n}}(\bfX^{n},t^n)$ converges in distribution to $\Coal_{0}(\bfX^{\infty},t)$ (for $L_{2,GHP}$).
\end{enumerate}
\end{theo}
Let us make two comments. The first is that there is a full Feller property on $\calN_1$, cf. Proposition~\ref{prop:FellerN1}. The second is that in the case of Erd\H{o}s-Rényi random graphs, condition \eqref{eq:supdiamunif} will be handled through a general technical lemma, Lemma~\ref{lem:supdiamunif}.

\begin{theo}[Feller property for fragmentation]
\label{theo:FellerGraphs}
Let $(\bfG^{n})_{n\geq 0}$ be a sequence of random variables in $\calN_2^{graph}$ converging in distribution to $\bfG$ in the $L_{2,GHP}^{surplus}$ metric. Then,
\begin{enumerate}[(i)]
\item  $(\Frag(\bfG,t))_{t\geq 0}$ is strong Markov with càdlàg trajectories (for $L_{2,GHP}^{surplus}$) in $\calN_2^{graph}$,
\item $(\Frag(\bfG^{n},t))_{t\geq 0}$ converges in distribution to $(\Frag(\bfG,t))_{t\geq 0}$ (for $L_{2,GHP}^{surplus}$),
\item if $t^n\xrightarrow[n\rightarrow\infty]{}t$, then $\Frag(\bfG^{n},t^n)$ converges in distribution to $\Frag(\bfG,t)$ (for $L_{2,GHP}^{surplus}$).
\end{enumerate}
\end{theo}

\begin{theo}[Almost Feller property for dynamical percolation]
\label{theo:almostFellerCoalFrag}
Let $(\bfX^{n})_{n\geq 0}$ be a sequence of random variables in $\calS^{graph}$ converging in distribution to $\bfX^{(\infty)}$ in the $L_{2,GHP}^{surplus}$ metric. Suppose also that for any $\alpha>0$ and any $T\geq 0$,
\begin{equation}
\label{eq:suplengthunif}
\lim_{\eps\rightarrow 0}\limsup_{n\rightarrow +\infty}\PP(\suplength(\Coal_0(\bfX^{n}_{\leq \eps},T))>\alpha)=0\;.
\end{equation}
Then,
\begin{enumerate}[(i)]
\item  $(\CoalFrag(\bfX^{(\infty)},t))_{t\geq 0}$ is strong Markov with càdlàg trajectories (for $L_{2,GHP}$) in $\calS^{length}$.
\item $(\CoalFrag(\bfX^{n},t))_{t\geq 0}$ converges in distribution to $(\CoalFrag(\bfX^{(\infty)},t))_{t\geq 0}$ (for $L_{2,GHP}$),
\item if $t^n\xrightarrow[n\rightarrow\infty]{}t$, then $\CoalFrag(\bfX^{n},t^n)$ converges in distribution to $\CoalFrag(\bfX^{(\infty)},t)$ for $L_{2,GHP}$.
\end{enumerate}
\end{theo}
A caveat is in order here: if $\bfX$ belongs to $\calS^{graph}$, $\Coal_0(\bfX,t)$ does not necessarily belong to $\calS^{graph}$  (but it does belong to $\calS^{length}$) since a component with an infinite surplus might form. In order to stay with components which are real graphs, one needs, and it is sufficient, the total mass of the components with positive surplus to be finite, cf. \cite{BhamidiBudhirajaWang2014}. See section~\ref{sec:perspectives} for more details. Consequently, $\CoalFrag$ does not define a Markov \emph{semigroup} on $\calS^{graph}$. 

In the same vein, notice that in Theorem~\ref{theo:almostFellerCoalFrag}, the initial convergence is in $L_{2,GHP}^{surplus}$ and the conclusion is in $L_{2,GHP}$. This is unavoidable, for the same reason as above: convergence in $L_{2,GHP}^{surplus}$ does not prevent the sequence $\bfX^{n}$ of having an infinite number of components with positive surplus whose total mass diverges when $n$ goes to infinity. These components can at positive time be glued to large components, augmenting their surplus indefinitely. See section~\ref{sec:perspectives} for more details. 

Finally, we shall prove a structural lemma for the multiplicative coalescent, Lemma~\ref{lemm:structureAldous}. It is at the base of Theorems~\ref{theo:AlmostFellercoal} and~\ref{theo:almostFellerCoalFrag}. Since its statement is too technical, we do not state it here, but I see it as one of the main results of the article.

\subsection{Main results for Erd\H{o}s-Rényi random graphs}

The main results for Erd\H{o}s-Rényi random graphs are the following.
\begin{theo}
  \label{theo:existencecoalfrag}
  $(\Coal_{0}(\Glambda,t))_{t\geq 0}$, $(\Frag(\Glambda,t))_{t\geq 0}$  and $(\CoalFrag(\Glambda,t))_{t\geq 0}$ are strong Markov processes with càdlàg trajectories (for $L_{2,GHP}$) in $\calS^{graph}$.
\end{theo}

\begin{theo}
\label{theo:cvcoalGnp}
Let $(\bfG^{n,\lambda,+}(t),\;t\geq 0)$ be the discrete coalescence process of intensity $n^{-4/3}$, started at $\calG(n,p(\lambda,n))$, equipped with the graph distance multiplied by $n^{-1/3}$ and the counting measure on vertices multiplied by $n^{-2/3}$. Then, the sequence of processes $(\bfG^{n,\lambda,+}(t),\;t\geq 0)$ converges to $(\Coal_0(\Glambda,t),\; t\geq 0)$ for $L_{2,GHP}$ as $n$ goes to infinity.
\end{theo}
\begin{theo}
\label{theo:cvfragGnp}
Let $(\bfG^{n,\lambda,-}(t),\;t\geq 0)$ be the discrete fragmentation process of intensity $n^{-1/3}$ started at $\calG(n,p(\lambda,n))$, equipped with the graph distance multiplied by $n^{-1/3}$ and the counting measure on vertices multiplied by $n^{-2/3}$. Then, the sequence of processes $(\bfG^{n,\lambda,-}(t),\;t\geq 0)$ converges to $(\Frag(\Glambda,t),\;t\geq 0)$ for $L_{2,GHP}^{surplus}$ as $n$ goes to infinity.
\end{theo}

\begin{theo}
\label{theo:cvpercodynGnp}
Let $(\bfG^{n,\lambda}(t),\;t\geq 0)$ be the dynamical percolation processes of parameter $p(\lambda,n)$ and intensity $n^{-1/3}$ started with $\calG(n,p(\lambda,n))$, equipped with the graph distance multiplied by $n^{-1/3}$ and the counting measure on vertices multiplied by $n^{-2/3}$. Then, the sequence of processes $(\bfG^{n,\lambda}(t),\;t\geq 0)$ converges to $(\CoalFrag(\Glambda,t),\;t\geq 0)$ for $L_{2,GHP}$ as $n$ goes to infinity.
\end{theo}
The rate $n^{-4/3}$ for discrete coalescence compared to $n^{-1/3}$ for discrete fragmentation and dynamical percolation process might be disturbing. This is due to the fact that when one performs dynamical percolation, one resamples the state of the edges. Thus, a specific edge appears at rate $n^{-1/3}p$, which is essentially $n^{-4/3}$, and disappears at rate $n^{-1/3}(1-p)$, which is essentially $n^{-1/3}$.

The last result we shall mention is the important fact that on $(\calG_\lambda)_{\lambda\in\RR}$, fragmentation is the time-reversal of coalescence:
\begin{prop}
\label{prop:retournement}
For any $\lambda\in\RR$ and $s\in\RR^+$, $(\calG_\lambda,\Coal_0(\calG_\lambda,s))$ and $(\Frag(\calG_{\lambda+s},s),\calG_{\lambda+s}))$ have the same distribution.
\end{prop}

\subsection{Description of the rest of the article}

Section~\ref{sec:preliminaries} contains some preliminary results: lemmas from \cite{Aldous97Gnp} and technical results about the distance $L_{GHP}$.

Section~\ref{sec:coal} is devoted to the proofs of the main results for coalescence: Theorem~\ref{theo:AlmostFellercoal} (the almost Feller property) and Theorem~\ref{theo:cvcoalGnp}. Probably the most important work lies inside Lemma~\ref{lemm:structureAldous}, which is a statement about the structure of the (multi)graph $W(x,t)$ in Aldous' multiplicative coalescent. It allows notably to reduce the proof of the Feller property from $\calN_2$ to $\calN_1$, where it is much easier to prove.

Section~\ref{sec:frag} is devoted to the main results for fragmentation: Theorem~\ref{theo:FellerGraphs} (the Feller property) and Theorem~\ref{theo:cvfragGnp}. We shall also show that for $\Glambda$, coalescence is the time-reversal of fragmentation, which is Proposition~\ref{prop:retournement}.

Section~\ref{sec:dynperc} is devoted to the proofs of the main results for dynamical percolation, Theorem~\ref{theo:almostFellerCoalFrag} (the almost Feller property) and Theorem~\ref{theo:cvpercodynGnp}. We also prove Theorem~\ref{theo:existencecoalfrag} there, i.e the fact that the coalescence, fragmentation and dynamical percolation on $\Glambda$ define processes in $\calS^{graph}$.

We finish the article by some perspectives in section~\ref{sec:perspectives}, and some technical tools are gathered in the appendix.

\section{Preliminary results and tools}
\label{sec:preliminaries}
\subsection{The multiplicative coalescent}
\label{subsec:multiplicative}
The main tool to analyze our coalescent and fragmentation processes will be a refinement of Aldous' work~\cite{Aldous97Gnp} on the multiplicative coalescent. In this section, we recall what we will use of his work.

Let $(N_{i,j})_{i,j\in \NN^*}$ be independent Poisson point processes on the real line with intensity $1$. Denote by $T_{i,j,n}$ the $n$-th jump-time of $N_{i,j}$. For $x\in\ell^2_+$, let $\MG(x,t)$ denote the weighted multigraph (with loops) with vertex set $\NN$, edge set $\cup_{n\in\NN}\{\{i,j\}\in\binom{\NN}{2}\mbox{ s.t. }T_{i,j,n}\mbox{ or }T_{j,i,n}\leq \frac{t}{2}x_ix_j\}$ and weight $x_i$ on vertex $i$. If one forgets loops and transforms any multiple edge into a single edge, $\MG(x,t)$ becomes $\calW(x,t)$, the nonuniform random graph of~\cite[section~1.4]{Aldous97Gnp}. Indeed, the graph $\calW(x,t)$ has set of vertices $\NN^*$ and each pair $\{i,j\}$ is an edge with probability $1-\exp(-tx_ix_j)$, independently for distinct pairs. The size, or mass, of a connected component of those (multi-)graphs is defined to be the sum of weights $x_i$ for $i$ in this component. Denoting by $X(x,t)$  the sequence of sizes, listed in decreasing order, of the connected components of $\calW(x,t)$, Aldous proved in \cite[Proposition~5]{Aldous97Gnp}, that $\{X(x,t)\,:\,t\geq 0\}$ defines a Markov process on $\ell^2_\searrow$ which possesses the Feller property. In words, the dynamics is as follows: two distinct connected components $c$ and $c'$ with sizes $m$ and $m'$ merge at a rate proportional to $m m'$, the product of their sizes, whence the term {\it multiplicative coalescent}.

Following Aldous, we denote by $S(x,t)$ the sum of squares of the sizes of the components of $\MG(x,t)$. We shall use later the following lemmas.
\begin{lemm}[\cite{Aldous97Gnp}, Lemma 20]
\label{lemm:Aldous20}
For $x$ in $l^2_\searrow$, 
$$\PP(S(x,t)>s)\leq \frac{tsS(x,0)}{s-S(x,0)},\quad s>S(x,0)\;.$$
\end{lemm}

\begin{lemm}[\cite{Aldous97Gnp}, Lemma 23]
\label{lemm:Aldous23}
Let $(z_i,\;1\leq i\leq n)$ be strictly positive vertex weights, and let $1\leq m<n$. Consider the bipartite random graph $\calB$ on vertices $\{1,2,\ldots,m\}\cup\{m+1,\ldots,n\}$ defined by: for each pair $(i,j)$ with $1\leq i\leq m<j\leq n$, the edge $\{i,j\}$ is present with probability $1-\exp(-tz_iz_j)$, independently for different pairs. Write $\alpha_1=\sum_{i=1}^mz_i^2$, $\alpha_2=\sum_{i=m+1}^nz_i^2$. Let $(Z_i)$ be the sizes of the components of $\calB$. Then,
$$\eps\PP(\sum_iZ_i^2 \geq \alpha_1+\eps)\leq (1+t(\alpha_1+\eps))^2\alpha_2,\quad \eps>0\;.$$
\end{lemm}
\begin{rem} As noticed in \cite[p. 842]{Aldous97Gnp}, Lemma~\ref{lemm:Aldous23} extends to $z\in\ell^2$. Notice that in \cite[Lemma 23]{Aldous97Gnp}, the upper bound was stated as $(2t(\alpha_1+\eps)+(t(\alpha_1+\eps))^2)\alpha_2$, but this cannot be true, as can be seen by taking $t=0$. In fact a term $z_{m+1}^2$ is missing in the right-hand side of line $-3$ in \cite[p. 841]{Aldous97Gnp}, which impacts all subsequent inequalities. Correcting this slight mistake leads to the bound above. 
\end{rem}

In \cite{Aldous97Gnp}, this lemma is used in conjunction with the following one
\begin{lemm}[\cite{Aldous97Gnp}, Lemma 17]
\label{lemm:Aldous17}
Let  $\tilde G$ be a graph with vertex weights $(\tilde x_i)$. Let $G$ be a subgraph of $\tilde G$ (that is, they have the same vertex-sets and each edge of $G$ is an edge of $\tilde G$) with vertex weights $x_i\leq \tilde x_i$. Let $\tilde a$ and $a$ be the decreasing orderings of the component sizes of $\tilde G$ and $G$. Then
$$\|\tilde a-a\|_2\leq \sum_{i}\tilde a_i^2-\sum_{i}a_i^2$$
provided $\sum_{i}a_i^2< \infty$.
\end{lemm}

Finally, the following lemma will be useful to prove convergence in the sense of Skorokhod: it will allow us to control the $\ell^2$ distance between $X(x,t')$ and $X(x_{\leq \eps},t')$ for $t'\leq t$ by the $\ell^2$-distance at time $t$. 
\begin{lemm}
\label{lemm:pourSkorL2}
Let $G=(V,E)$ be a multigraph whose vertices have weights $(x_i)_{i\in V}$. If $W\subset V$ and $E'\subset E$, let $\comp(W,E')$ denote the set of connected components of the graph $(W,E'\cap \binom{W}{2})$ and define
$$S(W,E'):=\sum_{m\in\comp(W,E')}\left(\sum_{i\in m}x_i\right)^2\;.$$
Now, let $W\subset V$ be such that no point of $V\setminus W$ belongs to a cycle in $(V,E)$. Then, for any $E'\subset E$
$$S(V,E)-S(W,E)\geq S(V,E')-S(W,E')\;,$$
provided $S(W,E)<\infty$.
\end{lemm}
\begin{proof}
For $i$ and $j$ in $V$ and $E'\subset E$, we denote by ``$i\sim j\in E'$'' the fact that $i$ and $j$ are distinct and connected by a path in $(V,E')$ and by ``$i\sim j\not\in E'$'' the negation of the previous statement, i.e that $i=j$ or $i$ and $j$ are not connected by a path in $(V,E')$. The hypothesis on $W$ implies that for any $E'\subset E$, 
\begin{equation}
\label{eq:E'interne}
\forall i,j\in W,\;\left(\begin{array}{c}i \sim j\in E' \\ \text{and} \\ i\sim j\not \in E' \cap \binom{W}{2}\end{array}\right)\Rightarrow \left(\begin{array}{c}i \sim j\in E \\ \text{and} \\ i\sim j\not \in E \cap \binom{W}{2}\end{array}\right)\;.
\end{equation}
Indeed, suppose that $i$ and $j$ satisfy the left hand side of the implication above. Since $E'\subset E$, we have $i\sim j \in E$. From the hypothesis on $i$ and $j$, there is a simple path $\gamma$ from $i$ to $j$ in $E'$ with a point $z\in \gamma \cap (V\setminus W)$. Denote by $i'$ the point on $\gamma$ just before $z$ and by $j'$ the point on $\gamma$ just after $z$. Now, suppose that we had $i\sim j \in E \cap \binom{W}{2}$. We would have a path $\gamma'$ in $(W,E\cap \binom{W}{2})$ from $i$ to $j$. Concatenating the portion of $\gamma$ from $i$ to $i'$, the portion of $\gamma$ from $j$ to $j'$ and the path $\gamma'$, we see that there is a path from $i'$ to $j'$ in $(V,E)$ which avoids $z$. From this path, we can extract a simple path in $(V,E)$ connecting $i'$ to $j'$, still avoiding $z$. Now, $z$ is a neighbour of $i'$ and $j'$ in $(V,E)$, and this gives a cycle going through $z$. Since $z$ is a point of $V\setminus W$, this contradicts the hypothesis on $W$. Thus, $i\sim j\not \in E \cap \binom{W}{2}$ holds and \eqref{eq:E'interne} is true.

Now,
\begin{eqnarray*}
  &&S(V,E')= \sum_{i\in V}x_i^2+\sum_{\substack{i,j\in V\\ i\sim  j\in E'}}x_ix_j\\
&=& S(W,E')+\sum_{i\in V\setminus W}x_i^2+\sum_{\substack{i,j\in V\\ i\sim j\in E'}}x_ix_j-\sum_{\substack{i,j\in W\\ i\sim j\in E'\cap \binom{W}{2}}}x_ix_j\\
&=& S(W,E')+\sum_{i\in V\setminus W}x_i^2+\sum_{\substack{i,j\in V\\ i\sim j\in E'\text{ and }i\sim j\not \in E'\cap \binom{W}{2}}}x_ix_j\\
\end{eqnarray*}
Now, \eqref{eq:E'interne} shows that the last sum on the right of the last equation is increasing in $E'\subset E$, and this shows the result.
\end{proof}

\subsection{Tools to handle $L_{GHP}$}

The following proposition is analogous to similar results concerning vague convergence of locally finite measures (see \cite[section~4]{KallenbergRandomMeasures}). With this proposition at hand, it is easy to show, for instance, that if $\nu$ is made of a countable number of components with strictly distinct positive masses, then $\nu_n$ converges to $\nu$ is equivalent to the fact that for every $k$, the $k$-th largest component of $\nu_n$ converges, for $d_{GHP}$, to the $k$-th largest component of $\nu$.  
\begin{prop}
\label{prop:Npolish}
$(\calN,L_{GHP})$ is a complete separable metric space, and if $\nu_n$, $n\geq 0$ and $\nu$ are elements of $\calN$, $(\nu_n)_{n\geq 0}$ converges to $\nu$ if and only if for every $\eps>0$ such that $\nu(\{[(X,d,\mu)]\in \calM\;:\;\mu(X)=\eps\})=0$, $\delta_{LP}(\nu_{>\eps}^{n},\nu_{>\eps})$ goes to zero as $n$ goes to infinity.
\end{prop}
\begin{proof}
The fact that $L_{GHP}$ is a metric is left to the reader. Let $D$ be a countable set dense in $\{[(X,d,\mu)]\in \calM : \mu(X)\not=0\}$. Let $\calD:=\{\sum_{k=1}^m\delta_{\bfX_{k}}:m\in\NN,\;\bfX_1,\ldots,\bfX_m\in D\}$. It is easy to show that $\calD$ is dense in $(\calN,L_{GHP})$. This shows separability.

Now, suppose that $\nu^{n}$ is a Cauchy sequence for $L_{GHP}$. Then, for any $k\geq 1$, $f_k\nu^{n}$ is a Cauchy sequence of finite measures for $\delta_{LP}$. From the completeness of the Lévy-Prokhorov distance on finite measures on a Polish space we get that for each $k$, there is some measure $\nu_k$ such that
$$\delta_{LP}(f_k\nu^{n},\nu_k)\xrightarrow[n\rightarrow \infty]{}0\;.$$
Notice that if $2\leq k\leq l$, $\nu_k=\nu_l$ on $\calM_{>\frac{1}{k-1}}$. Define
$$\nu=\sup_{k\geq 0}\II_{\calM_{>\frac{1}{k+1}}}\nu_{k+2}$$
so that for any $k\geq 1$, $f_k\nu =\nu_k$. Then, $\nu$ is an element of $\calN$ and 
$$L_{GHP}(\nu^{n},\nu)\xrightarrow[n\rightarrow \infty]{}0\;,$$
showing the completeness of $(\calN,L_{GHP})$.

Finally, suppose that $L_{GHP}(\nu^{n},\nu)$ goes to zero as $n$ goes to infinity and let $\eps>0$ be such that $\nu(\{[(X,d,\mu)]\in \calM\;:\;\mu(X)=\eps\})=0$. Then, for any $\alpha>0$, let $k$ be such that
$$\frac{1}{k}\leq \eps$$
let $N$ be such that 
$$\forall n\geq N,\;\forall A\in\calB(\calM),\;f_k\nu(A)\leq f_k\nu^{n}(A^\alpha)+\alpha\text{ and }f_k\nu^{n}(A)\leq f_k\nu(A^\alpha)+\alpha\;. $$
Then, for $n\geq N$ and $B\in\calB(\calM_{>\eps})$,
\begin{eqnarray*}
\nu_{>\eps}(B)&\leq &\nu(B\cap \calM_{>\eps+\alpha})+\nu(\calM_{>\eps}\setminus\calM_{>\eps+\alpha})\\
&=&f_k\nu(B\cap \calM_{>\eps+\alpha})+\nu(\calM_{>\eps}\setminus\calM_{>\eps+\alpha})\\
&\leq &f_k\nu^{n}((B\cap\calM_{>\eps+\alpha})^\alpha)+\alpha+\nu(\calM_{>\eps}\setminus\calM_{>\eps+\alpha})\\
&= &\nu^{n}_{>\eps}((B\cap\calM_{>\eps+\alpha})^\alpha)+\alpha+\nu(\calM_{>\eps}\setminus\calM_{>\eps+\alpha})\\
&\leq & \nu^{n}_{>\eps}(B^\alpha)+\alpha+\nu(\calM_{>\eps}\setminus\calM_{>\eps+\alpha})
\end{eqnarray*}
where we used the fact that $(B\cap\calM_{>\eps+\alpha})^\alpha\subset \calM_{>\eps}$ and $f_k$ equals $1$ on $\calM_{>\eps}$. Also, for $n\geq N$ and $B\in\calB(\calM_{>\eps})$,
\begin{eqnarray*}
\nu^{n}_{>\eps}(B)& = &f_k\nu^{n}(B)\\
&\leq & f_k\nu(B^\alpha)+\alpha\\
&\leq & \nu(B^\alpha)+\alpha\\
&\leq & \nu(B^\alpha\cap\calM_{>\eps})+\nu(\calM_{>\eps-\alpha}\setminus\calM_{>\eps})+\alpha\\
&=&\nu_{>\eps}(B^\alpha)+\nu(\calM_{>\eps-\alpha}\setminus\calM_{>\eps})+\alpha\\
\end{eqnarray*}
To finish the proof, note that since $\nu(\{[(X,d,\mu)]\in \calM : \mu(X)=\eps\})=0$, then 
$$\nu(\calM_{>\eps}\setminus\calM_{>\eps+\alpha})+\nu(\calM_{>\eps-\alpha}\setminus\calM_{>\eps})\xrightarrow[\alpha\rightarrow 0]{}0\;.$$
\end{proof}

We shall always use the following lemmas to bound $L_{GHP}$ from above.
\begin{lemm}
\label{lemm:Ldeuxappli}
Let $\bfX=(X,d,\mu)$ and $\bfX'=(X',d',\mu')$ belong to $\calN$ and fix $\eps\in]0,\frac{1}{2}]$. Suppose there exists two injective maps
$$\sigma:\comp(\bfX_{>\eps})\rightarrow \comp(\bfX')\mbox{ and }\sigma':\comp(\bfX'_{>\eps})\rightarrow \comp(\bfX)\;$$
such that:
$$\forall m\in \comp(\bfX_{>\eps}),\; d_{GHP}(m,\sigma(m))\leq \alpha$$
and
$$\forall m'\in \comp(\bfX'_{>\eps}),\;d_{GHP}(m',\sigma'(m'))\leq \alpha\;.$$
Then,
$$L_{GHP}(\nu_{\bfX},\nu_{\bfX'})\leq 8\alpha\#\comp(\bfX_{>\eps-\alpha})+16\eps\;,$$
and if $\eps>\alpha$, for $p>0$,
$$L_{GHP}(\nu_{\bfX},\nu_{\bfX'})\leq 8\alpha\frac{\sum_{m\in\comp(\bfX)}\mu(m)^p}{(\eps-\alpha)^p}+16\eps\;,$$
\end{lemm}
\begin{proof}
Consider any $\eps_0\geq\eps$. For a component $m$ in $\comp(\bfX_{>\eps_0})$, the difference between the masses $\mu(m)$ and $\mu'(\sigma(m))$ is at most $\alpha$, and the same holds between $m'$ and $\sigma'(m')$ when $m'\in\comp(\bfX'_{>\eps_0})$. Thus $\sigma'$ sends $\comp(\bfX'_{>\eps_0})$  into $\comp(\bfX_{>\eps_0-\alpha})$ and 
$$\#\{m\in \comp(\bfX'_{>\eps_0})\}\leq \#\{m\in \comp(\bfX_{>\eps_0-\alpha})\}\;.$$
Now, let $k\geq 1$ be such that
$$\frac{1}{k+1}\geq \eps\;.$$
Let $B\in\calB(\calM_{>\frac{1}{k+1}})$. Then, for any $m$ in $\comp(\bfX_{>\frac{1}{k+1}})\cap B$, $\sigma(m)$ belongs to $\comp(\bfX')\cap B^\alpha$. Then, notice that $f_k$ is $k(k+1)$-Lipschitz.
\begin{eqnarray*}
f_k\nu_{\bfX}(B)&=&\sum_{m\in\comp(\bfX_{>\frac{1}{k+1}})\cap B}f_k(m)\\
&\leq & \sum_{m\in\comp(\bfX_{>\frac{1}{k+1}})\cap B}f_k(\sigma(m))+\alpha k(k+1)\#\comp(\bfX_{>\frac{1}{k+1}})\\
&\leq & \sum_{\substack{m'\in\comp(\bfX')\\ m'\in B^\alpha}}f_k(m')+\alpha k(k+1)\#\comp(\bfX_{>\frac{1}{k+1}})\\
&=&f_k\nu_{\bfX'}(B^{\alpha})+\alpha k(k+1)\#\comp(\bfX_{>\frac{1}{k+1}})
\end{eqnarray*}
and symmetrically
$$f_k\nu_{\bfX'}(B)\leq f_k\nu_{\bfX}(B^{\alpha})+\alpha k(k+1)\#\comp(\bfX'_{>\frac{1}{k+1}})$$
Thus, for any $k\geq 1$ such that $\frac{1}{k+1}\geq \eps$,
\begin{eqnarray*}
\delta_{LP}(f_k\nu_{\bfX},f_k\nu_{\bfX'})&\leq &\alpha k(k+1)\#\comp(\bfX_{>\frac{1}{k+1}})\lor\#\comp(\bfX'_{>\frac{1}{k+1}})\\
&\leq &\alpha k(k+1)\#\comp(\bfX_{>\eps-\alpha})
\end{eqnarray*}
Thus,
\begin{eqnarray*}
  &&L_{GHP}(\nu_{\bfX},\nu_{\bfX'})\\
  &\leq & \alpha\#\comp(\bfX_{>\eps-\alpha})\sum_{k< \frac{1}{\eps}-1}2^{-k}k(k+1)+\sum_{k\geq \frac{1}{\eps}-1}2^{-k}\\
&\leq & 8\alpha\#\comp(\bfX_{>\eps-\alpha})+16\eps
\end{eqnarray*}

\end{proof}

If $\bfX$ and $\bfX'$ are two m.s-m.s with a finite number of finite components, one may measure their distance with $d_{GHP}$ (using Definition~\ref{defi:GHPMiermont}), with $L_{GHP}$ (using Definition~\ref{defi:dGHPcountableproduct}) or with 
$$1\land\inf_{\sigma}\sup_{m\in\comp(\bfX)}d_{GHP}(m,\sigma(m))=1\land \delta_{LP}(\nu_{\bfX},\nu_{\bfX'})$$
where the infimum is over bijections $\sigma$ between $\comp(\bfX)$ and $\comp(\bfX')$. Those three distances do not necessarily coincide, and the following lemma clarifies the links between them.
\begin{lemm}
\label{lemm:3GHP}
Let $\bfX=(X,d,\mu)$ and $\bfX'=(X',d',\mu')$ be two m.s-m.s in $\calN$ with a finite number of components. 
\begin{enumerate}[(i)]
\item If $d_{GHP}(\bfX,\bfX')<\infty$, then there is a bijection $\sigma$ from $\comp(\bfX)$ to $\comp(\bfX')$ such that:
$$\forall m\in\comp(\bfX),d_{GHP}(m,\sigma(m))\leq 2d_{GHP}(\bfX,\bfX')\;,$$
and thus,
$$L_{GHP}(\bfX,\bfX')\leq 16d_{GHP}(\bfX,\bfX')\#\comp(\bfX)\;.$$
\item If there exists a bijection $\sigma$ from $\comp(\bfX)$ to $\comp(\bfX')$ such that:
$$\sup_{m\in \comp(\bfX)}d_{GHP}(m,\sigma(m))<\infty\;,$$
then,
$$d_{GHP}(\bfX,\bfX')\leq \sup_{m\in \comp(\bfX)}d_{GHP}(m,\sigma(m))\#\comp(\bfX)\;,$$
and 
$$L_{GHP}(\bfX,\bfX')\leq 8 \sup_{m\in \comp(\bfX)}d_{GHP}(m,\sigma(m)) \#\comp(\bfX)\;.$$
\end{enumerate}
\end{lemm}
\begin{proof}
{\it Proof of (i).} Suppose that $d_{GHP}(\bfX,\bfX')<\eps<\infty$. Let $\calR\in C(X,X')$ and $\pi\in M(X,X')$ be such that
$$D(\pi;\mu,\mu')\lor \frac{1}{2}\dis(\calR)\lor \pi(\calR^c)\leq \eps\;.$$
Since $\calR$ has finite distortion,
  $$\forall (x,x'),(y,y')\in\calR, d(x,y)=+\infty\Leftrightarrow d'(x',y')=+\infty$$
which shows that each component $m$ of $\bfX$ (resp. $\bfX'$) is in correspondence through $\calR$ with exactly one component $\sigma(m)$ of $\bfX'$ (resp. $\bfX$). $\sigma$ is thus a bijection, and $\calR\cap m\times \sigma(m)\in C(m,\sigma(m))$ and has distortion at most $2\eps$. Furthermore, 
$$\pi|_{m\times\sigma(m)}((\calR\cap m\times \sigma(m))^c)=\pi(\calR^c\cap m\times \sigma(m))\leq \eps\;.$$
Finally, for any $A\in\calB(m)$,
\begin{eqnarray*}
|\pi|_{m\times\sigma(m)}(A\times \sigma(m))-\mu|_m(A)|&\leq & |\pi(A\times X')-\mu(A)|
  +\pi(A\times \sigma(m)^c)\\
&\leq & \eps +\pi(\calR^c)\\
&\leq & 2\eps
\end{eqnarray*}
and similarly, for any $A'\in\calB(\sigma(m))$, 
$$|\pi|_{m\times\sigma(m)}(m\times A')-\mu'|_{\sigma(m)}(A')|\leq 2\eps\;.$$
Thus, 
$$\forall m\in\comp(\bfX),d_{GHP}(m,\sigma(m))\leq 2\eps\;,$$
and the consequence on $L_{GHP}(\bfX,\bfX')$ comes from Lemma~\ref{lemm:Ldeuxappli} applied for any $\eps>0$ small enough, and letting $\eps$ go to zero. 

{\it Proof of (ii).} Suppose that there exists a bijection $\sigma$ from $\comp(\bfX)$ to $\comp(\bfX')$ such that:
$$\sup_{m\in \comp(\bfX)}d_{GHP}(m,\sigma(m))\leq \eps <\infty\;.$$
Then, for any $m$, let $\calR_m\in C(m,\sigma(m))$ and $\pi_m\in M(m,\sigma(m))$ be such that
$$D(\pi_m;\mu|_m,\mu|_{\sigma(m)}')\lor \frac{1}{2}\dis(\calR_m)\lor \pi_m(\calR_m^c)\leq \eps\;.$$
Let $\pi=\sum_{m\in\comp(\bfX)}\pi_m$ and $\calR=\cup_{m\in\comp(\bfX)}\calR_m$. Then, $\calR$ is a correspondence between $X$ and $X'$,
$$\frac{1}{2}\dis(\calR)\leq \frac{1}{2}\sup_{m}\dis(\calR_m)\leq \eps\;,$$
$$\pi(\calR^c)=\sum_{m}\pi_m(\calR_m^c)\leq\#\comp(\bfX)\eps\;.$$
Furthermore, for any $A\in\calB(X)$,
\begin{eqnarray*}
|\pi(A\times X')-\mu(A)|&\leq & \sum_{m\in\comp(\bfX)}|\pi((A\cap m)\times X')-\mu(A\cap m)|\\
&=& \sum_{m\in\comp(\bfX)}|\pi((A\cap m)\times \sigma(m))-\mu(A\cap m)|\\
&\leq & \#\comp(\bfX)\eps
\end{eqnarray*}
and symetrically, for any $A'\in\calB(X')$,
$$|\pi(X\times A')-\mu'(A')|\leq \#\comp(\bfX)\eps$$
Thus
$$D(\pi;\mu,\mu')\leq \#\comp(\bfX)\eps\;,$$
and we get
$$d_{GHP}(\bfX,\bfX')\leq \#\comp(\bfX)\eps\;.$$
The statement on $L_{GHP}(\bfX,\bfX')$ follows from the hypothesis and Lemma~\ref{lemm:Ldeuxappli} applied for any $\eps>0$ small enough, and letting $\eps$ go to zero. 
\end{proof}
Let us end this section with two remarks. First, notice that $L_{GHP}$ makes sense even between counting measures whose atoms are semi-metric spaces. Notice also that $\delta_{LP}(\nu_{\bfX},0)$ is at most the number of connected components of $\bfX$. Thus
$$L_{GHP}(\nu_{\bfX},0)\leq 2^{2-(\sup_{m\in\comp(\bfX)}\mu(m))^{-1}}\;.$$
One sees thus that if $\bfX^{n}$ is a sequence of m.s-m.s such that
$$\sup_{m\in\comp(\bfX^{n})}\mu(m)\xrightarrow[n\rightarrow +\infty]{}0$$
and whatever the diameters of the components are, then $\nu_{\bfX^{n}}$ converges to zero for $L_{GHP}$, which can be seen as an empty collection of measured metric spaces. Notably, $\supdiam$ is not continuous with respect to the $L_{GHP}$-distance.
\section{Proofs of the main results for coalescence}
\label{sec:coal}
\subsection{The Coalescent on $\calN_1$}
\label{sec:FellerN1}

On $\calN_1$, coalescence behaves very gently since there is a finite number of coalescence events in any finite time interval. Notably, for $\bfX\in\calN_1$, $(\Coal_\delta(\bfX,t))_{t\geq 0}$ is clearly c\`adl\`ag. The aim of this section is to prove Proposition~\ref{prop:FellerN1}, which is essentially a Feller property, together with a variant, Proposition~\ref{prop:FellerN1surplus}.
\begin{lemm}
\label{lemm:AdBrGoMigluingGHP}
Let $\eps\in]0;1[$, $\bfX=(X,d,\mu)$ and $\bfX'=(X',d',\mu')$ be two m.s-m.s with a finite number of finite components.

      Suppose that $P=\{(x_i , y_i),\;1 \leq i \leq  k\}$ are pairs of points in $X$ and $P'=\{(x'_i , y'_i),\; 1 \leq i \leq  k\}$ are pairs of points in $X'$. Suppose that there exists $\pi\in M(X,X')$ and $\calR\in C(X,X')$ such that:
$$D(\pi;\mu,\mu')\lor \pi(\calR^c)\lor \frac{1}{2}\dis(\calR)\leq \eps$$
and that for any $i\leq k$, $(x_i , x'_i)\in\calR$ and $(y_i , y'_i)\in\calR$. Then, for any $\delta,\delta'>0$,
$$d_{GHP}(\Coal_{\delta}(\bfX,P),\Coal_{\delta'}(\bfX',P'))\leq (2\eps+|\delta-\delta'|)(k+1)$$
\end{lemm}
\begin{proof}
This is essentially\footnote{\cite[Lemma~4.2]{AdBrGoMi2017} is stated for trees and for $\delta=0$.} Lemma~21 in \cite{AdBrGolimit} and Lemma~4.2 in \cite{AdBrGoMi2017}, thus we leave the details to the reader.
\end{proof}

\begin{lemm}
\label{lemm:couplageN1}
Let $\eps\in]0;1[$ and $\delta>0$, $\bfX=(X,d,\mu)$ and $\bfX'=(X',d',\mu')$ be two m.s-m.s with a finite number of finite components. If there exists $\pi\in M(X,X')$ and $\calR\in C(X,X')$ such that:
$$D(\pi;\mu,\mu')\lor \pi(\calR^c)\lor \frac{1}{2}\dis(\calR)\leq \eps$$
then, one may couple two Poisson random sets: $\calP$ of intensity $\frac{1}{2}\mu^{\otimes 2}\otimes\leb_{[0,T]}$ and $\calP'$ of intensity $\frac{1}{2}(\mu')^{\otimes 2}\otimes\leb_{[0,T]}$, such that with probability larger than $1-T\eps(10+8\mu(X)+8\mu'(X'))-\sqrt{2\eps+|\delta-\delta'|}$, for any $t\leq T$,
$$d_{GHP}(\Coal_{\delta}(\bfX,\calP_t),\Coal_{\delta'}(\bfX',\calP'_t))\leq (T\mu(X)^2+1)\sqrt{2\eps+|\delta-\delta'|}\;.$$
\end{lemm}
\begin{proof}
Let $\calP(\mu)$ denote the distribution of a Poisson random set of intensity measure $\mu$.  Using the coupling characterization of total variation distance and the gluing lemma  (cf \cite[p.~23]{Villani09}), one may construct three Poisson random sets on the same probability space, $P$, $\tilde P$ and $P'$ such that:
\begin{enumerate}[(i)]
\item $P=(X_i,Y_i,t_i)_{i=1,\ldots,N}$ has distribution $\calP(\frac{1}{2}\mu^{\otimes 2}\times\leb_{[0,T]})$,
\item $P'=(X'_i,Y'_i,t'_i)_{i=1,\ldots,N'}$ has distribution $\calP(\frac{1}{2}(\mu')^{\otimes 2}\times\leb_{[0,T]})$,
\item $\tilde P=(\tilde X_i,\tilde X_i',\tilde Y_i,\tilde Y_i',\tilde t_i)_{i=1,\ldots, \tilde N}$ has distribution $\calP(\frac{1}{2}\pi^{\otimes 2}\times\leb_{[0,T]})$
\end{enumerate}
and furthermore:
\begin{eqnarray*}
&&\PP[(X_i,Y_i,t_i)_{i=1,\ldots,N}\not=(\tilde X_i,\tilde Y_i,\tilde t_i)_{i=1,\ldots, \tilde N}]\\
&\leq& \|\calP(\frac{1}{2}\mu^{\otimes 2}\times\leb_{[0,T]})- \calP(\frac{1}{2}\pi_1^{\otimes 2}\times\leb_{[0,T]})\|
\end{eqnarray*}
and

\begin{eqnarray*}
&&\PP[(X'_i,Y'_i,t'_i)_{i=1,\ldots,N'}\not=(\tilde X'_i,\tilde Y'_i,\tilde t_i)_{i=1,\ldots, \tilde N}]\\
&\leq&  \|\calP(\frac{1}{2}(\mu')^{\otimes 2}\times\leb_{[0,T]})- \calP(\frac{1}{2}\pi_2^{\otimes 2}\times\leb_{[0,T]})\|\;.
\end{eqnarray*}
Now, for any $T>0$, using Lemma~\ref{lemm:dVTPoiss} in the appendix,
\begin{eqnarray*}
&&\|\calP(\frac{1}{2}\mu^{\otimes 2}\times\leb_{[0,T]})- \calP(\frac{1}{2}\pi_1^{\otimes 2}\times\leb_{[0,T]})\|\\
&\leq & 2\|\frac{1}{2}\mu^{\otimes 2}\times\leb_{[0,T]}-\frac{1}{2}\pi_1^{\otimes 2}\times\leb_{[0,T]}\|\\
&= & T\|\mu^{\otimes 2}-\pi_1^{\otimes 2}\|\\
&\leq & 2T(\mu(X)+\pi_1(X))\|\mu-\pi_1\|\\
&\leq &2T\eps(2\mu(X)+\eps)
\end{eqnarray*}
by hypothesis. Similarly,
$$\|\calP(\frac{1}{2}(\mu')^{\otimes 2}\times\leb_{[0,T]})- \calP(\frac{1}{2}\pi_2^{\otimes 2}\times\leb_{[0,T]})\|\leq 2T\eps(2\mu'(X')+\eps)\;.$$
Furthermore, $(\tilde X_i,\tilde X'_i,\tilde t_i)_{i=1,\ldots,\tilde N}$ and $(\tilde Y_i,\tilde Y'_i,\tilde t_i)_{i=1,\ldots,\tilde N}$ both have distribution $\calP(\frac{1}{2}\pi\otimes\leb_{[0,T]})$. Thus,
$$\PP(\exists i\leq \tilde N,\;(\tilde X_i,\tilde X'_i)\not\in\calR)\leq \frac{1}{2}T\pi(\calR^c)\leq T\eps$$
and 
$$\PP(\exists i\leq \tilde N,\;(\tilde Y_i,\tilde Y'_i)\not\in\calR)\leq T\eps\;.$$
Let $\calE$ be the event that $N=N'$ and for any $i$, $(\tilde X_i,\tilde X'_i)\in\calR$ and $(\tilde Y_i,\tilde Y'_i)\in\calR$. Altogether, we get that $\calE$ has probability at least $1-T\eps(10+8\mu(X)+8\mu'(X'))$.

Since the distortion of $\calR$ is at most $2\eps$, we get using Lemma~\ref{lemm:AdBrGoMigluingGHP} that on the event $\calE$, for any $t\leq T$
$$d_{GHP}(\Coal_\delta(\bfX,\calP_t),\Coal_{\delta'}(\bfX',\calP'_t))\leq (N+1)(2\eps+|\delta-\delta'|)$$
Since $N$ has distribution $\calP(\frac{1}{2}\mu(X)^2T)$, Markov's inequality implies
$$\PP\left(N\geq \frac{T\mu(X)^2}{\sqrt{2\eps+|\delta-\delta'|}}\right)\leq \sqrt{2\eps+|\delta-\delta'|}$$
this gives the result.
\end{proof}

In the proposition below, recall from section~\ref{subsec:gennot} that convergence of processes uses the Skorokhod topology (here for the metric space $(\calN_1,L_{1,GHP})$), which we shall always prove using Lemma~\ref{lemm:compactcv}.
\begin{prop}
\label{prop:FellerN1}
Let $\bfX^{n}=(X^{n},d^{n},\mu^{n})$, $n\geq 0$ be a sequence of elements in $\calN_1$ and $(\delta^{n})_{n\geq 0}$ a sequence of non-negative real numbers. Suppose that: 
\begin{enumerate}[(a)]
\item $(\bfX^{n})_{n\geq 0}$ converges (for $L_{1,GHP}$) to $\bfX^{\infty}=(X^{\infty},d^{\infty},\mu^{\infty})$ as $n$ goes to infinity
\item $\delta^n\xrightarrow[n\rightarrow\infty]{}\delta^\infty$
\end{enumerate}
Then,
\begin{enumerate}[(i)]
\item $(\Coal_{\delta^{n}}(\bfX^{n},t))_{t\geq 0}$ converges in distribution  (for $L_{1,GHP}$) to $(\Coal_{\delta^{\infty}}(\bfX^{\infty},t))_{t\geq 0}$,
\item if $t^n\xrightarrow[n\rightarrow\infty]{}t$, $\Coal_{\delta^{n}}(\bfX^{n},t^n)$ converges in distribution  (for $L_{1,GHP}$) to $\Coal_{\delta^{\infty}}(\bfX^{\infty},t)$.
\end{enumerate}
\end{prop}
\begin{proof}

Let us fix $\eps\in]0,1[$. Let $\eps_0\in]0,\eps/2[$ be such that $\eps_0\not\in\sizes(\bfX^{\infty})$ and 
$$\mu(X^\infty_{\leq \eps_0})\leq \eps\;.$$
Proposition~\ref{prop:Npolish} shows that $\delta_{LP}(\bfX^n_{>\eps_0},\bfX^\infty_{>\eps_0})$ goes to zero as $n$ goes to infinity. Let $\tilde\eps$ be a positive real number to be chosen later, depending only on $\eps$, $\eps_0$, $\mu^\infty(X^\infty)$ and $T$. Let $n$ be large enough so that
$$\delta_{LP}(\bfX^n_{>\eps_0},\bfX^\infty_{>\eps_0})\leq \tilde\eps $$
$$|\delta^n-\delta^\infty|\leq \tilde\eps $$
and
$$\|\sizes(\bfX^n)-\sizes(\bfX^\infty)\|_1\leq \eps\;.$$
Let $k:=\#\comp(X^\infty_{>\eps_0})$ and notice that
$$k\leq \frac{\mu^\infty(X^\infty)}{\eps_0}\;.$$
Lemma~\ref{lemm:3GHP} shows that
$$d_{GHP}(\bfX^n_{>\eps_0},\bfX^\infty_{>\eps_0})\leq \tilde\eps \frac{\mu^\infty(X^\infty)}{\eps_0}\;.$$
Notice also that:
$$\mu^n(X^n_{\leq \eps_0})\leq \mu^\infty(X^\infty_{\leq \eps_0})+\|\sizes(\bfX^n)-\sizes(\bfX^\infty)\|_1\leq 2\eps\;,$$
and that
$$\mu^n(X^n)\leq \mu^\infty(X^\infty)+\|\sizes(\bfX^n)-\sizes(\bfX^\infty)\|_1\leq \mu^\infty(X^\infty)+\eps\;,$$
Thus, using Lemma~\ref{lemm:couplageN1}, one may couple the coalescence on $X^n_{>\eps_0}$ and $X^{\infty}_{>\eps_0}$ in such a way that there is an event $A$ satisfying the following. On $A$, for any $t\leq T$,
\[d_{GHP}(\Coal_{\delta^n}(\bfX^n_{>\eps_0},t),\Coal_{\delta^{\infty}}(\bfX^\infty_{>\eps_0},t)) \leq \alpha(\tilde\eps,\eps_0,T,\mu^\infty(X^\infty))\]
with
\[\alpha(\tilde\eps,\eps_0,T,\mu^\infty(X^\infty)):=(T\mu^\infty(X^{\infty})^2+1)\sqrt{2\tilde\eps\frac{\mu^\infty(X^\infty)}{\eps_0}+\tilde\eps}\]
and furthermore,
\[\PP(A^c) \leq \beta(\tilde\eps,\eps_0,T,\mu^\infty(X^\infty))\]
with
\[\beta(\tilde\eps,\eps_0,T,\mu^\infty(X^\infty)):=T\tilde\eps\frac{\mu^\infty(X^\infty)}{\eps_0}(10+16\mu^\infty(X^\infty))+2\eps+2\sqrt{2\tilde\eps\frac{\mu^\infty(X^\infty)}{\eps_0}+\tilde\eps}\;.\]
Using Lemma~\ref{lemm:3GHP}, we obtain that on $A$, for any $t\leq T$,
$$ L_{GHP}(\Coal_{\delta^n}(\bfX^n_{>\eps_0},t),\Coal_{\delta}(\bfX^\infty_{>\eps_0},t))\leq 16\frac{\mu^\infty(X^\infty)}{\eps_0}\alpha (\tilde\eps,\eps_0,T,\mu^\infty(X^\infty))\;.$$

Now, independently from this coupling, let us couple two independent Poisson random set on $(X^n)^2\setminus(X^n_{>\eps_0})^2\times[0,T]$ and $(X^\infty)^2\setminus(X^\infty_{>\eps_0})^2\times [0,T]$ (with intensities given by the restrictions of $\mu^n$ and $\mu^\infty$). Let $B$ denote the event that these Poisson random sets are both empty. Then,
\begin{eqnarray*}
  \PP(B^c)&\leq &1-e^{-\frac{T}{2}(\mu^n(X^n)^2-\mu^n(X^n_{>\eps_0})^2)}+1-e^{-\frac{T}{2}(\mu^\infty(X^\infty)^2-\mu^\infty(X^\infty_{>\eps_0})^2)}\;,\\
  &\leq & T\mu^n(X^n)\mu^n(X^n_{\leq\eps_0})+T\mu^\infty(X^\infty)\mu^\infty(X^\infty_{\leq\eps_0})\;,\\
  &\leq &6T\eps(\mu^\infty(X^\infty)+\eps)=:\gamma(\eps,\mu^\infty(X^\infty),T)\;.
\end{eqnarray*}
On $A\cap B$, we obtain that for any $t\leq T$,
$$L_{GHP}(\Coal_{\delta^n}(\bfX^n,t),\Coal_{\delta}(\bfX^\infty,t))\leq 2(1+8 \frac{\mu^\infty(X^\infty)}{\eps_0})\alpha (\tilde\eps,\eps_0,T,\mu^\infty(X^\infty))\;.$$
Furthermore,
$$\PP((A\cap B)^c)\leq \beta(\tilde\eps,\eps_0,T,\mu^\infty(X^\infty))+\gamma(\eps,\mu^\infty(X^\infty),T)\;.$$
Thus, one may choose $\tilde\eps$ as a function of $T$, $\eps$, $\eps_0$ and $\mu^\infty(X^\infty)$ such that with probability at least $1-C\eps$,
$$ L_{GHP}(\Coal_{\delta^n}(\bfX^n,t),\Coal_{\delta}(\bfX^\infty,t))\leq \eps $$
for some finite constant $C$ depending only on $\mu(X^\infty)$.

Furthermore, since the multigraphs $\MG(X^\infty,t)$ and $\MG(X^n,t)$ are the same for any $t\leq T$ in this coupling,
\begin{eqnarray*}
&&\|\sizes(\Coal_{\delta^{n}}(X^n,t))-\sizes(\Coal_{\delta}(X^\infty,t))\|_1\\
&\leq&\|\sizes(\Coal_{\delta^{n}}(X^n_{>\eps_0},t))-\sizes(\Coal_{\delta}(X^\infty_{>\eps_0},t))\|_1\\
&& +\|\sizes(\Coal_{\delta^{n}}(X^n_{\leq \eps_0},t))-\sizes(\Coal_{\delta}(X^\infty_{\leq \eps_0},t))\|_1\\
&\leq & \|\sizes(X^n_{>\eps_0},t)-\sizes(X^\infty_{>\eps_0},t)\|_1+\mu(X^n_{\leq \eps_0})+ \mu(X^\infty_{\leq \eps_0})\\
&\leq& 4\eps\;.
\end{eqnarray*}
This shows $(i)$. To obtain $(ii)$, notice that for any $s$ and $\eta>0$,
$$\PP(\exists t\in[s,s+\eta]:\Coal_{\delta^{\infty}}(\bfX^\infty,t)\not=\Coal_{\delta^{\infty}}(\bfX^\infty,s))\leq \mu(X^\infty)^2\eta\;.$$
Thus, $(ii)$ is a simple consequence of $(i)$.
\end{proof}
We shall need the following variation of Proposition~\ref{prop:FellerN1} when studying simultaneous coalescence and fragmentation in section~\ref{sec:dynperc}.
\begin{prop}
\label{prop:FellerN1surplus}
Let $\bfX^{n}=(X^{n},d^{n},\mu^{n})$, $n\in\NNbar$ be a sequence of random variables in $\calN_1^{graph}$  and $(\delta^{n})_{n\geq 0}$ a sequence of non-negative real numbers. Suppose that: 
\begin{enumerate}[(a)]
\item $(\bfX^{n})_{n\geq 0}$ converges in distribution for $L_{1,GHP}^{surplus}$ to $\bfX^{\infty}$ as $n$ goes to infinity,
\item $\delta^n\xrightarrow[n\rightarrow\infty]{}\delta$.
\end{enumerate}
Then, 
\begin{enumerate}[(i)]
\item $(\Coal_{\delta^{n}}(\bfX^{n},t))_{t\geq 0}$ converges in distribution  (for the Skorokhod topology associated to $L_{1,GHP}^{surplus}$) to $(\Coal_{\delta}(\bfX^{\infty},t))_{t\geq 0}$,
\item if $t^n\xrightarrow[n\rightarrow\infty]{}t$, $\Coal_{\delta^{n}}(\bfX^{n},t^n)$ converges in distribution to $\Coal_{\delta}(\bfX^{\infty},t)$ for $L_{1,GHP}^{surplus}$.
\end{enumerate}
\end{prop}
\begin{dem}
Notice first that when $\bfX$ belongs to $\calN_1^{graph}$, then with probability one, $\Coal_{\delta}(\bfX^{n},t)$ is in $\calN_1^{graph}$ for any $t\geq 0$. Indeed, since $X$ has finite mass, there is with probability one a finite number of points in the Poisson process $\calP_t^+$ on $X^2$ for any $t\geq 0$. 

Now, the proof is esentially the same as the one of Proposition~\ref{prop:FellerN1}, except that since  $(\bfX^{n})_{n\geq 0}$ converges to $\bfX^{\infty}$ for $L_{1,GHP}^{surplus}$, one may use $d_{GHP}^{surplus}$ instead of $d_{GHP}$. The fact that the multigraphs $\MG(X^n,s)$ and $\MG(X^\infty,s)$ are the same for any $s\leq T$, and that in the coupling no point of the Poisson processes touches $\bfX^{n}_{\leq \eps_0}$ or $\bfX^{\infty}_{\leq \eps_0}$ implies that for each component of  $\Coal_{\delta^n}(X^n_{\geq \eps_0},s)$, its surplus is the same as the surplus of the corresponding component in $\Coal_{\delta}(X^\infty_{\geq \eps_0},s)$.
\end{dem}

\subsection{Structural result for Aldous' multiplicative coalescent}

Recall the definition of the multigraph $\MG(x,t)$ for $x\in\ell^2$ in section~\ref{subsec:multiplicative}. We shall use notations analogous to those in Definition~\ref{defi:calN}. For instance, for $x\in\ell^2_+$, $x_{\leq \eps}$ denotes the element in $\ell^2_+$ defined by:
$$\forall i\in\NN,\;x_{\leq \eps}(i)=x(i)\II_{x(i)\leq \eps}\;.$$
Also, for $i\in\NN$, $x\setminus \{i\}$ denotes the element in $\ell^2$ defined by:
$$\forall j\in\NN,\;(x\setminus\{i\})(j)=x(j)\II_{j\not=i}\;.$$

Notice that at time $0$, the components of $\MG(x,0)$ are the singletons $\{i\}$ for $i\in\NN$. Let us fix some $\eps>0$ and say that a component of $\MG(x,t)$ is significant if it has weight larger than $\eps$, i.e the sum of the weights $x_i$ of its vertices $i$ is larger than $\eps$. In Lemma~\ref{lemm:structureAldous} below and its corollary, we shall derive three scales (at time $0$), namely, Large, Medium and Small such that with high probability (as $\eps$ goes to zero), every significant component of $\MG(x,t)$ is made of a \emph{heart}\footnote{This term is not standard, but as will be apparent in the subsequent proofs, the heart of a significant component of $\MG(x,t)$ is uniquely defined when $x$, $t$ and the weight of a small component are defined. This notion has nothing to do with the core of a graph.} made of Large or Medium components of $\MG(x,0)$ to which are attached \emph{hanging trees} of small or medium components of $\MG(x,0)$ such that the component of the trees attached to the heart are small components (see Figure~\ref{fig:structure}) and the mass contained in the hanging trees is at most medium. Furthermore, these scales depend on $x$, $\eps$ and $t$ through the functions $\alpha\mapsto \|x_{\leq \alpha}\|_2$ and $K\mapsto \PP(S(x,t)\geq K)$. This picture will be fundamental in the proof of the Feller properties of the coalescent and dynamical percolation processes. Indeed, it will imply that small components of $\bfX\in\calN_2$ have small influence on the geometry of  the large components of $\Coal(\bfX,t)$. This will allow us to approximate $\Coal(\bfX,t)$ by a truncated version, $\Coal(\bfX_{\geq \eps},t)$ for $\eps$ small enough, and to reduce the Feller property on $\calN_2$ to the Feller property on $\calN_1$, where a finite number of identifications occur on every finite interval of time.
\begin{lemm}
\label{lemm:structureAldous}
Let $x\in\ell^2(\NN)$, $T\geq 0$ and $0<\eps<1$. Suppose that:
\begin{enumerate}[(i)]
\item\label{hyp:structure1} $K\geq 1$ is such that:
$$\PP(S(x,T)\geq K)\leq \frac{\eps}{100}\;,$$
\item\label{hyp:structure2} $\eps_1\in (0,\eps)$ is such that:
$$S(x_{\leq \eps_1},0)\leq \frac{\eps^2}{100(1+T+KT^2)}\;,$$
\item\label{hyp:structure3} $\eps_2\in (0,\eps_1)$ is such that:
  $$S(x_{\leq \eps_2},0)\leq \frac{2\eps_1^2\eps^2}{100(1+T(K+2))^2}\;.$$
  \end{enumerate}
  Then with probability larger than $1-\eps$, the following holds for any $t\leq T$,
\begin{enumerate}[(a)]
\item\label{structure1} every component of $\MG(x,t)$ of size larger than $\eps$ contains a component of $\MG(x,0)$ of size larger than $\eps_1$.
\item\label{structureforest} no component of $\MG(x_{\leq \eps_1},0)$ is contained in a cycle in $\MG(x,t)$.
\item\label{structureoneedge} for each component $m$ of $\MG(x_{\leq \eps_1},t)$ and each component $m'$ of $\MG(x_{> \eps_1},t)$, there is at most one edge between $m$ and $m'$ in $\MG(x,t)$.
\item\label{structurel2} $S(x,t)-S(x_{>\eps_2},t)\leq 2\eps_1^2$.
\item\label{structurecouronne} for any component $\{i\}$ of $\MG(x,0)$ of size larger than $\eps_1$, the difference between the sizes of the component containing $i$ in $\MG(x,t)$ and the one containing $i$ in $\MG(x_{>\eps_2},t)$ is less than $\eps_1$.
\end{enumerate}
\end{lemm}

We may now state a precise version of the structural decomposition sketched just before Lemma~\ref{lemm:structureAldous}. We state it uniformly on a convergent sequence in $\ell^2_\searrow$ because it will be convenient to prove the almost Feller property on $\calS$, Theorem~\ref{theo:AlmostFellercoal}.

When $\eps_2\leq \eps_1\leq \eps$ are three positive real numbers, we shall say that a component is \emph{significant} if it has size larger than $\eps$, \emph{Large} if it has size larger than $\eps_1$, \emph{Medium} for a size in $(\eps_2,\eps_1]$ and \emph{Small} for a size not larger than $\eps_2$.
\begin{coro}
  \label{coro:picture_structure}Let $x^n$ be a sequence in $\ell^2_\searrow$ converging to $x^\infty$ in $\ell^2$. Then, for any $\eps>0$, and any $T>0$ there exists $\eps_1$ and $\eps_2$ such that for any $n\in\NNbar$, with probability at least $1-\eps$ the following holds for any $t\in[0,T]$:
  \begin{enumerate}[(a)]
  \item\label{coro:structureheart} every significant component of $\MG(x^n,t)$ is made of a connected \emph{heart} made of Large or Medium components of $\MG(x^n,0)$ to which are attached \emph{hanging trees} (each one attached by a single edge to the heart) of Small or Medium components of $\MG(x^n,0)$ such that the components of the trees attached to the heart are Small components and the mass contained in the hanging trees is less than $\eps_1$,
  \item\label{coro:nocycle} no Medium or Small component of $\MG(x^n,0)$ belongs to a cycle in $\MG(x^n,t)$,
  \item\label{coro:S} $S(x^n,t)-S(x^n_{>\eps_2},t)\leq 2\eps_1^2$,
  \end{enumerate}
\end{coro}
\begin{figure}[!htbp]
\begin{center}
\psfrag{L1}{size $\geq \eps_1$}
\psfrag{M1}{size $\in [\eps_2,\eps_1[$}
\psfrag{S1}{size $< \eps_2$} 
\psfrag{L}{$L$}
\psfrag{M}{$\scriptstyle M$}
\psfrag{S}{$\scriptscriptstyle S$} 
\psfrag{H}{$Heart$}
\includegraphics[width=8cm]{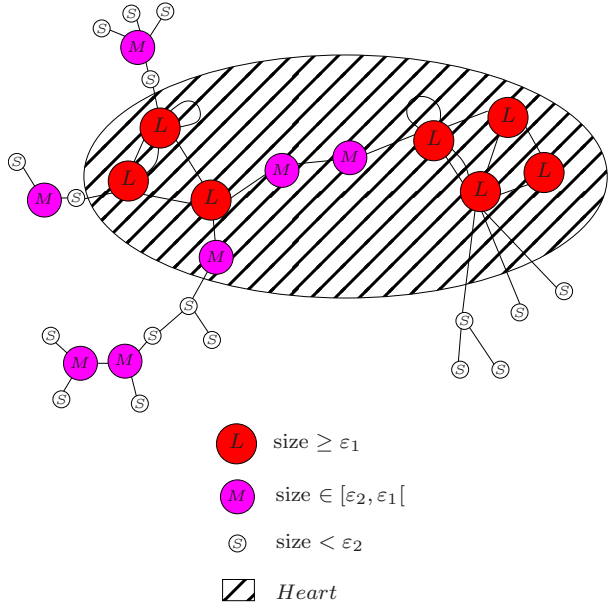}
\caption{The structure of a significant component}
\label{fig:structure}
\end{center}
\end{figure}
The proof of Lemma~\ref{lemm:structureAldous} relies essentially on Aldous' analysis of the multiplicative coalescent. 

\begin{proof}(of Lemma~\ref{lemm:structureAldous})

If for some $t\leq T$ there exists a significant component of $\MG(x,t)$ which does not contain any large component of $x$, then $S(x_{\leq \eps_1},t)>\eps^2$ and thus $S(x_{\leq \eps_1},T)>\eps^2$, since $S(x,\cdot)$ is nondecreasing. Thus Lemma~\ref{lemm:Aldous20} shows that the probability of (\ref{structure1}) is larger than $1-\eps/4$, as soon as hypothesis~(\ref{hyp:structure2}) of Lemma~\ref{lemm:structureAldous} holds.

Let $\{i\}$ be a component of $\MG(x,0)$ and define the event
\begin{center}
$A_i=\{$there exist at least two edges of $\MG(x,T)$ connecting $i$ to the same component of $\MG(x\setminus\{i\},T)\}$.
\end{center}
Due to the properties of the Poisson process defining $\MG(x,T)$, one sees that conditionnally on $\MG(x\setminus\{i\},T)$, the number of edges of $\MG(x,T)$ connecting a  fixed component $m$ of $\MG(x\setminus\{i\},T)$  to $i$  is a Poisson random variable with parameter $\sum_{j\in m}Tx_ix_j$. Thus,
\begin{eqnarray*}
\PP(A_i|\MG(x\setminus\{i\},T))&\leq &\sum_{m\;\mathrm{ c.c. of}\;\MG(x\setminus\{i\},T)}(Tx_i\sum_{j\in m}x_j)^2\;,\\
&=&T^2x_i^2S(x\setminus\{i\},T)\;,
\end{eqnarray*}
thus,
\begin{eqnarray}
\nonumber \PP(A_i\cap\{S(x,T)\leq K\})&\leq &\EE[\PP(A_i|\MG(x\setminus\{i\},T))\II_{S(x\setminus\{i\},T)\leq K}]\;,\\
\label{eq:forest}&\leq &KT^2x_i^2\;.
\end{eqnarray}
We obtain thus:
$$\PP(\cup_{i\in\NN\mathrm{ s.t. }\;x_i\leq \eps_1}A_i)\leq KT^2S(x_{\leq \eps_1},0)+\PP(S(x,T)>K)\;,$$
which shows that the probability of (\ref{structureforest}) is at least $1-\eps/4$ as soon as hypotheses~(\ref{hyp:structure1}) and~(\ref{hyp:structure2}) of Lemma~\ref{lemm:structureAldous} hold.

The proof of (\ref{structureoneedge}) is similar. Let 
\begin{center}
$B_T:=\{$there exist at least two edges of $\MG(x,T)$ connecting a component $m$ of $\MG(x_{\leq \eps_1},T)$ to a component $m'$ of $\MG(x_{>\eps_1},T)$\}.
\end{center}
Then,
\begin{eqnarray*}
&&\PP(B_T|\MG(x_{>\eps_1},T),\MG(x_{\leq\eps_1},T))\\
&\leq &\sum_{m\;\mathrm{ c.c. of}\;\MG(x_{\leq\eps_1},T)}\sum_{m'\;\mathrm{ c.c. of}\;\MG(x_{>\eps_1},T)}(\sum_{i\in m}x_i\sum_{j\in m'}x_jT)^2\;,\\
&=&T^2S(x_{\leq\eps_1},T)S(x_{>\eps_1},T)\;.
\end{eqnarray*}
Thus,
\begin{eqnarray}
\nonumber \PP(B_T)\leq \frac{\eps}{5}+\PP(S(x,T)\geq K)+\PP(S(x_{\leq\eps_1},T)\geq \frac{\eps}{5KT^2})\;,
\end{eqnarray}
which shows using Lemma~\ref{lemm:Aldous20} that the probability of (\ref{structureoneedge}) is at least $1-\eps/4$ as soon as hypotheses~(\ref{hyp:structure1}) and~(\ref{hyp:structure2}) of Lemma~\ref{lemm:structureAldous} hold (notice that $B_t\subset B_T$ if $t\leq T$).

Now, let $Y$ be the supremum, over Large components $\{i\}$ of $\MG(x,0)$, of the difference between the sizes of the component containing $i$ in $\MG(x,t)$ and the one containing $i$ in $\MG(x_{>\eps_2},t)$. 
Notice that if $Y\geq \alpha$, then $S(x,t)\geq S(x_{> \eps_2},t)+2\eps_1\alpha$, which implies $S(x,T)\geq S(x_{> \eps_2},T)+2\eps_1\alpha$ when \eqref{structureforest} holds, thanks to Lemma~\ref{lemm:pourSkorL2}. Thus points~(\ref{structurecouronne}) and (\ref{structurel2}) will be proved if we show that  $S(x,T)> S(x_{> \eps_2},T)+2\eps_1^2$ with probability at most $\eps/4$.

Define the event $$C=\{S(x,T)> S(x_{> \eps_2},T)+2\eps_1^2\}\;.$$ Lemma~\ref{lemm:Aldous23} shows that:
\begin{eqnarray*}
&&\PP(C\mbox{ and }S(x,T)\leq K\mbox{ and }S(x_{\leq \eps_2},T)\leq \beta\})\\
&\leq& (1+T(K+2\eps_1^2))^2\frac{\beta}{2\eps_1^2}\;.
\end{eqnarray*}
Thus,
$$\PP(C)\leq (1+T(K+2))^2\frac{\beta}{2\eps_1^2}+\PP(S(x_{\leq \eps_2},T)> \beta)+\PP(S(x,T)>K)\;,$$
which is less than $\eps/4$ if we take
\[\beta=\frac{2\eps_1^2\eps}{100(1+T(K+2))^2}\]
and if hypotheses~(\ref{hyp:structure1}), (\ref{hyp:structure2}) and (\ref{hyp:structure3}) of Lemma~\ref{lemm:structureAldous} hold.
\end{proof}

\begin{proof}(of Corollary~\ref{coro:picture_structure})
  Since $x^n$ converges to $x^{\infty}$ in $\ell^2$,
  $$\sup_{n\in\NN}\|x^n_{\leq \eps}\|_2\xrightarrow[\eps\rightarrow 0]{}0\;.$$
  Also, the Feller property of the multiplicative coalescent in $\ell^2$, cf. \cite{Aldous97Gnp} implies that the distributions of the sizes of $\MG(x^n,T)$ for $n\in \NNbar$ form a compact family of probability measures on $\ell^2$. Thus,
  \begin{equation}
    \label{eq:l2compact}
    \sup_{n\in \NNbar}\PP(S(x^n,T)\geq K)\xrightarrow[K\rightarrow +\infty]{}0\;.
  \end{equation}
  This shows that for any $T$ and $\eps$, one may find $K$, $\eps_1$ and $\eps_2$ such that the three hypotheses of Lemma~\ref{lemm:structureAldous} hold for $x^n$ uniformly  over $n\in\NNbar$. 

  Now suppose that $x$, $t$, $\eps$, $\eps_1$ and $\eps_2$ are such that (\ref{structure1}), (\ref{structureforest}), (\ref{structurecouronne}) and (\ref{structureoneedge}) of Lemma~\ref{lemm:structureAldous} hold. Let $m$ be a significant component of $\MG(x,t)$. It contains a large component $\{i\}$ of $\MG(x,0)$ by point (\ref{structure1}). Let $\sigma(m)$ denote the component of $\MG(x_{>\eps_2},t)$ containing $\{i\}$. Point (\ref{structurecouronne}) implies that two large components of $\MG(x,0)$ are connected in $\MG(x,t)$ if and only if they are connected in $\MG(x_{>\eps_2},t)$, that is only through Large or Medium components.  This shows  that $\sigma(m)$ does not depend on the choice of the large component $\{i\}$ included in $m$ and that there cannot be any large component in $m\setminus \sigma(m)$. Let us define the heart of $m$ as $\sigma(m)$. It is made of Large or Medium components, and $m\setminus \sigma(m)$ is a graph of medium or small components. Now if a medium component in $m\setminus \sigma(m)$ was directly connected to some component of $\sigma(m)$, it would be connected to $m$ in $\MG(x_{>\eps_2},t)$, and thus would belong to $\sigma(m)$. Thus, the exterior boundary of $\sigma(m)$ in $m$ is made of small components. Point (\ref{structureforest}) shows that no Medium or Small component of $\MG(x,0)$ belongs to a cycle in $\MG(x,t)$, which implies notably that $m\setminus \sigma(m)$ is a forest, and point (\ref{structureoneedge}) shows that each tree of this forest is attached by a single edge to $\sigma(m)$.
\end{proof}

A useful by-product of the proof of Lemma~\ref{lemm:structureAldous} and Corollary~\ref{coro:picture_structure} is the following simple lemma.
\begin{lemm}
  \label{lemm:forest} For $x\in\ell^2$, $\eps>0$ and $T>0$ let $A(x,\eps,T)$ be the event that for any $t\leq T$:
  \begin{itemize}
  \item $\MG(x_{\leq \eps},t)$ is a forest and
  \item there is at most one edge betweeen any connected component of $\MG(x_{\leq \eps},t)$ and any component of $\MG(x_{>\eps},t)$.
  \end{itemize}
  Suppose that $x^n$ converges to $x^\infty$ in $\ell^2_\searrow$ as $n$ goes to infinity. Then, for any $T>0$
  $$\inf_{n\in\NNbar}\PP(A(x^n,\eps,T))\xrightarrow[\eps\rightarrow 0]{}1\;.$$
\end{lemm}

\begin{rem}
\label{rem:N2moinsS}
Let us give an example of an m.s-m.s which is in $\calN_2$ but not in $\calS$. Let $I_i$, $i\geq 1$ be disjoint copies of the interval $[0,1]$, with its usual metric, and equip $I_i$ with the measure $\frac{1}{i}(\delta_0+\delta_1)$. Then, $\bfX\in\calN_2\setminus \calS$. In fact, thanks to Lemma~\ref{lemm:structureAldous}, for any $\eps>0$ and $t>0$ every component of $\Coal_0((\bfX)_{>\eps},t)$ is unbounded since it contains a forest of an infinite number (since the sizes are not in $\ell^1$) of components of diameter 1.
\end{rem}

\subsection{The Coalescent on $\calS$}
\label{subsec:coalS}

The aim of this section is to prove Theorem~\ref{theo:AlmostFellercoal}. We shall first prove two lemmas. 
\begin{lemm}
\label{lemm:calSdansN}
Let $\bfX$ be an m.s-m.s. 
\begin{enumerate}[(i)]
\item If $\bfX\in\calS$ then, almost surely, for any $t\geq 0$, $\Coal_0(\bfX,t)$ belongs to $\calN_2$.
\item If $\bfX\in\calN_2$ and is a length space, then, almost surely,  for any $t\geq 0$, $\Coal_0(\bfX,t)$ is a length space and the commutation relation \eqref{eq:commutation} holds.
\item If the components of $\bfX$ are $\RR$-graphs then, almost surely, for any $t\geq 0$, the components of $\Frag(\bfX,t)$ are $\RR$-graphs. Consequently, if $\bfX$ belongs to $\calS^{graph}$, then, almost surely, for any $t\geq 0$, $\Frag(\bfX,t)$ belongs to $\calS^{graph}$.
\end{enumerate}
\end{lemm}
\begin{proof}
\begin{enumerate}[(i)]
\item Suppose that $\bfX$ belongs to $\calS$ and let us show that with probability one, $\Coal_0(\bfX,t)$ has totally bounded components for any $t\geq 0$. Let $\alpha>0$ and $\eps\in]0,1[$ be fixed and let
$$B(\eps):=\{\supdiam(\Coal_0(\bfX_{\leq \eps},T))\leq  \alpha\}\;.$$
Since $\bfX$ satisfies \eqref{eq:supdiam}, 
$$\PP(B(\eps)^c)\xrightarrow[\eps\rightarrow 0]{}0$$
Now, we perform coalescence and use the obvious coupling between $\Coal_0(\bfX,t)$ and $\Coal_0(\bfX_{>\eps_2},t)$. Recall the multigraph $\MG(X,t)$ introduced in section~\ref{subsec:coal}. We let $S(X,t)$ denote the sum of the squares of the masses of the components in $\Coal_0(\bfX,t)$ (or $\MG(X,t)$). Let $\eps_1$ and $\eps_2$ be positive numbers to be chosen soon and let $A(\eps)$ be the event that for any $t\leq T$:
\begin{enumerate}[(a)]
\item every significant component of $\MG(X,t)$ is made of a connected \emph{heart} made of Large or Medium components of $\MG(X,0)$ to which are attached \emph{hanging trees} (each one attached by a single edge to the heart) of Small or Medium components of $\MG(X,0)$ such that the components of the trees attached to the heart are Small components and the mass contained in the hanging trees is less than $\eps_1$,
\item no Medium or Small component of $\MG(X,0)$ belongs to a cycle in $\MG(X,t)$,
\item $S(X,t)-S(X_{>\eps_2},t)\leq 2\eps_1^2$.
\end{enumerate}
Then, Corollary~\ref{coro:picture_structure} (with $x^n=x^\infty=\sizes(X)$) shows that one can choose $\eps_1$ and $\eps_2$ (as functions of $\eps$, $T$ and $\sizes(\bfX)$) such that:
$$\PP(A(\eps)^c)\xrightarrow[\eps\rightarrow 0]{}0\;.$$
Then, on $A(\eps)\cap B(\eps)$, we have that for any $t\in[0,T]$, any component of size larger than $\eps$ of $\Coal_0(\bfX,t)$ can be covered with a finite number of balls of radius $2\alpha$. Indeed, if $m$ is such a component, one may first cover the heart with a finite number of balls of radius $\alpha$ since the heart of $m$ is composed of a finite number of totally bounded components of $\bfX$ glued together, and then if we increase the radius to $2\alpha$, those balls will cover the whole component $m$ because we are on $B(\eps)$. Making $\eps$ go to zero, we see that with probability one, for any $t\in[0,T]$ every component of $\Coal_0(\bfX,t)$ can be covered with a finite number of balls of radius $2\alpha$. Then, letting $\alpha$ go to zero, we see that with probability one, for any $t\in[0,T]$ every component of $\Coal_0(\bfX,t)$ is totally bounded, so $\Coal_0(\bfX,t)\in\calN_2$. 

\item The fact that $\Coal_0(\bfX,t)$ is a length space is an immediate consequence of Remark~\ref{rem:frag} $(iv)$. If $\bfX=(X,d,\mu)\in\calN_2$, using the same notation as above, one can still guarantee that:
$$\PP(A(\eps)^c)\xrightarrow[\eps\rightarrow 0]{}0\;.$$
On $A(\eps)$, for any $x$ and $y$ in a component of $\Coal_0(\bfX,t)$ of mass larger than $\eps$ (i.e a significant component), there is only a finite number of simple paths from $x$ to $y$, and every such simple path takes a finite number of shortcuts of the Poisson process $\calP_t^+$. Letting $\eps$ go to zero, this holds almost surely for any component of $\Coal_0(\bfX,\calP_t^+)$. Furthermore, since $\ell_{X}$ is diffuse and $\calP^-$ and $\calP^+$ are independent, almost surely one has, for any $t$, and using the notation of Lemma~\ref{lemm:commutdeter}
$$\calP_t^{-,d}\cap\{x\in X:\exists y\in X,\; (x,y)\text{ or }(y,x)\in \calP_t^+\}=\emptyset\;.$$
Thus, Lemma~\ref{lemm:commutdeter} shows that $(ii)$ holds.
\item Using Lemma~\ref{lemm:shortcut}, $X$ is isometric to $\Coal_0(X',A)$ where $X'$ is an m.s-m.s whose components are real trees, and $A\subset \bigcup_{m\in\comp(X')}m^2$ is finite on any $m^2$. Again, since  $\ell_{X}$ is diffuse, almost surely, for any $t$, and using the notation of Lemma~\ref{lemm:commutdeter}
$$\calP_t^{-,d}\cap\{x\in X:\exists y\in X,\; (x,y)\text{ or }(y,x)\in A\}=\emptyset\;.$$
Thus, Lemma~\ref{lemm:commutdeter} shows that 
$$\Frag(X,\calP_t^-)=\Frag(\Coal_0(X',A),\calP_t^-)=\Coal_0(\Frag(X',\calP_t^-),A)\;.$$
The components of $\Frag(X',\calP_t^-)$ are $\RR$-trees, thus the components of $\Coal_0(\Frag(X',\calP_t^-),A)$ are $\RR$-graphs. The last part of $(iii)$ follows from the fact that $\calS^{graph}$ is clearly stable by fragmentation.
\end{enumerate}
\end{proof}
\begin{rem}
  \label{rem:coalScoalfragS}
\begin{enumerate}[(i)]
\item If $\bfX\in\calN$ and $\calP$ is as in Definition~\ref{defi:cut}, it may happen that $\Frag(\bfX,\calP)$ has a component of mass zero. In this case, $\Frag(\bfX,\calP)$ does not belong to $\calN$, stricly speaking. However, $\Frag(\bfX,\calP)$ is at zero $L_{GHP}$-distance from an element of $\calN$, which is $\Frag(\bfX,\calP)|_{\cup_{\eps>0}\calM_{>\eps}}$. In fact, we could have defined $\calN$ as the quotient of the set of counting measures on $\calM$ with respect to the equivalence relation defined by being at zero $L_{GHP}$-distance. This space is isometric to $\calN$ modulo the addition of components of null masses. Then  $\Frag(\bfX,\calP)$ would have always belonged to $\calN$. But I feel that it would have obscured the definition of $\calN$. In the sequel, we shall keep in mind that components of null masses are neglected.
\item It is apparent from the proof of point {\it (ii)} above that when $\bfX$ belongs to $\calS^{graph}$, then $\Coal_0(\bfX,t)$ has components which are precompact $\RR$-trees with a countable number of identifications. Thus, $\Coal_0(\bfX,t)$ is in $\calS^{graph}$ if and only if the numbers of identifications on any component is finite. One consequence of this is that if $\Coal_0(\bfX,t)$ is in $\calS^{graph}$ for some $t\geq 0$, then $\Coal_0(\bfX,s)$ and $\CoalFrag(\bfX,s)$ are in $\calS^{graph}$ for every $s\in[0,t]$.
  \end{enumerate}
\end{rem}

\begin{lemm}
  \label{lemm:calSstable}
  Let $\bfX^{n}=(X^{n},d^{n},\mu^{n})$, $n\geq 0$ be a sequence of random variables in $\calS$ and $(\delta^{n})_{n\geq 0}$ be a sequence of non-negative real numbers. Suppose that: 
\begin{enumerate}[(i)]
\item $(\bfX^{n})$ converges in distribution for $L_{2,GHP}$ to $\bfX^{\infty}=(X^{\infty},d^{\infty},\mu^{\infty})$ as $n$ goes to infinity,
\item $\delta^n\xrightarrow[n\rightarrow\infty]{}0$,
\item For any $\alpha>0$ and any $T>0$, $$\limsup_{n\in\NN} \PP(\supdiam(\Coal_{\delta^{n}}(\bfX^{n}_{\leq \eps},T))>\alpha)\xrightarrow[\eps\rightarrow 0]{} 0\;.$$
\end{enumerate}
Then, with probability 1, $\Coal_0(\bfX^{\infty},t)$ belongs to $\calS$ for any $t\geq 0$.
\end{lemm}
\begin{proof}
First, the Feller property of the multiplicative coalescent, \cite[Proposition~5]{Aldous97Gnp},  shows that $\sizes(\Coal_{\delta^{n}}(\bfX^{n},T))$ converges in distribution (in $\ell_\searrow^2$) to $\sizes(\Coal_{0}(\bfX^{\infty},T))$. Together with Skorokhod's representation theorem and Lemma~\ref{lemm:forest}, this implies that:
\begin{equation}
  \label{eq:Xinftyforest}
  \PP[\MG(\bfX^{\infty}_{\leq \eps},T)\text{ is not a forest }]\xrightarrow[\eps \rightarrow 0]{}0\;.
\end{equation}
Notice that under the obvious coupling, when $\MG(\bfX^{\infty}_{\leq \eps},t+s)$ is a forest, Lemma \ref{lemm:obviouscouplingforest} implies that:
$$\supdiam(\Coal_0(\Coal_0(\bfX^{\infty},t)_{\leq \eps},s))\leq \supdiam(\Coal_0(\bfX^{\infty}_{\leq \eps},t+s))\;.$$
Thus, thanks to Lemmas \ref{lemm:forest} and \ref{lemm:calSdansN} it is enough to show that with probability one, $\bfX^{\infty}$ satisfies \eqref{eq:supdiam} for any $t\geq 0$.

Let $P_{\bfX^{\infty}}$ be the distribution of $\bfX^{\infty}$.  Then for $P_{\bfX^{\infty}}$-almost every $\bfX$ and every $t\in[0,T]$ and $\alpha>0$,
\begin{eqnarray*}
  &&\limsup_{\eps\rightarrow 0}\PP[\supdiam(\Coal_0(\bfX_{\leq \eps},t))>\alpha]\\
  &=&\limsup_{\eps\rightarrow 0}\PP[\supdiam(\Coal_0(\bfX_{\leq \eps},t))>\alpha\text{ and }\MG(\bfX_{\leq \eps},T)\text{ is a forest}]\\
  &\leq&\limsup_{\eps\rightarrow 0}\PP[\supdiam(\Coal_0(\bfX_{\leq \eps},T))>\alpha\text{ and }\MG(\bfX_{\leq \eps},T)\text{ is a forest}]  
\end{eqnarray*}
Thus,
\begin{eqnarray*}
  &&P_{\bfX^{\infty}}\{\bfX\in\calN_2\;:\sup_{\substack{t\leq T\\ \alpha >0}}\limsup_{\eps\rightarrow 0}\PP[\supdiam(\Coal_0(\bfX_{\leq \eps},t))>\alpha]>0\}\\
  &=&\sup_{\alpha>0}P_{\bfX^{\infty}}\{\bfX\;:\;\limsup_{\eps\rightarrow 0}\PP[\supdiam(\Coal_0(\bfX_{\leq \eps},T))>\alpha]>0\}\\
  &=&\sup_{\alpha>0}\sup_{\eta>0}\lim_{\eps\rightarrow 0}P_{\bfX^{\infty}}\{\bfX:\PP[\exists\eps'\in ]0,\eps], \supdiam(\Coal_0(\bfX_{\leq \eps'},T))>\alpha]>\eta\}\\
      &\leq &\sup_{\alpha>0}\sup_{\eta>0}\lim_{\eps\rightarrow 0}\frac{1}{\eta}\PP[\exists\eps'\in ]0, \eps]\;:\supdiam(\Coal_0(\bfX^{\infty}_{\leq \eps'},T))>\alpha]\;.
\end{eqnarray*}
Thus, using \eqref{eq:Xinftyforest}, it is sufficient to prove that for any $T\geq 0$, $\alpha>0$ and $\tilde\eps>0$,
\begin{equation}
\label{eq:supdiamlimite}
\PP\left[\begin{array}{c}\sup_{\eps'\in ]0, \eps]}\supdiam(\Coal_0(\bfX^{\infty}_{\leq \eps'},T))>\alpha \\ \text{ and  } \\ \MG(x^{\infty}_{\leq \tilde \eps},T)\text{ is a forest}\end{array}\right]\xrightarrow[\eps\rightarrow 0]{}0\;.
\end{equation}

Let $x^{\infty}:=\sizes(\bfX^{\infty})$. Notice first that there exists a decreasing sequence of positive numbers $(\eps_p)_{p\geq 0}$  going to zero and such that:
$$\forall p\in \NN,\;\PP[\eps_p\in x^{\infty}]=0\;.$$
Fix $\tilde \eps\geq \eps$ and choose the sequence so that $\eps_0\leq \tilde \eps$. Then, using the obvious coupling and Lemma~\ref{lemm:obviouscouplingforest}, 
\begin{eqnarray*}
&&\PP\left[\begin{array}{c}\sup_{\eps'\in ]0, \eps]}\supdiam(\Coal_0(\bfX^{\infty}_{\leq \eps'},T))>\alpha \\ \text{ and  } \\ \MG(x^{\infty}_{\leq \tilde \eps},T)\text{ is a forest}\end{array}\right]\\
&\leq & \PP[\supdiam(\Coal_0(\bfX^{\infty}_{\leq \eps},T))>\alpha\text{ and  }\MG(x^{\infty}_{\leq \tilde \eps},T)\text{ is a forest}]
\end{eqnarray*}
 Then, we have:
\begin{eqnarray*}
&&  \lim_{\eps\rightarrow 0}\PP(\supdiam(\Coal_{0}(\bfX^{\infty}_{\leq \eps},T))>\alpha\mbox{ and  }\MG(x^{\infty}_{\leq \tilde\eps},T)\mbox{ is a forest})\\
  &=&\lim_{m\rightarrow \infty}\PP(\supdiam(\Coal_{0}(\bfX^{\infty}_{\leq \eps_m},T))>\alpha\;.
  \end{eqnarray*}
Furthermore, define $\bfX_{m,p}:=(\bfX_{\leq \eps_m})_{>\eps_p}$ for $m\leq p$. Then, 
$$ \PP(\supdiam(\Coal_{0}(\bfX^{\infty}_{\leq \eps_m},T))>\alpha)
=\lim_{p\rightarrow \infty}\PP(\supdiam(\Coal_{0}(\bfX^{\infty}_{m,p},T))>\alpha)$$
Now, Proposition~\ref{prop:FellerN1} implies that $(\Coal_{\delta^{n}}(\bfX^{n}_{m,p},T)$ converges in distribution to $(\Coal_{0}(\bfX^{\infty}_{m,p},T)$ for any $m\leq p$.  Since we are dealing here with finite collections of m.s-m.s with positive masses, this entails that for any $m\leq p$,
\begin{eqnarray*}
&&\PP(\supdiam(\Coal_{0}(\bfX^{\infty}_{m,p},T))>\alpha)\\
&\leq & \limsup_{n\rightarrow\infty}\PP(\supdiam(\Coal_{\delta^{n}}(\bfX^{n}_{m,p},T))>\alpha)\\
&\leq &\limsup_{n\rightarrow\infty}\PP(\supdiam(\Coal_{\delta^{n}}(\bfX^{n}_{m,p},T))>\alpha\mbox{ and  }\MG(\bfX^{n}_{\leq \tilde\eps},T)\mbox{ is a forest})\\
&&+\limsup_{n\rightarrow\infty}\PP(\MG(\bfX^{n}_{\leq \eps},T)\mbox{ is not a forest}))\\
&\leq & \limsup_{n\rightarrow\infty}\PP(\supdiam(\Coal_{\delta^{n}}(\bfX^{n}_{\leq \eps_m},T))>\alpha\mbox{ and  }\MG(\bfX^{n}_{\leq \tilde\eps},T)\mbox{ is a forest})\\
&&+\limsup_{n\rightarrow\infty}\PP(\MG(\bfX^{n}_{\leq \tilde\eps},T)\mbox{ is not a forest}))\\
&\leq & \limsup_{n\rightarrow\infty}\PP(\supdiam(\Coal_{\delta^{n}}(\bfX^{n}_{\eps_m},T))>\alpha)\\
&&+\limsup_{n\rightarrow\infty}\PP(\MG(\bfX^{n}_{\leq \tilde\eps},T)\mbox{ is not a forest}))
\end{eqnarray*}
Then, using Lemma \ref{lemm:forest} and the hypothesis on $\supdiam(\Coal_{\delta^{n}}(\bfX^{n}_{\eps},T))$ one sees that the right-hand side above goes to zero when we make $m$ and then $\tilde\eps$ go to zero. This ends the proof of \eqref{eq:supdiamlimite}.
\end{proof}

We are now in position to prove  Theorem~\ref{theo:AlmostFellercoal}.

\begin{proof}{\it (of Theorem~\ref{theo:AlmostFellercoal})}

The fact that the trajectories stay in $\calS$ has been proved in Lemma~\ref{lemm:calSstable}. The strong Markov property follows from the commutation property for coalescence, cf. Remark~\ref{rem:deltagluing} $(ii)$. So to prove $(i)$, it remains to prove that the trajectories are càdlàg (almost surely). We shall in fact prove this in the course of proving point $(ii)$.

Let $t^n\xrightarrow[n\rightarrow\infty]{}t$ and let $T=\sup_n t^n$. Let us fix $\eps\in]0,1[$ and let $(x^n)_{n\in\NN\cup\{\infty\}}:=(\sizes(\bfX^{n}))_{n\in\NN\cup\{\infty\}}$. We know that $x^n$ converges in distribution to $x^\infty$. Using Skorokhod's representation theorem, Corollary~\ref{coro:picture_structure} and \eqref{eq:l2compact}, we obtain that there exists $K(\eps)\in]0,+\infty[$, $\eps_1\in ]0,\eps[$ and  $\eps_2\in]0,\eps_1[$ such that for every $n\in\NNbar$, with probability larger than $1-\eps$ the event $\calA_n$ holds, where $\calA_n$ is the event that points \eqref{coro:structureheart}, \eqref{coro:nocycle} and \eqref{coro:S}  of Corollary~\ref{coro:picture_structure} hold for any $t\in[0,T]$ and $S(x^n,T)\leq K(\eps)$.

Let $\delta^\infty:=0$. On this event $\calA_n$, the Gromov-Hausdorff-Prokhorov distance\footnote{In fact, here we could talk simply of Hausdorff-Prokhorov distance since there is a trivial embedding of one measured semi-metric space into the other.} between a significant component  of $\Coal_{\delta^n}(\bfX^n,t)$ (at any time $t\leq T$) and its heart is at most $\alpha:=\delta^n+\supdiam(\Coal_{\delta^n}(\bfX^n_{\leq \eps_1},T))+\eps_1$. Let $\sigma$ be the function from $\comp(\Coal_{\delta^n}(\bfX^n,t)_{>\eps+\alpha})$ to $\comp(\Coal_{\delta^n}(\bfX^n_{>\eps_2},t))$ which maps a component to its heart, and let $\sigma'$ denote the function from $\comp(\Coal_{\delta^n}(\bfX^n_{>\eps_2},t))_{>\eps+\alpha}$ to $\comp(\Coal_{\delta^n}(\bfX^n,t))$ which maps a component to the (unique, on $\calA_n$) component of $\comp(\Coal_{\delta^n}(\bfX^n,t))$ which contains it. These functions satisfy the hypotheses of Lemma~\ref{lemm:Ldeuxappli}, with $\eps$ replaced by $\eps+\alpha$. This shows that on $\calA_n$, we have for every time $t\leq T$ and every $\eps'_2\leq \eps_2$:
\begin{eqnarray*}
&&L_{GHP}(\Coal_{\delta^n}(\bfX^n,t),\Coal_{\delta^n}(\bfX^n_{> \eps'_2},t))\\
&\leq &8\alpha\frac{S(x^n,t)}{\eps^2}+16(\eps+\alpha)\;,\\
&\leq & 17(\delta^n+\supdiam(\Coal_{\delta^n}(\bfX^n_{\leq \eps_1},T))+\eps_1)\frac{8K(\eps)}{\eps^2}+16\eps\;.
\end{eqnarray*} 
Now, recall from Lemma~\ref{lemm:calSstable} that
$$\PP(\supdiam(\Coal_{\delta^{\infty}}(\bfX^{\infty}_{\leq \eps},T))>\alpha)\xrightarrow[\eps\rightarrow 0]{} 0\;.$$ 
Thus, using the hypothesis on $\supdiam(\Coal_{\delta^{\infty}}(\bfX^{\infty}_{\leq \eps},T))$, one may choose $\eps_1$ small enough (and thus $\eps_2$ small enough) to get that for every $n$ large enough (possibly infinite), with probability larger than $1-2\eps$, we have for every time $t\leq T$ and every $\eps'_2\leq \eps_2$:
$$L_{GHP}(\Coal_{\delta^n}(\bfX^n,t),\Coal_{\delta^n}(\bfX^n_{> \eps'_2},t))\leq 40 \eps\;.$$
Furthermore, since \eqref{coro:S}  of Corollary~\ref{coro:picture_structure} holds on $\calA_n$,
\begin{eqnarray*}
&&\|\sizes(\Coal_{\delta^{n}}(\bfX^{n},t))-\sizes(\Coal_{\delta^{n}}(\bfX^{n}_{> \eps_2},t))\|_2^2\\ &\leq & S(x^{n},t)-S(x^{n}_{>\eps_2},t)\leq \eps\;,
\end{eqnarray*}
where the first inequality comes from Lemma~\ref{lemm:Aldous17}. This shows that
\footnote{At this stage, since we did not proved yet the càdlàg property, it is not guaranteed that the event in the probability is measurable, due to the uncountable supremum. This is why we use the outer measure $\PP^*$. However, once the càdlàg property is proved, one may remove the star.}
\begin{equation}
\label{eq:finiteapprox}
\lim_{\eps_2\rightarrow 0}\lim_{N\rightarrow\infty}\sup_{\substack{n\geq N\\ n\in\NNbar}}\PP^{*}(\sup_{t\leq T}L_{2,GHP}(\Coal_{\delta^{n}}(\bfX^{n},t),\Coal_{\delta^{n}}(\bfX^{n}_{> \eps_2},t))>\eps)=0\;.
\end{equation}

Let us prove that the trajectories of $\Coal(\bfX^\infty,\cdot)$ are almost surely càdlàg if  $\bfX^\infty$ belongs to $\calS$. Let $\bfY^n:=\bfX^\infty_{>\frac{1}{n}}$. Notice that $(\Coal_0(\bfY^n,t))_{t\geq 0}$ is càdlàg: it is right-continuous and piece-wise constant, with a finite number of jumps. Then, equation~\eqref{eq:finiteapprox} applied with $\bfX^n=\bfX^\infty$ shows that the hypotheses of Lemma~\ref{lemm:cvcadlag} are satisfied with $\omega^n=\Coal(\bfY^n,\cdot)$ and $\omega^\infty=\Coal(\bfX^\infty,\cdot)$. This shows that the trajectories of $\Coal(\bfX^\infty,\cdot)$ are almost surely càdlàg, thus finishing to prove point $(i)$ of the theorem.

Now, let $(\alpha_p)_{p\geq 0}$ be a decreasing sequence of positive numbers going to zero such that:
$$\forall p\in \NN,\;\PP[\alpha_p\in x^{\infty}]=0\;.$$
For any $p$, $\bfX^{n}_{> \alpha_p}$ converges to $\bfX^{\infty}_{>\alpha_p}$ in distribution for $L_{1,GHP}$. Proposition~\ref{prop:FellerN1} implies that $(\Coal_{\delta^{n}}(\bfX^{n}_{> \alpha_p},t))_{t\leq T}$ converges to $(\Coal_{0}(\bfX^{\infty}_{> \alpha_p},t))_{t\leq T}$ for the Skorokhod topology  associated to $L_{2,GHP}$.  Together with \eqref{eq:finiteapprox}, Lemma~\ref{lemm:compactcv} and inequality~\eqref{eq:dSkdc}, this proves $(ii)$.

Furthermore, Proposition~\ref{prop:FellerN1} implies that $\Coal_{\delta^{n}}(\bfX^{n}_{> \alpha_p},t^{n})$ converges in distribution to $\Coal_{0}(\bfX^{\infty}_{>\alpha_p},t)$ as $n$ goes to infinity for $L_{1,GHP}$, and thus for $L_{2,GHP}$. Together with \eqref{eq:finiteapprox} we obtain that $\Coal_{\delta^{n}}(\bfX^{n},t^{n})$ converges to $\Coal_{0}(\bfX^{\infty},t)$ in distribution for $L_{2,GHP}$. This proves $(iii)$. 
\end{proof}
\begin{rem}
Notice that if $\delta_n>0$ and $\bfX^n\in\calN_2\setminus\calN_1$, $\Coal_{\delta_n}(\bfX^n,t)$ is not in $\calN_2$ for $t>0$ since the components are not totally bounded. Thus, the terms ``convergence in distribution for $L_{2,GHP}$''in Theorem~\ref{theo:AlmostFellercoal} should be understood in a larger space, where components are allowed not to be totally bounded. 

However, we do not insist on this because when $\delta^n>0$, we shall always use Theorem~\ref{theo:AlmostFellercoal}  with $\bfX^n\in\calN_1$ for any $n\in\NN$, in which case $\Coal_{\delta_n}(\bfX^n,t)$ is in $\calN_1$ for any $t$.
\end{rem}

Finally, we finish this section by exhibiting a sufficient condition for \eqref{eq:supdiamunif} which will be useful in section~\ref{subsec:CoalER}. In words, it says that when the diameter of a component of $\Coal_{\delta_n}(\bfX^n,T)$ goes to zero when its size goes to zero, uniformly in $n$, then one may restrict in \eqref{eq:supdiamunif} to components which are attached to a significant component by at most one edge of the multigraph $\MG(\bfX^n,T)$. In Lemma~\ref{lem:supdiamunif}, we consider the obvious coupling between $\Coal_{\delta_n}(\bfX^n,T)$, $\Coal_{\delta_n}(\bfX^n_{\leq \eps_1},T)$, $\MG(\bfX^n_{>\eps_1},T)$ and $\MG(\bfX^n,T)$.

\begin{lemm}
  \label{lem:supdiamunif}
Let $\bfX^{n}=(X^{n},d^{n},\mu^{n})$, $n\geq 0$ be a sequence of random variables in $\calN_2$ and $(\delta^{n})_{n\geq 0}$ a sequence of non-negative real numbers.   For $T$, $\eps$ and $\eps_1>0$, let $\calC_n(\eps,\eps_1,T)$ denote the set of components $m$ of $\Coal_{\delta_n}(\bfX^n_{\leq \eps_1},T)$ which are included in a component of $\Coal_{\delta_n}(\bfX^n,T)$ of size at least $\eps$ and such that there is exactly one edge between $m$ and $\MG(\bfX^n_{>\eps_1},T)$ in $\MG(\bfX^n,T)$. Let $d_n(\eps,\eps_1,T)$ denote the supremum of the diameters of elements of $\calC_n(\eps,\eps_1,T)$. Suppose that: 
\begin{enumerate}[(i)]
\item $(\sizes(\bfX^{n}))_{n\in\NN}$ converges in distribution, in $\ell^2$ to a random variable $x^{\infty}$ as $n$ goes to infinity.
\item for any $\alpha>0$
  \[\limsup_{\eps\rightarrow 0}\limsup_{n\rightarrow \infty}\PP(\supdiam(\Coal_{\delta_n}(\bfX^n,T)_{<\eps})>\alpha )=0\]
\item  for any $\alpha>0$ and $\eps>0$,
  \[\limsup_{\eps_1\rightarrow 0}\limsup_{n\rightarrow \infty}\PP(d_n(\eps,\eps_1,T)>\alpha )=0\]
\end{enumerate}
Then,
   \[\limsup_{n\in\NN} \PP(\supdiam(\Coal_{\delta^{n}}(\bfX^{n}_{\leq \eps},T))>\alpha)\xrightarrow[\eps\rightarrow 0]{} 0\;.\]
\end{lemm}
\begin{proof}
  Let $A_n(\eps,\eps_1)$ denote the event that there is a component of $\Coal_{\delta_n}(\bfX^n_{\leq \eps_1},T)$ which is included in a component of $\Coal_{\delta_n}(\bfX^n,T)$ of size at least $\eps$ and such that there is more than one edge between $m$ and $\MG(\bfX^n_{>\eps_1},T)$ in $\MG(\bfX^n,T)$. Using the Skorokhod representation theorem, Corollary~\ref{coro:picture_structure} and \eqref{eq:l2compact}, one sees that for any $\eps>0$, 
  \[\limsup_{n\rightarrow \infty}\PP(A_n(\eps,\eps_1))\xrightarrow[\eps_1\rightarrow 0]{}0\;.\]
  Recall also Lemma~\ref{lemm:forest}. Letting $B_n(\eps_1)$ denote the event that there is more than one edge between a connected component of $\Coal_{\delta^{n}}(\bfX^n_{\leq \eps_1},T)$ and $\Coal_{\delta^{n}}(\bfX^n_{>\eps_1},T)$, Lemma~\ref{lemm:forest} shows that
  \[\limsup_{n\rightarrow \infty}\PP(B_n(\eps_1))\xrightarrow[\eps_1\rightarrow 0]{}0\;.\]
  On $B_n(\eps_1)^c$, the diameter of a component of $\Coal_{\delta^{n}}(\bfX^n_{\leq \eps_1},T)$ is at most the diameter of the component of $\Coal_{\delta^{n}}(\bfX^n,T)$ which contains it. Thus, for any $\eps_1$ and $\eps$, 
  \begin{eqnarray*}
    &&\PP(\supdiam(\Coal_{\delta^{n}}(\bfX^{n}_{\leq \eps_1},T))>\alpha)\\
    &\leq & \PP(d_n(\eps,\eps_1,T)>\alpha)+\PP(\supdiam(\Coal_{\delta_n}(\bfX^n,T)_{<\eps})>\alpha )\\
    &&+\PP(A_n(\eps,\eps_1))+\PP(B_n(\eps_1))\;,
  \end{eqnarray*}
  this gives the result.
\end{proof}

\subsection{Convergence of the coalescent on Erd\H{o}s-Rényi random graphs}
\label{subsec:CoalER}
In this section we prove Theorem~\ref{theo:cvcoalGnp}. Recall that $\overline{\calG}_{n,\lambda}$ is the element of $\calN_2^{graph}$ obtained from  $\calG(n,p(\lambda,n))$ by assigning to each edge a length $n^{-1/3}$ and to each vertex a mass $n^{-2/3}$. We know by Theorem~\ref{theo:cvGnp} that $\overline{\calG}_{n,\lambda}$ converges in distribution (for $L_{2,GHP}$) to $\Glambda$. In view of Theorem~\ref{theo:AlmostFellercoal}, it is sufficient to prove that for any $T$ and $\alpha>0$: 
\begin{equation}
  \label{eq:supdiamcalG_n}
  \limsup_{n\rightarrow \infty} \PP(\supdiam(\Coal_{n^{-1/3}}((\overline{\calG}_{n,\lambda})_{\leq \eps},T))>\alpha)\xrightarrow[\eps\rightarrow 0]{} 0\;.
  \end{equation}
To this end, we shall use Lemma~\ref{lem:supdiamunif}. The notion of depth-first exploration process on a finite graph $G=(V,E)$, as defined in \cite[sections~1 and 2]{AdBrGolimit}, will be useful. This depth-first exploration process defines an order $\sigma$ on $V$ (a bijection from $V$ to $\{0,\ldots,n-1\}$ with $n=|V|$), a height process $h$ (from $\{1,\ldots,|V|\}$ to $\NN$) and, for each connected component $C$, a rooted tree $(\rho_C,\calT_C)$ such that $\rho_C$ is the first vertex visited in $C$ and $\calT_C$ spans $C$. Furthermore, denoting by $d_{\calT_C}$ the graph metric on $\calT_C$, one has, for any connected component $C$ of $G$:
\[\forall i\in C,\;h(\sigma(i))=d_{\calT_C}(\rho_C,i)\;.\]
We shall need the following lemma.
\begin{lemm}
  \label{lemm:DFSdiam}
  Let $G$ be a finite graph with vertex set $V$. Let $h$ be the height process associated to the depth-first exploration process on $G$ and denote by $\sigma$ the order induced by the depth-first exploration on $V$. Let $I$ denote a subset of $V$ such that the subgraph induced by $I$ in $G$ is connected and such that, denoting by $C$ the connected component of $G$ containing $I$, either $C=I$ or there is exactly one edge connecting $I$ to $C\setminus I$. Then, 
  \[\diam(I)\leq 2+3\max_{i,j\in C\,:\,|\sigma(i)-\sigma(j)|\leq |I|}|h(\sigma(i))-h(\sigma(j))|\;,\]
  where the diameter is computed either for the distance $d_{\calT_C}$ or for the graph distance on $G$.
\end{lemm}
\begin{dem}
  If $C=I$, the result is trivial, so let us suppose that $I\neq C$ and let $x$ denote the unique vertex of $I$ connected to $C\setminus I$. We consider two cases.
  \begin{enumerate}
  \item Suppose that $\rho_C\in C\setminus I$. In this case, the vertices of $I$ are explored consecutively, i.e $\sigma(I)$ is an interval of length $|I|$, and the first vertex explored in $I$ is $x$. Thus, for any $z\in I$,
\[d_{\calT_C}(z,\rho_C)=d_{\calT_C}(z,x)+d_{\calT_C}(x,\rho_C)\;.\]
Thus, if $z$ and $z'$ denote two vertices of $I$,
\begin{eqnarray*}
  d_{\calT_C}(z,z')&\leq &d_{\calT_C}(z,x)+d_{\calT_C}(z',x)\;,\\
  &=&d_{\calT_C}(z,\rho_C)-d_{\calT_C}(x,\rho_C)+d_{\calT_C}(z',\rho_C)-d_{\calT_C}(x,\rho_C)\;,\\
  &=&h(\sigma(z))-h(\sigma(x))+h(\sigma(z'))-h(\sigma(x))\;,\\
  &\leq &2\max_{i,j\in C\,:\,|\sigma(i)-\sigma(j)|\leq |I|}|h(\sigma(i))-h(\sigma(j))|
\end{eqnarray*}
since $\sigma(I)$ is an interval of length $|I|$.
\item Suppose that $\rho_C\in I$. Then, the vertices of $C\setminus I$  are explored consecutively. Let $J$ denote the subset of $I$ composed of the vertices of $I$ which are explored before those of $C\setminus I$. $J$ contains $\rho_C$ and $x$ (which might be the same vertex). Notice that the vertices of $J$ are explored consecutively (i.e $\sigma(J)$ is an interval of length $|J|$), and those of $I\setminus J$ also (i.e $\sigma(I\setminus J)$ is an interval of length $|I|-|J|$). Notice also that $J$ forms a subtree of $\calT_C$. When $z$ and $z'$ belong to $C$, we shall denote by $[z,z']$ the unique path from $z$ to $z'$ in $\calT_C$ and by $T(z)$ the subtree above $z$, i.e the set of vertices $u$ such that $z\in[\rho_C,u]$. Notice that if $z\in I\setminus J$, then $T(z)$ is included in $I\setminus J$, and that if $u\in T(z)$,
  \begin{equation}
    \label{eq:subtree}
  d_{\calT_C}(z,u)=h(\sigma(u))-h(\sigma(z))\;.
\end{equation}
  Now, let $z$ and $z'$ belong to $I$ and consider the following cases.
  \begin{enumerate}[(a)]
  \item If $z$ and $z'$ belong to $J$, then
\begin{eqnarray*}
  d_{\calT_C}(z,z')&\leq &d_{\calT_C}(z,\rho_C)+d_{\calT_C}(z',\rho_C)\;,\\
  &=&h(\sigma(z))-h(\sigma(\rho_C))+h(\sigma(z'))-h(\sigma(\rho_C))\;,\\
  &\leq & 2\max_{i,j\in C\,:\,|\sigma(i)-\sigma(j)|\leq |J|}|h(\sigma(i))-h(\sigma(j))|
\end{eqnarray*}
since $\rho_C$ belongs to $J$.
\item If $z\in J$ and $z'\in I\setminus J$, let $y$ denote the vertex of $[\rho_C,z']\cap J$ closest (for $d_{\calT_C}$) to $z'$ and $y'$  the vertex of $[\rho_C,z']\cap I\setminus J$ closest to $\rho_C$. Since $J$ is explored consecutively and $\calT_C$ is the depth-search tree, $y$ and $y'$ are neighbours in $\calT_C$, $[\rho_c,y]$ is included in $J$ and $z'\in T(y')$. We already proved in point $(a)$ above that 
\[d_{\calT_C}(z,y)\leq 2\max_{i,j\in C\,:\,|\sigma(i)-\sigma(j)|\leq |J|}|h(\sigma(i))-h(\sigma(j))|\;.\]
  Thus, using \eqref{eq:subtree}:
\begin{eqnarray*}
  d_{\calT_C}(z,z')& \leq &d_{\calT_C}(z,y)+d_{\calT_C}(y,y')+d_{\calT_C}(y',z')\;,\\
  &\leq &2\max_{i,j\in C\,:\,|\sigma(i)-\sigma(j)|\leq |J|}|h(\sigma(i))-h(\sigma(j))|\\
  &&+1+h(\sigma(z'))-h(\sigma(y'))\;,\\
  &\leq &2\max_{i,j\in C\,:\,|\sigma(i)-\sigma(j)|\leq |J|}|h(\sigma(i))-h(\sigma(j))|\\
  &&+1+\max_{i,j\in C\,:\,|\sigma(i)-\sigma(j)|\leq |I|-|J|}|h(\sigma(i))-h(\sigma(j))|\;.
\end{eqnarray*}
\item If $z$ and $z'$ belong to $I\setminus J$, the arguments are similar: one finds $y$ and $y'$ in $I\setminus J$, $t$ and $t'$ in $J$ such that $z\in T(y)$, $z'\in T(y')$, $t$ is a neighbour of $y$ and $t'$ is a neighbour of $y'$. Notice that we already proved that
\[d_{\calT_C}(t,t')\leq 2\max_{i,j\in C\,:\,|\sigma(i)-\sigma(j)|\leq |J|}|h(\sigma(i))-h(\sigma(j))|\;.\]
  Then, using \eqref{eq:subtree}
\begin{eqnarray*}
  d_{\calT_C}(z,z')& \leq &d_{\calT_C}(z,y)+d_{\calT_C}(y,t)+d_{\calT_C}(t,t')\\
  &&+d_{\calT_C}(t',y')+d_{\calT_C}(y',z')\;,\\
  &\leq &h(\sigma(z))-h(\sigma(y))+1+d_{\calT_C}(t,t')+1\\
  &&+h(\sigma(z'))-h(\sigma(y'))\;,\\
  &\leq & 2+3\max_{i,j\in C\,:\,|\sigma(i)-\sigma(j)|\leq |I|}|h(\sigma(i))-h(\sigma(j))|\;.\end{eqnarray*}
  \end{enumerate}
  \end{enumerate}
  
\end{dem}



Let us denote by $h_{n,\lambda}$ the height process associated to the depth-first exploration process on $\overline{\calG}_{n,\lambda}$, and let $\overline h_{n,\lambda}$ be its rescaled version:
$$\overline{h}_{n,\lambda}(x):=\frac{1}{n^{1/3}}h_{n,\lambda}(xn^{2/3})\;.$$
Now let us consider the depth-first exploration process of $\Coal_{n^{-1/3}}(\overline{\calG}_{n,\lambda},T)$. When $\calP_T^+$ has intensity $\gamma$, $N^+(\calG(n,p),\calP_T^+)$ is equal in distribution to $\calG(n,p')$ with:
$$p'=p+(1-p)(1-e^{-\gamma T})$$
When $\gamma=n^{-4/3}$, $p'=p(\lambda'_n,n)$ with 
$$\lambda'_n\xrightarrow[n\rightarrow\infty]{}\lambda+T$$
and $\Coal_{n^{-1/3}}(\overline{\calG}_{n,\lambda},T)$ is equal, in distribution, to $\overline{\calG}_{n,\lambda'_n}$. The difference between $\lambda'_n$ and $\lambda_n+T$ is unimportant for us (for instance using Lemma~\ref{lemm:forest}, $\supdiam(\Coal_{n^{-1/3}}((\overline{\calG}_{n,\lambda})_{\leq \eps},T))$ is essentially nondecreasing in $T$), so we shall continue as if $\lambda'_n=\lambda+T$. Under this approximation, the rescaled version of the height process associated to the exploration of $\Coal_{n^{-1/3}}(\overline{\calG}_{n,\lambda},T)$ has the same distribution as $\overline{h}^{n,\lambda+T}$, and we shall adopt the same notation to keep things simple.

Now, we shall use Lemma~\ref{lem:supdiamunif} to prove \eqref{eq:supdiamcalG_n}. Let us use the notations of Lemma~\ref{lem:supdiamunif}, with $\bfX^n=\overline{\calG}_{n,\lambda}$ and $\delta^n=n^{-1/3}$. Hypothesis $(i)$ of Lemma~\ref{lem:supdiamunif} is satisfied, so we only need to prove hypotheses $(ii)$ and $(iii)$. Lemma~\ref{lemm:DFSdiam} ensures that
$$d_n(\eps,\eps_1)\leq  \frac{2}{n^{1/3}}+3\sup_{\substack{x,y\in\RR^+\\|x-y|\leq \eps_1}}|\overline{h}^{n,\lambda+T}(x)-\overline{h}^{n,\lambda+T}(y)|$$
where the supremum is restricted to pairs $(x,y)$ such that $x$ and $y$ belong to a same excursion of $\overline{h}^{n,\lambda+T}$ above zero, of length at least $\eps$. Thus, we see that hypotheses $(ii)$ and $(iii)$ of Lemma~\ref{lem:supdiamunif} will be established if we can prove that for any $\alpha>0$,
\begin{equation}
\label{eq:oscillation}
\limsup_{n\rightarrow\infty}\PP\left(\sup_{\substack{x,y\in\RR^+\\|x-y|\leq \eps_1}}|\overline{h}^{n,\lambda+T}(x)-\overline{h}^{n,\lambda+T}(y)|>\alpha\right)\xrightarrow[\eps_1\rightarrow 0]{} 0\;,
\end{equation}
with the supremum restricted to pairs $(x,y)$ such that $x$ and $y$ belong to the same excursion of $\overline{h}^{n,\lambda+T}$ above zero.

Let $B^\lambda$ be a Brownian motion with quadratic drift, defined by $B^\lambda_t=B_t+\lambda t-\frac{t^2}{2}$ with $B$ a standard Brownian motion. Let $W^\lambda$ be $B^\lambda$ reflected above its current minimum:
$$W^\lambda_t:=B^\lambda_t-\min_{0\leq s\leq t}B^\lambda_s\;.$$ 
Then \eqref{eq:oscillation} is a consequence of the fact that $\overline{h}^{n,\lambda+T}$ converges in distribution to  $2W^\lambda$ for the sup norm on $\RR^+$. It seems however that this convergence is not written in the literature, so in order to use only available sources, one may rely on the work done in \cite{AdBrGolimit} as follows. One may  separate the analysis of the supremum on the $N$ largest excursions and on the others. Let $B_{N,n}(\eps)$ be the event that the maximal height of the $i$-th largest component in $\overline{\calG}_{n,\lambda+T}$ exceeds $\eps$ for some $i> N$. The equation p.402 below equation (24) in \cite{AdBrGolimit} shows that for any $\eps>0$,
\begin{equation}
\label{eq:maxheight}
\lim_{N\rightarrow\infty}\limsup_{n\rightarrow\infty}\PP(B_{N,n}(\eps))= 0\;.
\end{equation}
Then, for a fixed $N$, one may argue as in the proof of \cite[Theorem~24, p.398]{AdBrGolimit}: conditionally on the sizes, the rescaled height processes associated to those components are independent and each one converges in distribution (for the uniform topology) to a continuous excursion (a tilted Brownian excursion). Together with the convergence of the sizes and Skorokhod's representation theorem, this proves that the $N$ largest excursions of $\overline{h}^{n,\lambda+T}$ converges as a vector in $C([0,+\infty[)^N$ to a random vector of continuous functions with bounded support. This implies that \eqref{eq:oscillation} holds when the supremum is restricted to pairs $(x,y)$ such that $x$ and $y$ belong to one of the $N$ largest components of $\overline{h}^{n,\lambda+T}$. Together with \eqref{eq:maxheight}, this shows \eqref{eq:oscillation} and this ends the proof of Theorem~\ref{theo:cvcoalGnp}.


\section{Proofs of the results for fragmentation}
\label{sec:frag}

The main goal of this section is to prove the Feller property for fragmentation on $\calN_2^{graph}$, Theorem~\ref{theo:FellerGraphs} and to apply it to prove Theorem~\ref{theo:cvfragGnp}. It is very close to the work performed in \cite{AdBrGoMi2017}, which proves a continuity result for a fragmentation restricted to the core of a graph (and stopped when you get a tree). The main difference is that we want in addition to perform fragmentation on the tree part of the graphs. Another technical difference will be detailed at the beginning of section \ref{subsec:FellerGraphs}. Unfortunately, those differences force us to make substantial modifications to the arguments of \cite{AdBrGoMi2017}.
 
\subsection{Notation}
\label{subsec:notationsgraphs}
We need to introduce a few more definitions to deal with fragmentation of $\RR$-graphs. For more details, we refer to \cite{AdBrGoMi2017}.

A \emph{multigraph with edge-lengths} is a triple $(V,E,(\ell(e))_{e\in E})$ where $(V,E)$ is a finite connected multigraph and for every $e\in E$, $\ell(e)$ is a strictly positive number. One may associate to such a multigraph with edge-lengths a compact $\RR$-graph with a finite number of leaves by performing on $V$ (seen as a metric space as the disjoint union of its elements) the $\ell(e)$-gluing along $e$ for each edge $e\in E$. 

\emph{Until the end of the article, we shall say that an $\RR$-graph is finite if it is compact and has a finite number of leaves.} Equivalently, it can be associated to a multigraph with edge-lengths as above. This terminology, applied to trees, comes from \cite{EvansPitmanWinter06}.

Let $G$ be an $\RR$-graph. When there is only one geodesic between $x$ and $y$ in $G$, we denote by $[x,y]$ its image. Recall the notion of the core of $G$ defined in section~\ref{subsec:Rgraphs}. If $S$ is a closed connected subset of $G$ containing $\core(G)$, then for any $x\in G$, there is a unique shortest path $\gamma_x$ going from $x$ to $S$. We denote by $p_S(x)$ the unique point belonging to $\gamma_x\cap S$. When $G$ is not a tree and $S=\core(G)$, we let $\alpha_G(x):=p_S(x)$. 

For any $\eta>0$ and any $\RR$-graph $G$ which is not a tree, let
$$R_\eta(G):=\core(G)\cup\{x\in G\mbox{ s.t. }\exists y:\;x\in [y,\alpha_G(y)] \text{ and }d(y,x)\geq \eta\}\;.$$
When $(T,\rho)$ is a rooted $\RR$-tree, we let $R_\eta(T):=R_\eta(T,\rho)$ be defined as above, with $\alpha_G(y)$ replaced by the root $\rho$ and $\core(G)$ replaced by $\{\rho\}$. Thus the definition of $R_\eta(G)$ extends the definition of $R_\eta(T)$ for a rooted $\RR$-tree $(T,\rho)$ in \cite{EvansPitmanWinter06}. Notably, \cite[Lemma~2.6~(i)]{EvansPitmanWinter06} shows that for any $\eta>0$, $R_\eta(G)$ is a finite $\RR$-graph. For a (non-rooted) meaured tree $(T,\mu)$, with $\mu$ a positive finite measure, we let $R_\eta(T)$ denote $R_\eta(T,\rho)$ where $\rho$ is a random root, sampled according to $\frac{\mu(\cdot)}{\mu(T)}$. Finally, if $G$ belongs to $ \calN_2^{graph}$, we let
\[R_\eta(G):=\bigcup_{m\in\comp(G)}R_\eta(m)\;,\]
where the random roots of components which are trees are sampled independently.

The \emph{$\eps$-enlargement} of a correspondence $\calR\in C(X,X')$ is defined as:
$$\calR^\eps:=\{(x,x')\in X\times X':\exists (y,y')\in\calR, d(x,y)\lor d(x',y')\leq \eps\}\;.$$
It is a correspondence containing $\calR$ with distortion at most $\dis(\calR)+4\eps$.

If $\calR\in C(X,X')$, two Borel subsets $A\subset X$ and $B\subset X'$ are said to be in correspondence through $\calR$ if $\calR\cap (A\times B)\in C(A,B)$. 

Let $\eps>0$. If $X$ and $X'$ are $\RR$-graphs with surplus at least $2$, an $\eps$-overlay is a correspondence $\calR\in C(X,X')$ with distortion less than $\eps$ and such that there exists a multigraph isomorphism $\chi$ between the kernels $\ker(X)$ and $\ker(X')$ satisfying:
\begin{enumerate}
\item $\forall v\in k(X), (v,\chi(v))\in\calR$,
\item For every $e\in e(X)$, $e$ and $\chi(e)$ are in correspondence through $\calR$ and $|\ell_X(e)-\ell_{X'}(\chi(e))|\leq \eps$.
\end{enumerate}
If $X$ and $X'$ have surplus one, an $\eps$-overlay is a correspondence with distortion less than $\eps$ such that the unique cycles of $X$ and $X'$ are in correspondence and the difference of their lengths is at most $\eps$. If $X$ and $X'$ are trees, an $\eps$-overlay is simply a correspondence with distortion less than $\eps$.

We let $\calN_2^{tree}$ be the set of elements $\bfX\in\calN_2^{graph}$ whose components are trees.

\subsection{Reduction to finite $\RR$-graphs}
The following lemmas will be useful to reduce the proof of the Feller property to finite $\RR$-graphs, notably to adapt the arguments of \cite{EvansPitmanWinter06}.
\begin{lemm}
\label{lemm:approxReta}
Let $\eta\in (0,1]$ and $T>0$. Let $\bfG$ belong to $\calN_2^{graph}$. Let $S$ be a closed connected subset of $\bfG$ such that $R_\eta(\bfG)\subset S\subset \bfG$. Suppose that for each component $H$ of $\bfG$, $S\cap H$ is a connected $\RR$-graph. Let $\bfS:=(S,d|_{S\times S},p_S\sharp\mu)$. Then, with probability at least $1-T\eta^{1/7}$, for any $t\in[0,T]$, under the obvious coupling,
$$L_{2,GHP}^{surplus}(\Frag(\bfG,t),\Frag(\bfS,t))\leq 34\eta^{1/7}(1+\sum_{H\in\comp(\bfG)}\mu(H)^2)^2\;.$$
\end{lemm}
\begin{proof}
Let $\calP$ be a poisson random set of intensity measure $\ell_G\otimes\leb_\RR^+$ on $G\times \RR^+$ and let us use it to perform the fragmentation on $S$ and $G$. Define:
$$G^\eta_t:=\{x\in G\setminus S\mbox{ s.t. }\exists y\in \calP_t\cap (G\setminus S) \cap [x,\alpha_G(x)]\}\;.$$
Notice that a component $m$ of $\Frag(\bfS,t)$ is endowed with the distance $d|_{m\times m}$ and the measure $(p_S\sharp\mu)|_m$, while a component $m$ of $\Frag(\bfG,t)$ is endowed with the distance $d|_{m\times m}$ and the measure $\mu|_{m}$.

If $m$ is a component of $\Frag(\bfG,\calP_t)$ such that $m\cap S=\emptyset$ then $m\subset G^\eta_t$. Notably, if $t\in[0,T]$, $H\in\comp(\bfG)$, $m$ is a component of $\Frag(\bfG,\calP_t)$ included in $H$ and $\mu(m)>\mu(G^\eta_T\cap H)$, then $m$ must intersect $S$. Furthermore, if $m\cap S\not=\emptyset$ then $m\cap S$ is a component of $\Frag(\bfS,\calP_t)$. 

Let $H\in\comp(\bfG)$. For any component $m$ of $\Frag(H,\calP_t)$ such that $m\cap S\not=\emptyset$, we claim that
\begin{equation}
  \label{eq:dGHPsigmaGmiss}
  d_{GHP}(m,m\cap S)\leq \eta\lor \mu(G^\eta_T\cap H)
\end{equation}
Indeed, let $\calR:=\{(x,p_S(x))\;:\;x\in m\}$, which has distortion at most $2\eta$, and define $\pi:=((Id,p_S)\sharp\mu)|_{m\times m\cap S}$. Then, $\pi(\calR^c)=0$ and
\begin{eqnarray*}
  D(\pi;\mu|_{m},(p_S\sharp\mu)|_m)&=&\sup_{A\in\calB(m\cap S)}\mu(p_S^{-1}(A)\setminus m)\\
  &=&\mu(p_S^{-1}(m)\setminus m)\\
  &\leq & \mu(G^\eta_t\cap H)\\
&\leq & \mu(G^\eta_T\cap H)\;.
\end{eqnarray*}
This shows \eqref{eq:dGHPsigmaGmiss}. Furthermore,
\begin{eqnarray}
&&\nonumber \|\sizes(\Frag(\bfG,t))-\sizes(\Frag(\bfS,t))\|_2^2\\
&\leq &\sum_{\substack{m\in\Frag(\bfG,t)\\m\cap S\not=\emptyset}}\mu(p_S^{-1}(m)\setminus m)^2+\sum_{\substack{m\in\Frag(\bfG,t)\\m\cap S=\emptyset}}\mu(m)^2\\
\nonumber  &\leq &2\sum_{H\in\comp(\bfG)}\mu(G^\eta_t\cap H)^2\\
\label{eq:L2reductionFrag} &\leq &2\sum_{H\in\comp(\bfG)}\mu(G^\eta_T\cap H)^2\;.
\end{eqnarray}
Using Fubini's theorem,
$$\EE[\mu(G^\eta_T\cap H)^2]\leq\mu(H)^2(1-e^{-\eta T})\leq \mu(H)^2\eta T\;.$$
Thus,
$$\PP\left(\sum_{H\in\comp(\bfG)}\mu(G^\eta_T\cap H)^2\geq \eta^{6/7}\sum_{H\in\comp(\bfG)}\mu(H)^2\right)\leq T\eta^{1/7}\;.$$
Now, let us place ourselves on the event 
$$\calE:=\{\sum_{H\in\comp(\bfG)}\mu(G^\eta_T\cap H)^2< \eta^{6/7}\sum_{H\in\comp(\bfG)}\mu(H)^2\}$$
and define $\alpha:=\eta^{3/7}\sqrt{1+\sum_{H\in\comp(\bfG)}\mu(H)^2}$.  Notice that on $\calE$, we have for any $H\in\comp(\bfG)$:
$$\mu(G_T^\eta\cap H)\leq \sqrt{\sum_{H\in\comp(\bfG)}\mu(G^\eta_T\cap H)^2}\leq\alpha\;.$$
Let $\sigma$ assign to each component of $\Frag(\bfS,t)$ the component of $\Frag(\bfG,t)$ which contains it, and let $\sigma'$ assign to a component $m$ of $\comp((\Frag(\bfG,t))_{>\alpha^{1/3}+\alpha})$ the component $m\cap S$ of $\comp(\Frag(\bfS,t))$. From \eqref{eq:dGHPsigmaGmiss} we deduce that for any component $m$ of $\Frag(\bfS,t)$,
$$d_{GHP}(m,\sigma(m))\leq \alpha$$
and notice that $m$ and $\sigma(m)$ have the same surplus. Also, for any component $m'$ of $\Frag(\bfG,t)$,
$$d_{GHP}(m',\sigma'(m'))\leq \alpha\;,$$
and  $m'$ and $\sigma'(m')$ have the same surplus. According to Lemma~\ref{lemm:Ldeuxappli} this shows that on the event $\calE$:
\begin{eqnarray*}
&& L_{GHP}^{surplus}(\Frag(\bfG,t),\Frag(\bfS,t)\\
&\leq& 8\alpha\frac{\sum_{m\in\comp(\Frag(\bfG,t))}\mu(m)^2}{\alpha^{2/3}}+16(\alpha^{1/3}+\alpha)\\&\leq& 16\alpha+\alpha^{1/3}\left(16+8\sum_{H\in\comp(\bfG)}\mu(H)^2\right)\;.
\end{eqnarray*}
And, thanks to \eqref{eq:L2reductionFrag},
$$\|\sizes(\Frag(\bfG,t))-\sizes(\Frag(\bfS,t)\|_2^2\leq \eta^{6/7}\sum_{H\in\comp(\bfG)}\mu(H)^2$$
which shows the result. 
\end{proof}
We shall need a slight variation of the preceding lemma at time zero for rooted trees. When $(\bfX,\rho)$ and $(\bfX',\rho')$ are two rooted m.s-m.s, rooted respectively at $\rho$ and $\rho'$, we define:
$$d_{GHP}^{root}((\bfX,\rho), (\bfX',\rho')) =\inf_{\substack{\pi\in M(X,X')\\ \calR\in C_{\rho,\rho'}(X,X')}}\{D(\pi;\mu,\mu')\lor \frac{1}{2}\dis(\calR)\lor \pi(\calR^c)\}$$
where $C_{\rho,\rho'}(X,X')$ is the set of correspondences between $\bfX$ and $\bfX'$ which contain $(\rho,\rho')$.

Using natural correspondences and couplings, one may see that $R_\eta(T)$ approximates nicely a rooted tree.
\begin{lemm}
  \label{lemm:approxrootedtree}
  Let $\bfT=(T,d,\mu)$ be a measured real tree and $\rho\in T$ a root. Let $\bfR_\eta(T)$ be the measured real tree $R_\eta(T)$ equipped with the measure $p_{R_\eta(T)}\sharp\mu$. Then,
  \[d_{GHP}^{root}((\bfT,\rho), (\bfR_\eta(T),\rho))\leq \eta\;.\]
\end{lemm}
\begin{proof}
  Let $p:=p_{R_\eta(T)}$. Take
  \[\calR:=\{(x,p(x))\in T\times R_\eta(T)\,:\,x\in T\}\]
  which is a correspondence containing $(\rho,\rho)$ of distortion at most $2\eta$. Then take
  \[\pi:=(Id,p)\sharp\mu\]
  i.e $\pi(C)=\mu(\{x\in T\,:\,(x,p(x))\in C\})$,   which verifies
  \[D(\pi;\mu,p\sharp\mu)=0\;.\]
  and $\pi(\calR^c)=0$.
\end{proof}
\subsection{The Feller property for trees}
\label{subsec:FellerTrees}

If $x$ and $y$ belong to a rooted tree $(T,\rho)$, we denote by $[x,y]$ the unique geodesic between $x$ and $y$, and we say that $x\leq y$ if $x$ belongs to $[\rho,y]$. The subtree above $x$ is then defined as
$$\{y\in T\,:\,x\leq y\}\;.$$

The following lemma is a slight extension of \cite[Lemma~6.3]{EvansPitmanWinter06} designed to take measures into account.
\begin{lemm}
\label{lemm:EvansPitmanWinterbis} Let $\bfT=(T,d,\mu)$ be a measured real tree, $\rho\in T$ and $\eps> 0$. There exists $\eta>0$ (depending only on $\eps$), and $\delta > 0$  (depending on $\bfT$, $\rho$ and $\eps$) such that if $\bfT'=(T',d',\mu')$ is a measured  $\RR$-tree rooted at $\rho$ and $d_{GHP}^{root}((\bfT,\rho), (\bfT',\rho')) < \delta$, then there exist finite subtrees $S\subset R_\eta(T)$ and $S'\subset T'$ such that $\rho\in S$, $\rho'\in S'$ and:
\begin{enumerate}[(i)]
\item $\delta_H(S,T) < \eps$ and $\delta_H(S', T') < \eps$,
\item there is a bijective measurable map $\psi: S \rightarrow S'$ that preserves length measure and has distortion at most $\eps$,
\item $\psi(\rho)=\rho'$,
\item the length measure of the set of points $a\in S$ such that $\{b \in S \;:\; \psi(a) \leq b\}\not= \psi(\{b \in S \;:\; a \leq b\})$ (that is, the set of points $a$ such that the subtree above $\psi(a)$ is not the image under $\psi$ of the subtree above $a$) is less than $\eps$.
\item \label{itemnouveau} there is a correspondence $\calR\in C(S,S')$ and a measure $\pi\in M(S,S')$ such that:
  \begin{enumerate}[(a)]
  \item $\forall x\in S$ $(x,\psi(x))\in\calR$
  \item $\pi(\calR^c)\leq \eps$
  \item $D(\pi;p_S\sharp\mu,p_{S'}\sharp\mu')\leq \eps$
  \item $\dis(\calR)\leq 2\eps$.
  \end{enumerate}
\end{enumerate}
\end{lemm}
\begin{proof}
Notice that in \cite[Lemma~6.3]{EvansPitmanWinter06}, $\bfT$ and $\bfT'$ are supposed to be finite trees, but we shall soon be back to this case.
  
Suppose that $\eta>0$ and $d_{GHP}^{root}((\bfT,\rho), (\bfT',\rho')) < \delta$ ($\delta$ and $\eta$ will be chosen small enough later). Define, to lighten notation:
\[\bfT_\eta:=R_\eta(T)\text{ and }\bfT'_\eta:=R_{\eta}(T')\;.\]
Using Lemma~\ref{lemm:approxrootedtree},
\[d_{GHP}^{root}((\bfT_\eta,\rho), (\bfT'_\eta,\rho'))<\delta+2\eta=:\tilde\delta\;.\]
Then, there exists a correspondence $\calR_0\in C(T_\eta,T'_\eta)$ and a measure $\pi_0\in M(T,T')$ such that:
  \begin{enumerate}[(a)]
  \item $(\rho,\rho')\in\calR_0$
  \item $\pi_0(\calR_0^c)\leq \tilde\delta$
  \item $D(\pi_0;p_{T_\eta}\mu,p_{T'_\eta}\mu')\leq \tilde\delta$
  \item $\dis(\calR)\leq 2\tilde\delta$.
  \end{enumerate}
  
Now, we perform the proof of \cite[Lemma~6.3]{EvansPitmanWinter06} and we shall use their notation. We introduce a function $f$ from $T_\eta$ to $T_\eta'$. First, let $f(\rho):=\rho'$ and then for each $x\in T_\eta$, one chooses $f(x)\in T_\eta'$ such that $(x,f(x))\in \calR_0$ (notice that this can be done in a measurable way). Then, letting $x_1,\ldots,x_n$ to be the leaves of $\bfT_\eta$ one defines $x'_i=f(x_i)$ and let $T''$ be the subtree of $T_\eta'$ spanned by $\rho',x'_1,\ldots,x'_n$. Finally, $\overline f(x)$  is defined to be the closest point to $f(x)$ on $T''$. Notice that $x'_i=\overline f(x_i)$. \cite[Lemma~6.3]{EvansPitmanWinter06} shows that $T''$ has leaves $x'_1,\ldots,x'_n$ (and root $\rho'=\overline f(\rho)$), that $\delta_H(T_\eta,T'')<3\tilde\delta$ and that the function $\overline f$ from $T_\eta$ to $T''$ has distortion at most $8\tilde\delta$. It is easy to see that
\begin{equation}
\label{eq:fbarf}
\forall x\in T,\;d'(\overline f(x),f(x))\leq 4\tilde\delta\;.
\end{equation}

In the proof of \cite[Lemma~6.3]{EvansPitmanWinter06}, they then take $y_1\in[\rho,x_1]$ and $y'_1\in[\rho',x'_1]$ such that $d(\rho,y_1)=d'(\rho',y'_1)=d(\rho,x_1)\land d'(\rho',x'_1)$ and define $\psi$ from $S_1:=[\rho,y_1]$ to $S'_1:=[\rho',y'_1]$ in the obvious way. The proof then proceeds inductively, defining $z_{k+1}$ (resp. $z'_{k+1}$) as the closest point to $x_{k+1}$ on $S_k$ (resp. to $x'_{k+1}$ on $S'_k$), letting $y_{k+1}\in ]z_{k+1},x_{k+1}]$ and $y'_{k+1}\in ]z'_{k+1},x'_{k+1}]$ be such that
$$ d(z_{k+1},y_{k+1})=d'(z'_{k+1},y'_{k+1})=d(z_{k+1},x_{k+1})\land d'(z'_{k+1},x'_{k+1})$$
defining $\psi$ from $]z_{k+1},y_{k+1}]$ to $]z'_{k+1},y'_{k+1}]$ in the obvious way and gluing $]z_{k+1},y_{k+1}]$ to $S_k$ to get $S_{k+1}$ (resp. $]z'_{k+1},y'_{k+1}]$ to $S'_k$ to get $S'_{k+1}$). Finally, let $S:=S_n$ and $S':=S'_n$. They prove then that:
$$\dis(\psi)<280\tilde\delta,\;\delta_H(S,T_\eta)<56\tilde\delta\;\quad\text{and}\quad\delta_H(S',T_\eta')<58\tilde\delta\;.$$
They also prove in \cite[p.113]{EvansPitmanWinter06} that the length measure of the set of points mentioned in $(iv)$ is at most $224\tilde\delta n$ where $n$ is the number of leaves of $R_\eta(T)$. Notice that $n$ depends only on $T$ and $\eta$, let us call it $n(\eta,T)$. This shows that $\psi$, $S$ and $S'$ satisfy $(i)-(iv)$ above if $\delta$ and $\eta$ are chosen small enough: first fix a positive $\eta\leq\eps/2000$, and then, choose $\delta<\frac{\eps}{2000n(\eta,T)}$.

Also, it is shown in  \cite[inequality (6.28)]{EvansPitmanWinter06}) that for any $k$, $d(x_k,y_k)\lor d'(x'_k,y'_k)\leq 12\tilde\delta$.

Now, let us show that:
\begin{equation}
\label{eq:fbarpsi}
  \forall x\in S,\;d'(\overline f(x),\psi(x))\leq 56\tilde\delta.
\end{equation}
Let $x\in]z_k,y_k]$, then $d'(\psi(x),y'_k)=d(x,y_k)$ (recall that $\psi(y_k)=y'_k$). Then, $|d(x,y_k)-d(x,x_k)|\leq 12\tilde\delta$ and $|d'(\psi(x),y'_k)-d'(\psi(x),x'_k)|\leq 12\tilde\delta$. Since $\overline f$ has distortion at most $8\tilde\delta$, $|d(x,x_k)-d'(\overline f(x),x'_k)|\leq 8\tilde\delta$. We get
$$|d'(\psi(x),x'_k)-d'(\overline f(x),x'_k)|\leq 32\tilde\delta\;.$$
Let $z$ be the closest point to $\overline f(x)$ on $[\rho',x'_k]$. Then,
\begin{eqnarray*}
d'(\overline f(x),z)&=&\frac{1}{2}[d'(\overline f(x),\rho')+d'(\overline f(x),x'_k)-d'(\rho',x'_k)]\\
&\leq&\frac{3}{2}\dis(\overline{f})+\frac{1}{2}[d(x,\rho)+d(x,x_k)-d(\rho,x_k)]\\
&\leq &12\tilde\delta\;,
\end{eqnarray*}
since $x\in[\rho,x_k]$. Finally, since $\psi(x)\in[\rho',x'_k]$,
\begin{eqnarray*}
d'(\overline f(x),\psi(x))&=&d'(\overline f(x),z)+d'(z,\psi(x))\\
&=&d'(\overline f(x),z)+ |d'(x'_k,\psi(x))-d'(x'_k,z)|\\
&\leq & 2d'(\overline f(x),z)+ |d'(x'_k,\psi(x))-d'(x'_k,\overline f(x))|\\
&\leq & 24\tilde\delta+32\tilde\delta\;.
\end{eqnarray*}
This shows \eqref{eq:fbarpsi}. 

Now, let $\calR$ be defined by:
$$\calR:=\left\{(x,x')\in S\times S'\;:\;\exists (y,y')\in \calR_0,\;\left(\begin{array}{c}d(x,y)\leq 100\tilde\delta\\ and \\ d'(x',y')\leq 100\tilde\delta\end{array}\right)\right\}\;,$$
and define $\pi:=(p_S\otimes p_{S'})\sharp\pi_0$. It remains to prove point $(v)$. First, recall that $(x,f(x))\in\calR_0$ for any $x\in T_\eta$. Thus $(v)(a)$ is satisfied thanks to \eqref{eq:fbarpsi} and \eqref{eq:fbarf}. This shows also that $\calR$ is a correspondence on $S\times S'$. 

Then, 
$$\dis(\calR)\leq \dis(\calR_0)+400\tilde\delta$$
which is less than $\eps$ and shows $(v)(d)$ if $\tilde\delta$ is chosen small enough. Since $\tilde\delta_H(S,T_\eta)\lor \tilde\delta_H(S',T_\eta')<58\tilde\delta$, we see that for any $x\in T_\eta$ and $x'\in T_\eta'$,
$$d(x,p_S(x))<58\tilde\delta\text{ and }d(x',p_{S'}(x'))<58\tilde\delta\;.$$
Thus, if $(x,x')\in\calR_0$, then $(p_S(x),p_S(x'))\in\calR$ and this gives
$$\pi(\calR^c)\leq \pi_0(\calR_0^c)\leq \tilde\delta\;,$$
which shows  $(v)(b)$ if $\tilde\delta$ is chosen small enough. Finally, since $\pi=(p_S\otimes p_{S'})\sharp\pi_0$ one sees that
$$D(\pi;p_S\sharp\mu,p_{S'}\sharp\mu')\leq D(\pi_0,\mu,\mu')<\tilde\delta\;.$$
This ends the proof.

\end{proof}
Now, let us prove the Feller property for trees.
\begin{prop}
\label{prop:FellerTrees}
Let $(\bfX^{n})_{n\geq 0}$ be a sequence in $\calN_2^{tree}$ converging to $\bfX$ (in the $L_{2,GHP}$ metric). Then
\begin{enumerate}[(i)]
\item  $(\Frag(\bfX,t))_{t\geq 0}$ is strong Markov with càdlàg trajectories (for $L_{2,GHP}$) in $\calN_2^{tree}$,
\item $(\Frag(\bfX^{n},t))_{t\geq 0}$ converges in distribution to $(\Frag(\bfX,t))_{t\geq 0}$ (for the Skorokhod topology associated to $L_{2,GHP}$),
\item if $t^n\xrightarrow[n\rightarrow\infty]{}t$, then $\Frag(\bfX^{n},t^n)$ converges in distribution to $\Frag(\bfX,t)$ (for $L_{2,GHP}$).
\end{enumerate}
\end{prop}
\begin{proof}
First, we argue that one may without loss of generality suppose that $\bfX^{n}$ and $\bfX$ contain a single component. Indeed, fix $\eps>0$. Since $\sizes(\bfX^{n})$ converges to $\sizes(\bfX)$ in $\ell^2$, one may choose $\eps'\not\in\sizes(\bfX)$ such that:
$$\|\sizes(\bfX_{\leq \eps'})\|_2^2\lor\sup_{n\in\NN}\|\sizes(\bfX^{n}_{\leq \eps'})\|_2^2\leq \eps\;.$$
Then, since $\bfX^{n}_{>\eps'}$ converges to $\bfX_{>\eps'}$ as $n$ goes to infinity, they have the same number of components for $n$ large enough. Call this number $K$. One may list them as follows: let $T_i^{n}$ (resp. $T_i$), $i=1,\ldots,K$ be the components of $\bfX^{n}$ (resp. of $\bfX$) such that for any $i$, $T_i^{n}$ converges to $T^{n}$. Fix $t>0$. Then, for any coupling between $(\Frag(\bfX,s))_{s\in[0,t]}$ and $(\Frag(\bfX^{n},s))_{s\in[0,t]}$, one has:
\begin{eqnarray*}
&&\|\sizes(\Frag(\bfX,s))-\sizes(\Frag(\bfX^{n},s'))\|_2^2\\
&\leq & \sum_{i=1}^K\|\sizes(\Frag(T_i,s))-\sizes(\Frag(T_i^{n},s'))\|_2^2\\
&& +\|\sizes(\bfX_{\leq \eps'})\|_2^2+\|\sizes(\bfX^{n}_{\leq \eps'})\|_2^2
\end{eqnarray*}
and 
$$L_{GHP}(\Frag(\bfX,s),\Frag(\bfX^{n},s'))\leq\sum_{i=1}^KL_{GHP}(\Frag(T_i,s),\Frag(T_i^{n},s'))+16\eps\;.$$
This shows that to prove $(i)$ and $(iii)$, one may suppose that $\bfX^{n}$ has a single component. Also, to prove $(ii)$, it is sufficient to prove that for any fixed $i$ and $n$, one may find a coupling such that  
$$\sup_{s\in[0,t]}L_{GHP}(\Frag(T_i^{n},s),\Frag(T_i,s))\xrightarrow[n\rightarrow\infty]{\PP}0\;.$$
In the sequel, we suppose that $\bfX^{n}=:\bfT^{n}$ (resp. $\bfX=:\bfT$) contains a single component.

Let us firt prove $(i)$. The strong Markov property was already noticed, see Remark~\ref{rem:frag2}, so let us prove that the trajectories are almost surely càdlàg. Then, for any $\eta>0$, $\Frag(R_{\eta}(\bfX),\cdot)$ clearly has càdlàg trajectories: it is right-continuous and piecewise constant with a finite number of jumps, since $R_{\eta}(\bfX)$ has finite length. Then, Lemma~\ref{lemm:approxReta} and Lemma~\ref{lemm:cvcadlag}  show that $\Frag(\bfX,\cdot)$ also has càdlàg trajectories.

Now, let us prove $(ii)$.  Let us fix $\eps>0$. We know (cf. \cite[Proposition~2.1]{AdBrGoMi2017}) that we can take $\rho$ (resp. $\rho^{n}$) a root in $T$ (resp. $T^{n}$) such that $d_{GHP}^{root}((\bfT,\rho), (\bfT^{n},\rho^{n}))$ goes to zero as $n$ goes to infinity: it can be done by taking the roots sampled from the respective measures (normalized to be probability measures). For $n$ large enough, $d_{GHP}^{root}((\bfT,\rho), (\bfT^{n},\rho^{n}))$ is small enough so that one may apply Lemma~\ref{lemm:EvansPitmanWinterbis}.  

Let us call $(\bfT',\rho')= (\bfT^{n},\rho^{n})$ for such a large $n$, in order to lighten the notation. Notice that one may suppose that $\mu'(T')\leq \mu(T)+\eps$. Let $\delta$, $\eta$, $S$, $S'$, $\psi$, $\calR$ and $\pi$ be as in Lemma~\ref{lemm:EvansPitmanWinterbis}. Define
$$\bfS:=(S,d|_{S\times S},p_S\sharp\mu)$$
and
$$\bfS':=(S',d|_{S'\times S'},p_{S'}\sharp\mu')\;.$$

Let $t>0$. Lemma~\ref{lemm:approxReta} ensures that with probability at least $1-2t\eps^{1/7}$, for any $s\in[0,t]$,
$$L_{2,GHP}(\Frag(\bfT,s),\Frag(\bfS,s))\leq 7\eps^{1/7}(1+\mu(T))^4\;,$$
and 
\begin{equation}
\label{eq:approxTprim}  L_{2,GHP}(\Frag(\bfT',s),\Frag(\bfS',s))\leq 7\eps^{1/7}(1+\mu'(T'))^4\leq  7\eps^{1/7}(1+\mu(T)+\eps)^4\;.
\end{equation}

For any $z\in S$ (resp. $z'\in S'$) we let $S_z$ (resp. $S'_{z'}$) be the subtree above $z$ (resp. above $z'$):
$$S_z:=\{x\in S\;:\;z\in[\rho,x]\}\;.$$
Let us define
$$Bad:=\{a\in S:S_{\psi(a)}\not=\psi(S_a)\}$$
so that Lemma~\ref{lemm:EvansPitmanWinterbis} ensures that $\ell_S(Bad)\leq\eps$.

Now, let $\calP$ be a Poisson random set of intensity $\ell_S\otimes\leb_{\RR^+}$ on $S\times \RR^+$. Then, for any $s$, $\psi(\calP_s)$ is a  Poisson random set of intensity $s\ell_S$ on $S$ (since $\psi$ is a measure-preserving bijection), and we want to show that for any $s\leq t$, the fragmentation of $\bfS$ along $\calP_s$, $\Frag(\bfS,\calP_s)$, and that of $\bfS'$ along $\psi(\calP_s)$, $\Frag(\bfS',\psi(\calP_s))$, are close in $L_{GHP}$-distance with large probability.

Notice first that $\Frag(\bfS,\calP_s)$ and $\Frag(\bfS',\psi(\calP_s))$ have the same number of components on $\calP_t\cap Bad=\emptyset$. If $m$ is a component of $\Frag(\bfS,\calP_s)$, it can be written as $S_{z_s}\setminus \bigcup_{i=1}^k{S_{z_{i,s}}}$ for some points $z_s,z_{1,s},\ldots z_{k,s}$ in $\calP_s\cup\{\rho\}$ (we identify $S_z\setminus\{z\}$ with $S_z$ since it is at zero $d_{GHP}$-distance). If $\calP_t \cap Bad=\emptyset$, then for any $s\leq t$, $\psi(m)=\psi(S_{z_s})\setminus \bigcup_{i=1}^k{S_{\psi(z_{i,s})}}$ and this is a component of $\Frag(\bfS',\psi(\calP_s))$.

Thus, let us place ourselves on the event $\calE_1:=\{\calP_t\cap Bad =\emptyset\}$ and define $\sigma$ (which depends on $s$) to be the bijection from $\Frag(\bfS,\calP_s)$ to $\Frag(\bfS',\psi(\calP_s))$ which maps a component $m$ to $\psi(m)$. Since $\calR$ contains the pairs $(x,\psi(x))$ for $x\in S$, $\calR|_{m\times\psi(m)}$ is a correspondence between $m$ and $\psi(m)$ with distortion at most $\eps$. Furthermore $\pi|_{m\times \psi(m)}$ is a measure on $m\times\psi(m)$ which satisfies
$$\pi|_{m\times \psi(m)}(\calR|_{m\times\psi(m)}^c)\leq\pi(\calR^c)\leq \eps\;.$$
It remains to bound $D(\pi|_{m\times \psi(m)};(p_S\sharp\mu)|_m,(p_{S'}\sharp\mu')|_{\psi(m)})$ from above. For any Borel subset $A$ of $m$,
\begin{eqnarray*}
&&|\pi|_{m\times \psi(m)}(A\times \psi(m))-p_S\sharp\mu(A)|\\
&\leq&|\pi(A\times S')-p_S\sharp\mu(A)|+\pi(A\times S')-\pi(A\times\psi(m))\\
&\leq&|\pi(A\times S')-p_S\sharp\mu(A)|+ \pi(\{(x,x')\in S\times S':x\in m,\;x'\not\in\psi(m)\})\;.
\end{eqnarray*}
A symmetric inequality holds for $A'$ a Borel subset of $\psi(m)$, and we get:
$$D(\pi|_{m\times \psi(m)};\mu|_m,\mu'|_{\psi(m)})\leq D(\pi;p_S\sharp\mu,p_{S'}\sharp\mu')+\pi(m\times\psi(m)^c)+\pi(m^c\times\psi(m))\;.$$
Now, notice that for any $x\in S$,
$$x\in m\text{ and }x'\not\in\psi(m)\Rightarrow [\psi(x),x']\cap\psi(\calP_t)\not=\emptyset$$
and
$$x\not\in m\text{ and }x'\in\psi(m)\Rightarrow [\psi(x),x']\cap\psi(\calP_t)\not=\emptyset\;,$$
where $[\psi(x),x']$ is the geodesic between $\psi(x)$ and $x'$. Thus,
$$\pi(m\times\psi(m)^c)+\pi(m^c\times\psi(m))\leq \pi\{(x,x')\in S\times S':[\psi(x),x']\cap\psi(\calP_t)\not=\emptyset\}\;.$$
Let us denote by $\calE_2$ the event 
$$\calE_2:=\left\{\pi\{(x,x')\in S\times S':[\psi(x),x']\cap\psi(\calP_t)\not=\emptyset\}\leq \sqrt{\eps}\right\}\;.$$
On $\calE_1\cap\calE_2$, we get, for any $s\leq t$ and any component $m$ of $\Frag(\bfS,\calP_s)$:
$$D(\pi|_{m\times \psi(m)};p_S\sharp\mu|_m,p_S\sharp\mu'|_{\psi(m)})\leq \eps+\sqrt{\eps}\;.$$
Furthermore,
\begin{eqnarray*}
&&\|\sizes(\Frag(\bfS,\calP_s))-\sizes(\Frag(\bfS',\psi(\calP_s)))\|_2^2\\
&\leq &\sum_{m\in\comp(\Frag(\bfS,\calP_s))}(p_S\sharp\mu(m)-p_{S'}\sharp\mu'(\psi(m)))^2\\
&\leq&\sup_{m\in\comp(\Frag(\bfS,\calP_s))}|p_S\sharp\mu(m)-p_{S'}\sharp\mu'(\psi(m))|\\
&&\times \sum_{m\in\comp(\Frag(\bfS,\calP_s))}p_S\sharp\mu(m)+p_{S'}\sharp\mu'(\psi(m))\\
&\leq&\sup_{m\in\comp(\Frag(\bfS,\calP_s))}D(\pi|_{m\times \psi(m)};p_S\sharp\mu|_m,p_S\sharp\mu'|_{\psi(m)})(\mu(T)+\mu'(T'))\\
&\leq & (\eps+\sqrt{\eps})(2\mu(T)+\eps)
\end{eqnarray*}
Thus, on $\calE_1\cap\calE_2$, we obtain, for any $s\leq t$:
$$L_{2,GHP}(\Frag(\bfS,\calP_s),\Frag(\bfS',\psi(\calP_s)))\leq (\eps+\sqrt{\eps})\lor \sqrt{(\eps+\sqrt{\eps})(2\mu(T)+\eps)}\;.$$

It remains to bound from above the probability of $(\calE_1\cap\calE_2)^c$. Since $Bad$ has length measure at most $\eps$,
$$\PP(\calE_1^c)\leq t\eps\;.$$
Notice that since $\calR$ contains $(x,\psi(x))$ for any $x\in S$ and has distortion less than $2\eps$,
$$\pi\{(x,x')\in S\times S':d'(\psi(x),x')>2\eps\}\leq \pi(\calR^c)<\eps$$
Then, using Fubini's theorem,
\begin{eqnarray*}
&&\EE[\pi\{(x,x')\in S\times S':[\psi(x),x']\cap\psi(\calP_t)\not=\emptyset\}]\\
&\leq&\eps+\EE[\pi\{(x,x')\in S\times S':[\psi(x),x']\cap\psi(\calP_t)\not=\emptyset\text{ and }d'(\psi(x),x')\leq 2\eps\}]\\
&=&\eps+\int_{S\times S'} \PP([\psi(x),x']\cap\psi(\calP_t)\not=\emptyset)\II_{ d'(\psi(x),x')\leq 2\eps}\;d\pi(x,x')\\
&\leq &2t\eps\pi(S\times S')\\
&\leq &2t\eps(\mu(T)+\eps)\;.
\end{eqnarray*}
Thus, by Markov's inequality,
$$\PP(\calE_2^c)\leq 2t\sqrt{\eps}(\mu(T)+\eps)\;,$$
which ends the proof of point $(ii)$ of the theorem (through Lemma~\ref{lemm:compactcv} and inequality~\eqref{eq:dSkdc}).

Finally, let us prove $(iii)$. Suppose that $t^{n}$ converges to $t$ as $n$ goes to infinity and let $\tilde t:=\sup_nt^{n}+1$. Again, it is sufficient to suppose that $\bfX^{n}$ and $\bfX$ have only one component, so let us suppose that $(\bfT,\rho)$ and $(\bfT^{n},\rho^{n})$, $n\geq 0$ are rooted trees such that $d_{GHP}^{root}((\bfT,\rho), (\bfT^{n},\rho^{n}))$ goes to zero as $n$ goes to infinity. Now recall inequality \eqref{eq:approxTprim} above: for any $\eps>0$,  for $n$ large enough, we found finite subtrees $S^{n}\subset T^{n}$ such that, with probability at least $1-2T\eps^{1/7}$, for any $s\in[0,\tilde t]$,
$$L_{2,GHP}(\Frag(\bfT^{n},s),\Frag(\bfS^{n},s))\leq 7\eps^{1/7}(1+\mu(T)+\eps)^4\;,$$
and furthermore, $\ell_{T^{n}}(S^{n})\leq \ell_T(R_\eta(T))<\infty$, since $S^{(n)}$ has the same length measure as a subset of $R_\eta(T)$, $\eta$ depending only on $\eps$ (cf. Lemma~\ref{lemm:EvansPitmanWinterbis}). Then,
$$\PP(\Frag(\bfS^{n},t^{n})\not=\Frag(\bfS^{n},t))\leq 1-e^{-|t-t^{n}|\ell_T(R_\eta(T))}\;.$$
For $n$ large enough, this is less than $\eps$. Then, we have that with probability at least $1-2\tilde t\eps^{1/7}-\eps$,
$$L_{2,GHP}(\Frag(\bfT^{n},t^{n}),\Frag(\bfT^{n},t))\leq 7\eps^{1/7}(1+\mu(T)+\eps)^4\;.$$
All in all, we proved that $L_{2,GHP}(\Frag(\bfT^{n},t^{n}),\Frag(\bfT^{n},t))$ converges in probability to zero when $n$ goes to infinity. But using point $(i)$, we know that $\Frag(\bfT^{n},t)$  converges to $\Frag(\bfT,t)$ as $n$ goes to infinity. This ends the proof of $(iii)$.
\end{proof}

\subsection{The Feller property for graphs, Theorem~\ref{theo:FellerGraphs}}
\label{subsec:FellerGraphs}
We now want to prove Theorem~\ref{theo:FellerGraphs}, which is the analog of Proposition~\ref{prop:FellerTrees} for graphs. However, this cannot be true without strengthening the metric $L_{GHP}$. For instance, consider the situation depicted in Figure~\ref{fig:cexFellergraph}. There, $G_n$ converges to $G$ for $d_{GHP}$, but the probability that $a$ is separated from $b$ in $G_n$ when fragmentation occurs (until a fixed time $t>0$) is exactly $1-(1-(1-e^{-t/n})^{2})^n$, which is asymptotically $0$, whereas the probability that this event occurs in $G$ is strictly positive. However, if we impose that the surplus of $G_n$ converges to the surplus of $G$, such a situation cannot happen anymore, and one may recover the Feller property. 

\begin{figure}[!htbp]
\begin{center}
\psfrag{a}{$a$}
\psfrag{b}{$b$}
\psfrag{cd}{$\cdots$}
\psfrag{mun}{$\mu_n=\frac{1}{2}(\delta_a+\delta_b)$}
\psfrag{mu}{$\mu=\frac{1}{2}(\delta_a+\delta_b)$}
\psfrag{d}[]{$diam(G_n)=diam(G)=1$}
\psfrag{Gn}{$G_n:$}
\psfrag{G}{$G:$}
\psfrag{i}[]{length $\frac{1}{n}$}
\includegraphics[width=12cm]{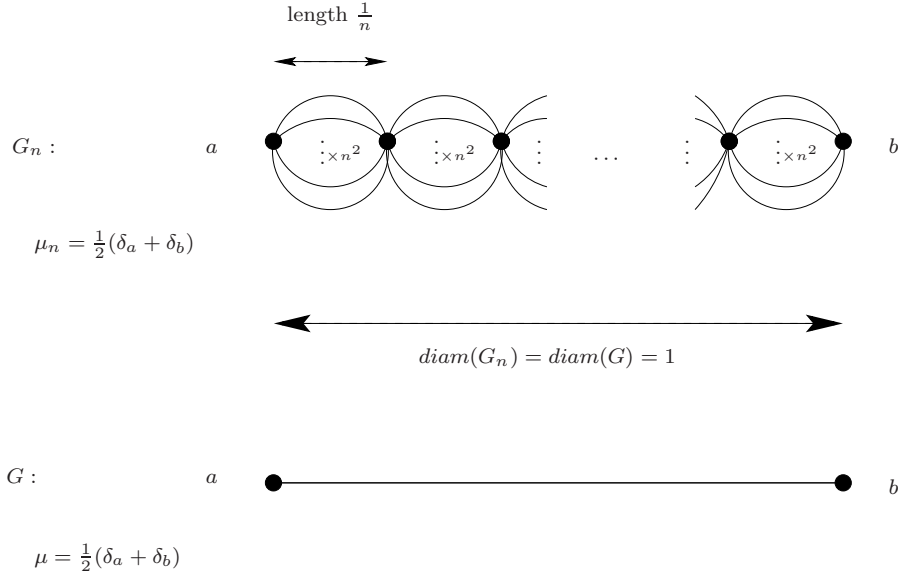}
\caption{$G_n$ is composed of $n$ graphs in series each one made of $2$ intervals of length $1/n$ in parallel. $G_n$ converges to $G$ for $d_{GHP}$ when $n$ goes to infinity, but $\Frag(G_n,t)$ will not converge to $\Frag(G,t)$ for $t>0$.}
\label{fig:cexFellergraph}
\end{center}
\end{figure}

Let us notice that this problem was treated a bit differently in \cite{AdBrGoMi2017}: they recover continuity (in probability) of fragmentation by imposing that $G_n$ and $G$ live on some common subspace $\calA_r$ for some  $r>0$, where $\calA_r$ contains the graphs which have surplus and total length of the core bounded from above by $1/r$ and minimal edge length of the core bounded from below by $r$ (see section~6.4 in \cite{AdBrGoMi2017} for a precise statement). When one wants to have Feller-type properties, this seems to us less natural than imposing convergence of the surplus. In fact, the work below shows that if $G_n$ converges to $G$ in the Gromov-Hausdorff topology while having the same surplus for $n$ large enough, then there is some $r>0$ such that for $n$ large enough, $G_n$ and $G$ belong to $\calA_r$. The converse statement is also true and is a consequence of \cite[Proposition~6.5]{AdBrGoMi2017}.

To prove Theorem~\ref{theo:FellerGraphs}, we first notice that the proof of section~\ref{subsec:FellerTrees} extends to the case where one replace trees by graphs having the same core.

\begin{lemm}
\label{lemm:FellerRootCore}
Let $\bfG=(G,d,\mu)$ be a  measured $\RR$-graph which is not a tree. Let $(\bfG^{n})_{n\geq 0}$ be a sequence of measured $\RR$-graphs such that for each $n$, there is a correspondence $\calR^{n}\in C(G,G^{n})$, a measure $\pi^{n}\in M(G,G^{n})$ and a homeomorphism $\psi^{n}:\core(G)\rightarrow \core (G^{n})$ such that:
\begin{itemize}
\item $\psi^{n}$ preserves the length-measure,
\item $\forall x\in \core(G^{n})\;(x,\psi^{n}(x))\in \calR^{n}$,
\item $\dis(\calR^{n})\lor\pi^{n}((\calR^{n})^c)\lor D(\pi^{n},\mu,\mu^{n})\xrightarrow[n\rightarrow\infty]{}0$.
\end{itemize}
Then, the sequence of processes $\Frag(\bfG^{n},\cdot)$ converges in distribution to $\Frag(\bfG,\cdot)$ for $L_{2,GHP}^{surplus}$.
\end{lemm}
\begin{proof}
It is a straightforward extension of the arguments of section~\ref{subsec:FellerTrees}, replacing roots by cores and using $\psi^{n}$ to map fragmentation on $\core(G^{n})$ to fragmentation on $\core(G)$. 
\end{proof}
To prepare the proof of  Theorem~\ref{theo:FellerGraphs} we shall need a series of lemmas, but before, let us explain the idea of the proof of the theorem.  If $\bfG^{n}$ is close enough to $\bfG$, Lemma~\ref{lemm:overlay} below shows that their cores are homomorphic multigraphs with edges having almost the same length. One may then shorten some edges of the core of $G$ and other edges of the core of $\bfG^{n}$ in such a way that the two cores become homeomorphic as metric spaces with a length measure. Lemma ~\ref{lemm:shortening} shows that one does not lose too much doing this. Finally, Lemma~\ref{lemm:FellerRootCore} then shows that the fragmentations on the two graphs are close to each other.

\begin{lemm}
\label{lem:gammaab}
Let $(G,d)$ and $(G',d')$ be $\RR$-graphs and $\calR\in C(G,G')$. Let $(a,a')\in\calR$, $(b,b')\in\calR$ and $(c,c')\in\calR$. Suppose that $a$ belongs to a geodesic between $b$ and $c$. Let $\gamma_{a',b'}$ (resp. $\gamma_{a',c'}$) be a geodesic from $a'$ to $b'$ (resp. from $a'$ to $c'$). Then,
$$\forall a''\in \gamma_{a',b'}\cap \gamma_{a',c'},\;d'(a'',a')\leq 3\dis(\calR)\;.$$
\end{lemm}
\begin{proof}
Let $ a''\in \gamma_{a',b'}\cap \gamma_{a',c'}$. Then,
\begin{eqnarray*}
d'(a',a'') &=& d'(a',b')+d'(a',c')-d'(a'',b')-d'(a'',c')\\
&\leq & d'(a',b')+d'(a',c')-d'(b',c')\\
&\leq & d(a,b)+d(a,c)-d(b,c)+3\dis(\calR)\\
&=& 3\dis(\calR)
\end{eqnarray*}
where we used the triangle inequality in the second step and the fact that $a$ belongs to a geodesic between $b$ and $c$ in the last step.
\end{proof}

The following should be compared to \cite[Proposition~5.6]{AdBrGoMi2017}.

\begin{lemm}
\label{lemm:overlay}
Let $G$ be an $\RR$-graph and $\eps>0$. There exists $\delta$ depending on $\eps$ and $G$ such that if $G'$ is an $\RR$-graph with the same surplus as $G$ and if $\calR_0\in C(G,G')$ is such that $\dis(\calR_0)<\delta$, then there exists an $\eps$-overlay $\calR\in C(G,G')$ containing $\calR_0$.
\end{lemm}
\begin{proof}
If $G$ has surplus $0$, there is nothing to prove. In the sequel, we suppose that $G$ has surplus at least 2, the easier proof for unicyclic $G$ is left to the reader. Furthermore, to lighten notation and make the argument clearer, we shall suppose that the vertices of $\ker(G)$ are of degree $3$, leaving the adaptation to the general case to the reader.

Let $\eta:=\min_{e\in e(G)}\ell(e)$. Notice that one may view $\core(G)$ (and $\core(G')$) as a multigraph with edge-lengths, and we shall adopt this point of view in this proof. However, not all the edges of this graph correspond to geodesics in $G$. Divide each edge of $\core(G)$ into five pieces of equal length, introducing thus four new vertices of degree 2 for each edge (all degrees will be relative to the core). The new graph obtained satisfies the following:
\begin{enumerate}[(i)]
\item all the edges remain of length larger than $\eta/5$,
\item every edge $e$ is the unique  geodesic between its two endpoints, and for any path $\gamma$ between these endpoints which does not contain $e$, $\ell(\gamma)-\ell(e)>\eta/5$,
\item for every three vertices $a$, $b$, $c$ such that $b\sim a$ and $a\sim c$, $a$ belongs to a geodesic between $a$ and $c$.
\end{enumerate}
Let us call $\tilde\core(G)$ this new graph (it is indeed a graph, not merely a multigraph), which has the same surplus as $G$, and write $x_1,\ldots,x_n$ for its vertices, which are of degree $2$ or $3$.

Let $G'$ be an $\RR$-graph with the same surplus as $G$ and $\calR_0\in C(G,G')$. Let $x'_1,\ldots,x'_n$ be elements of $G'$ such that $(x_i,x'_i)\in\calR_0$. Now, we shall build a subgraph of $G'$ by mapping recursively edges adjacent to a given vertex in $\tilde\core(G)$ to a geodesic in $G'$. Suppose for instance that $x_1$ has degree $3$ (the argument is analogous for vertices with degree 2). Let $x_i$, $x_j$ and $x_k$ be its neighbours in $\tilde\core(G)$, with $i<j<k$. Choose a geodesic $\gamma_{x'_1,x'_i}$ between $x'_1$ and $x'_i$, then choose a geodesic $\gamma$ between $x'_1$ and $x'_j$, and let $z_1^1$ be the point of $\gamma_{x'_1,x'_i}\cap\gamma$ which is the furthest from $x'_1$ (see Figure~\ref{fig:buildoverlay}). Let us call $\gamma_{z_1^1,x'_j}$ the subpath of $\gamma$ from $z_1^1$ to $x'_j$. Notice that the path using $\gamma_{x'_1,x'_i}$ from $x'_1$ to $z_1^1$ and $\gamma$ from $z_1^1$ to $x'_j$ is a geodesic. Finally choose a geodesic $\gamma$ between $x'_1$ and $x'_k$ and let $z_1^2$ be the point of $(\gamma_{x'_1,x'_i}\cup\gamma_{x'_1,x'_j})\cap\gamma$ which is the furthest from $x'_1$. Let us call $\gamma_{z_1^2,x'_k}$ the subpath of $\gamma$ from $z_1^2$ to $x'_k$. Let $S'_1:=\gamma_{x'_1,x'_i}\cup\gamma_{z_i^1,x'_j}\cup\gamma_{z_1^2,x'_k}$. Define $x''_1$ to be the point between $z_1^1$ and $z_1^2$  which is the furthest from $x'_1$. If $x_1$ is of degree $2$, there is only one point $z_1^1$ defined and $x''_1$ is this one.

\begin{figure}[!htbp]
\begin{center}
\psfrag{x1}{$x_1$}
\psfrag{xi}{$x_i$}
\psfrag{xj}{$x_j$}
\psfrag{xk}{$x_k$}
\psfrag{xp1}{$x'_1$}
\psfrag{xpi}{$x'_i$}
\psfrag{xpj}{$x'_j$}
\psfrag{xpk}{$x'_k$}
\psfrag{g1}{$\gamma_{x'_1,x'_i}$}
\psfrag{g2}{$\gamma_{z_1^1,x'_j}$}
\psfrag{g3}{$\gamma_{z_1^2,x'_k}$}
\psfrag{z11}{$\scriptstyle z_1^1$}
\psfrag{z12}{$\scriptstyle z_1^2=x''_1$}
\includegraphics[width=12cm]{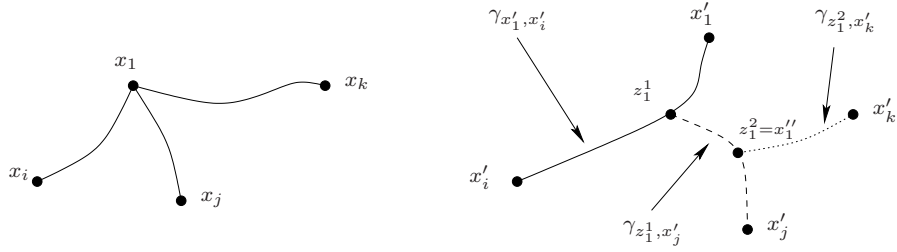}
\caption{One maps $\core(G)$ to $\core(G')$ by first mapping the neighborhood of each vertex of $\core(G)$ to a subset of $G'$. Here $x''_1$ is a vertex of $\ker(G')$.}
\label{fig:buildoverlay}
\end{center}
\end{figure}

Then, we proceed similarly for $r=2,\ldots,n$: we inspect the neighbours of $x_r$. Notice that we do not need to choose a new geodesic between $x'_r$ and a neighbour $x'_j$ for $j<r$, we just keep the one already built. Doing this, we obtain $S'_r$ the union of the geodesics chosen going from $x'_r$ to the points associated to the neighbours of $x_r$, we get two points $z_r^1$ and $z_r^2$ if $x_r$ is of degree 3 and only one point $z_r^1$ if $x_r$ is of degree 1. We define $x''_r$ to be the one between $z_r^1$ and $z_r^2$  which is the furthest from $x'_r$. 

Finally, let $S'=\cup_{i=1}^n S'_i$, with all the vertices $z_i^b$ and $x'_i$  for $1\leq i\leq n$. This is a graph with edge-lengths (notice that the edges have pairwise disjoint interiors). Some edge-lengths might be zero. Thanks to point $(iii)$ above and Lemma~\ref{lem:gammaab}, we know that:
$$d'(z_i^1,x_i')\leq 3\dis(\calR_0)\;,$$
and when $x'_i$ is of degree $3$,
$$d'(z_i^2,x_i')\leq 3\dis(\calR_0)\;.$$
Thus, for any $b,b'\in\{1,2\}$ and any $i\not=j$,
\begin{eqnarray*}
d'(z_i^b,z_j^{b'})&\geq &d'(x'_i,x'_j)-6\dis(\calR_0)\\
&\geq & \frac{\eta}{5} -7\dis(\calR_0)\;.
\end{eqnarray*}
Thus, if $\dis(\calR_0)<\eta/35$, two points $z_i^b$ and $z_j^{b'}$ are always distinct. This shows that $S'$ has the same surplus as $\core(G)$. Since $G'$ has the same surplus as $G$, we deduce that $S'$ contains $\core(G')$. Let $S''$ be the subgraph of $S'$ spanned by $x''_1,\ldots x''_n$, in the sense that we forget the vertices $z_i^1$ when $x_i$ is of degree $3$, and we remove the semi-open path going from $x'_i$ to $z_i^1$. Notice that $S''$ has positive edge-lengths and its edges have pairwise disjoint interiors. $S''$ has the same surplus as $S'$, so it contains again $\core(G')$. But all the vertices in $S''$ have degree $2$ or $3$, so $S''=\core(G')$ as a set.

Now, consider $S''$ as a graph with edge-lengths and with vertices $x''_i$, $i=1,\ldots,r$. Let $\chi_0$ be the map from $\tilde\core(G)$ to $S''$ which maps $x_i$ to $x''_i$. We shall see that it is a graph isomorphism if $\dis(\calR_0)$ is small enough. Indeed, from the inequalities above, we get that for any $i$ and $j$,
\[|d'(x''_i,x''_j)-d(x_i,x_j)|\leq 7\dis(\calR_0)\;\;.\]
Now, for any edge $e=(x_i,x_j)$ and any $k$ distinct from $i$ and $j$,
\begin{eqnarray*}
  d'(x''_i,x''_k)+d'(x''_k,x''_j)
  &\geq &d(x_i,x_k)+d(x_k,x_j)-14\dis(\calR_0)\\
  &\geq &d(x_i,x_j)+\frac{\eta}{5}-14\dis(\calR_0)\\
  &\geq &d'(x''_i,x''_j)+\frac{\eta}{5}-21\dis(\calR_0)\;,
\end{eqnarray*}
where we used point $(ii)$ above in the last inequality. Thus, if $\dis(\calR_0)<\frac{\eta}{105}$, $x''_i$ and $x''_j$ are neighbours in $S''$ as soon as $x_i$ and $x_j$ are neighbours in $\tilde\core(G)$. Furthermore, from the construction of $S''$, one sees that the number of edges in $S''$ is at most the number of edges of $\tilde\core(G)$. Thus, $\chi_0$ is a graph isomorphism and we deduce from the last inequality that for any edge $e$, $\chi_0(e)$ is the unique geodesic between its endpoints. Furthermore, let $x''_i$ and $x''_j$ be neighbours in $S''$. If $\dis(\calR_0)<\frac{\eta}{210}$, we see from the last inequality that every path $\gamma$ from $x''_i$ to $x''_j$ which does not contain $[x''_i,x''_j]$ satisfies:
\begin{equation}
\label{eq:geodsec}
\ell'(\gamma)> \ell'([x''_i,x''_j])+\frac{\eta}{10}\;.
\end{equation}


Let us define $\calR'$ by adding to $\calR_0$ the pairs $(x_i,x''_i)$ for $i=1,\ldots,r$. Then, $\dis(\calR')\leq 7\dis(\calR_0)$. Let $\calR$ be the $3\dis(\calR_0)$-enlargement of $\calR'$. It has distortion at most $19\dis(\calR_0)$. Let $x$ belong to an edge $[x_i,x_j]$ of $\tilde\core(G)$ and let $x'$ be such that $(x,x')\in\calR_0$. Let $\gamma_{x',x''_i}$ (resp. $\gamma_{x',x''_j}$) be a geodesic between $x'$ and $x'_i$ (resp. between $x'$ and $x''_j$). Then, let $\gamma$ be the path from $x''$ to $x''_j$ obtained by concatenating $\gamma_{x',x''_i}$ and $\gamma_{x',x''_j}$. We have
\begin{eqnarray*}
\ell'(\gamma)&\leq&d'(x''_i,x')+d'(x',x''_j)\;,\\
&\leq &d(x_i,x)+d(x,x_j)+2\dis(\calR')\;,\\
&=& d(x_i,x_j)+2\dis(\calR')\;,\\
&\leq & \ell'([x''_i,x''_j])+3\dis(\calR')\;.
\end{eqnarray*}
Thus, if $21\dis(\calR_0)<\frac{\eta}{10}$, $3\dis(\calR')< \frac{\eta}{10}$ and we deduce from \eqref{eq:geodsec} that  $\gamma$ contains $[x''_i,x''_j]$. Thus, defining $x''$ to be the furthest point from $x'$ on $\gamma_{x',x''_i}\cap\gamma_{x',x''_j}$, we see that $x''$ belongs to the geodesic $[x''_i,x''_j]$. Lemma~\ref{lem:gammaab} ensures that $d'(x',x'')\leq 3\dis(\calR')$. Thus, $(x,x'')\in\calR$. Similarly, one shows that for every $x''$ in $[x''_i,x''_j]$ there is an $x$ in $[x_i,x_j]$ such that $x\in\calR$. We have shown that for each edge $e$ of $\tilde\core(G)$, $e$ and $\chi_0(e)$ are in correspondence via $\calR$.
 Now, notice that the multigraph with edge-lengths $S''$ obtained by keeping only vertices of degree $3$ is $\core(G')$ seen as a multigraph with edge-lengths. The isomorphism $\chi_0$ induces an isomorphism $\chi$ between $\core(G)$ and $\core(G')$ (by restricting $\chi_0$ to vertices of degree $3$), and we have (since every edge of $\core(G)$ was divided into five parts):
$$|\ell(e)-\ell'(\chi(e))|\leq 30\dis(\calR_0)\;.$$
Furthermore, the same correspondence $\calR$ as before is suitable to have that for each edge $e$ of $\core(G)$, $e$ and $\chi(e)$ are in correspondence via $\calR$.

This ends the proof by taking $\dis(\calR_0)< \delta$ for $\delta$ small enough, namely less than $\frac{\eps}{40}\land\frac{\eta}{210}$.
\end{proof}

\begin{lemm}
\label{lemm:corresalpha}
Let $(G,d)$ and $(G',d')$ be $\RR$-graphs and $\calR\in C(G,G')$. Suppose that $\core(G)$ and $\core(G')$ are in correspondence through $\calR$. Let $(v,v')$ and $(x,x')\in\calR$ with $v\in\core(G)$ and $v'\in\core(G')$. Then,
$$d(\alpha_G(x),v)\leq d'(\alpha_{G'}(x'),v')+5\dis(\calR)\;.$$
\end{lemm}
\begin{proof}
Since $\core(G)$ and $\core(G')$ are in correspondence through $\calR$, one may find $y\in\core(G)$ and $y'\in\core(G')$ such that:
$$(y,\alpha_{G'}(x'))\in\calR\text{ and }(\alpha_G(x),y')\in\calR\;.$$
Let us distinguish two cases.
\begin{itemize}
\item $d(y,v)\geq d(\alpha_G(x),v)$. Then, 
$$d(\alpha_G(x),v)\leq d(y,v)\leq d'(\alpha_{G'}(x'),v')+\dis(\calR)$$
and the result follows.
\item $d(y,v)< d(\alpha_G(x),v)$. 
Then,
\begin{eqnarray*}
d(x,v) &=& d(x,\alpha_G(x))+d(\alpha_G(x),v)\\
&\geq & d(x,\alpha_G(x))+d(y,v)\\
&\geq & d'(x',y')+d'(\alpha_{G'}(x'),v')-2\dis(\calR)\\
&=& d'(x',\alpha_{G'}(x'))+d'(\alpha_{G'}(x'),y')+d'(\alpha_{G'}(x'),v')-2\dis(\calR)\\
&= & d'(x',v')+d'(\alpha_{G'}(x'),y')-2\dis(\calR)\\
&\geq & d(x,v)+d'(\alpha_{G'}(x'),y')-3\dis(\calR)\;.
\end{eqnarray*}
Thus, 
$$d'(\alpha_{G'}(x'),y')\leq 3\dis(\calR)$$
which  implies:
$$d(y,\alpha_{G}(x))\leq 4\dis(\calR)$$
Finally,
\begin{eqnarray*}
d(\alpha_G(x),v)&\leq&d(\alpha_G(x),y)+d(y,v)\\
&\leq & 4\dis(\calR)+d(y,v)\\
&\leq & 5\dis(\calR)+d'(\alpha_{G'}(x'),v')
\end{eqnarray*}
\end{itemize}
\end{proof}

Let us introduce some notation for the following lemmas (see Figure~\ref{fig:shortening}).

\begin{figure}[!htbp]
\begin{center}
\psfrag{u}[]{$\scriptstyle u$}
\psfrag{v}[]{$\scriptstyle v$}
\psfrag{vb}[]{$\scriptstyle v-be$}
\psfrag{va}[]{$\scriptstyle v-ae$}
\psfrag{G}[]{$\scriptstyle \bfG$}
\psfrag{Ge}[]{$\scriptstyle \bfG^{(e,a,b)}$}
\psfrag{dist}[]{$\scriptstyle glued$}
\includegraphics[width=8cm]{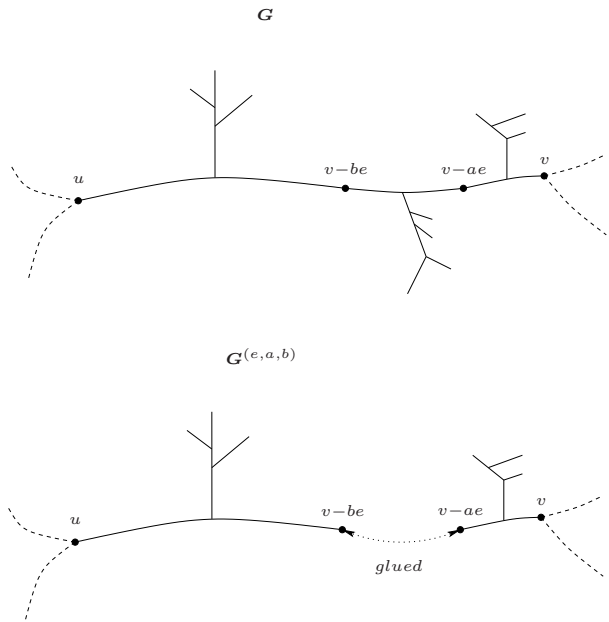}
\caption{$\bfG^{(e,a,b)}$ is \emph{the $(a,b)$-shortening of $G$ along $e=(u,v)$}.}
\label{fig:shortening}
\end{center}
\end{figure}

\begin{defi}
\label{defi:shortening}
For any graph $G$, for each oriented edge $e=(u,v)\in\ker(G)$ and each $\eta\in[0,\ell(e)]$, we denote by $v-\eta e$ the point at distance $\eta$ from $v$ on the edge $(u,v)$, on $\core(G)$. For $ a<b$ in $[0,\ell(e)]$, let $]v-b e,v-a e[$ be the open oriented arc between $v-b e$ and $v-ae $ in $(u,v)$. 

We define $\bfG^{(e,a,b)}$ \emph{the $(a,b)$-shortening along $e$} as the measured $\RR$-graph $(H,d_H,\mu_H)$ obtained from $G$ as follows:
\begin{itemize}
\item $H=G\setminus \alpha_{G}^{-1}(]v-b e,v-a e[)$,
\item $d_H$ is obtained from $(H,d|_{H\times H})$ by gluing it along the equivalence relation generated by $\{(v-b e,v-a e)\}$ (thus $d_H(v-be,v-ae)=0$),
\item $\mu_H$ is the restriction of $\mu$ on $H$.
\end{itemize}
\end{defi}
Notice that $\bfG^{(e,a,b)}$ has the same surplus as $G$.

\begin{lemm}
\label{lem:gammaG}
Let $G$ be an $\RR$-graph, and define:
$$\gamma_G(\eta):=\sup_{e=(u,v)\in\ker(G)}\diam(\alpha_G^{-1}(]v-\eta e,v[))\;.$$
Then, 
$$\gamma_G(\eta)\xrightarrow[\eta\rightarrow 0]{}0\;.$$
\end{lemm}
\begin{proof}
Suppose on the contrary that 
$\gamma_G(\eta)\xrightarrow[\eta\rightarrow 0]{}\gamma>0\;.$
Then, one may find an edge $e$ and a sequence of pairs $(x_n,y_n)_{n\in\NN}$ in $\alpha_G^{-1}(]v-e,v[)$ such that:
$$d(\alpha_G(x_n),v)\lor d(\alpha_G(y_n),v)\xrightarrow[n\rightarrow\infty]{}0\;,$$
$$\forall n\in\NN,\;d(x_n,y_n)\geq \gamma\;,$$
and
$$\forall n\in\NN,\; d(\alpha_G(x_n),v)\land d(\alpha_G(y_n),v)>0\;.$$
Let $z_n\in\{x_n,y_n\}$ be such that $d(z_n,v)=d(x_n,v)\lor d(y_n,v)$. Up to extracting a subsequence, one may also suppose that $d(\alpha_G(z_n),v)$ is strictly decreasing and that for any $n$, $d(\alpha_G(z_n),v)<\gamma/4$. This implies that for $n\not=m$,
\begin{eqnarray*}
d(z_n,z_m)&\geq &d(z_n,\alpha_G(z_n))\\
&=&d(z_n,v)-d(\alpha_G(z_n),v)\\
&\geq &\frac{\gamma}{2}-d(\alpha_G(z_n),v)\\
&\geq& \frac{\gamma}{4}\;.
\end{eqnarray*}
This contradicts the precompactness of $G$.
\end{proof}

\begin{lemm}
\label{lemm:shortening}
Let $\bfG=(G,d,\mu)$ be a measured $\RR$-graph with surplus at least one, let $e$ be an edge of $\core(G)$ and $a<b\in[0,\ell(e)]$. Let:
$$\tilde\gamma_G(\eps):=\sum_{e=(u,v)\in\ker(G)}\mu(\alpha_G^{-1}(]v-\eps e,v[))\;.$$
Then, under the natural coupling between $\Frag(\bfG,.)$ and $\Frag(\bfG^{(e,a,b)},.)$ we have, with probability at least $1-t(b-a)$, for any $s\in[0,t]$,
$$L_{2,GHP}^{surplus}(\Frag(\bfG,s),\Frag(\bfG^{(e,a,b)},s))\leq C(G)\sqrt{\tilde\gamma_G(b)\lor \gamma_G(b)}$$
where $C(G)$ is a positive and finite constant depending only on $diam(G)$ and $\mu(G)$.
\end{lemm}
\begin{proof}
Let $e=(u,v)$, and $\calP$ be a Poisson random set of intensity $\ell_G\times\leb_{\RR^+}$ on $(G,d)$. Then, $\calP':=[\calP\setminus \alpha_{G}^{-1}(]v-b e,v-a e[)]\times \RR^+$ is a Poisson random set of intensity $\ell_{G'}\times\leb_{\RR^+}$ on $G'\times \RR^+$ with $G':=G\setminus \alpha_{G}^{-1}(]v-b e,v-a e[)$. Let $t>0$ be fixed and let $\calE$ denote the event
$$\calE:=\{\calP_t\cap ]v-be,v-ae[=\emptyset\}\;,$$
and let us suppose that $\calE$ holds. Let $\eps>0$ be such that
\begin{equation}
  \label{eq:epsshortening}
  \eps\geq \mu(\alpha_G^{-1}(]v-be,v-ae[))\;.
\end{equation}
Let us take $s\leq t$ and let $m$ be a component of $\Frag(G,\calP_s)$. Then,
\begin{itemize}
\item if $m\subset\alpha_G^{-1}(]v-b e,v-a e[)$, then $\mu(m)\leq \eps$,
\item if $m\cap \alpha_G^{-1}(]v-b e,v-a e[)=\emptyset$, then $m$ is a component of $\Frag(G',\calP'_s)$,
\item if $m\cap \alpha_G^{-1}(]v-b e,v-a e[)\not=\emptyset$ but $m\not\subset\alpha_G^{-1}(]v-b e,v-a e[)$, then $m$ is the unique component of $\Frag(G,\calP_s)$ which intersects $]v-b e,v-a e[$, and $m\setminus \alpha_G^{-1}(]v-b e,v-a e[)$ is a component of $\Frag(G',\calP'_s)$.
\end{itemize}
This shows that the function 
$$\sigma:\left\lbrace\begin{array}{ccl}\comp(\Frag(G,\calP_s)_{>\eps}) &\rightarrow &\comp(\Frag(G',\calP'_s))\\ m &\mapsto& m\setminus\alpha_G^{-1}(]v-b e,v-a e[)\end{array}\right.$$
is well defined and injective. This shows also that the function $\sigma'$ from $\comp(\Frag(G',\calP'_s)_{>\eps})$ to $\comp(\Frag(G,\calP_s)_{>\eps})$ which maps $m'$ to the unique $m$ which contains it is well defined and injective.

Now, let $m\in\comp(\Frag(G,\calP_s)_{>\eps})$ and let 
$$m':=\sigma(m)=m\setminus\alpha_G^{-1}(]v-b e,v-a e[)\;.$$ 
Let 
$$\calR_m:=\{(x,x):x\in m'\}\cup\{(x,v):x\in m\cap\alpha_G^{-1}(]v-b e,v-a e[\}$$
and $\pi_m$ be the measure in $M(m,m')$ defined by:
$$\pi_m(C)=\mu(\{x\in m':(x,x)\in C\})\;.$$
Let $d'$ be the distance on $m'$. Notice that for any $x,y$ in $m'$,
$$|d(x,y)-d'(x,y)|\leq b$$
Thus,
\begin{equation}
  \label{eq:disrac}
  \dis(\calR_m)\leq b+2\diam(\alpha_G^{-1}(]v-b e,v-a e[))\leq 3\gamma_G(b)\;.
\end{equation}
Also,
\begin{equation}
  \label{eq:pirac}
  \pi_m(\calR_m^c)=\mu|_{m'}(\{x\in m':(x,x)\in\calR_m^c\})=0\;,
\end{equation}
For $A$ a Borel subset of $m$,
$$\pi(A\times m')=\mu(A\cap m')$$
and for $A'$ a Borel subset of $m'$,
$$\pi(m\times A')=\mu(A')$$
Thus,
\begin{equation}
  \label{eq:Drac}
  D(\pi;\mu|_m,\mu|_{m'})\leq \mu(\alpha_G^{-1}(]v-b e,v-a e[)\leq \tilde\gamma_G(b)\;.
\end{equation}
Inequalities \eqref{eq:disrac}, \eqref{eq:pirac} and \eqref{eq:Drac} show that for any $m\in\comp(\Frag(G,\calP_s)_{>\eps})$,
\[d_{GHP}(m,\sigma(m))\leq \alpha:=\tilde\gamma_G(b)\lor 3\gamma_G(b)\]
The same inequalities show also that for any $m'\in\comp(\Frag(G',\calP'_s)_{>\eps})$,
\[d_{GHP}(m',\sigma'(m'))\leq \alpha\]
Now, let us fix 
$$\eps= \sqrt{\alpha}+\alpha$$
so that \eqref{eq:epsshortening} is satisfied. Using the last inequality of Lemma~\ref{lemm:Ldeuxappli} with $p=1$ and $\alpha$, $\eps$, $\sigma$ and $\sigma'$ as above, we have shown  that as soon as $\calE$ holds, for any $s\in[0,t]$,
\begin{eqnarray*}
&&L_{GHP}(\Frag(\bfG,\calP_s),\Frag(\bfG^{(e,a,b)},\calP'_s))\\
  &\leq &8\alpha\frac{\mu(G)}{\sqrt{\alpha}}+16\eps\;,\\
  &=&17(\tilde\gamma_G(b)\lor 3\gamma_G(b))+(16+\mu(G))\sqrt{\tilde\gamma_G(b)\lor 3\gamma_G(b)}\;.
\end{eqnarray*}
Furthermore, 
$$\|\sizes(\Frag(\bfG,\calP_s))-\sizes(\Frag(\bfG^{(e,a,b)},\calP'_s))\|_2^2\leq 2\tilde \gamma_G(b)^2.$$
Also, for any $m$ in $\comp(\Frag(G,\calP_s)_{>\eps})$, $m$ and $\sigma(m)$ have the same surplus (recall the gluing in Definition~\ref{defi:shortening}). The same is true for $m'$ and $\sigma'(m')$. Notice also that:
$$\tilde\gamma_G(b)\leq \mu(G)\text{ and }\gamma_G(b)\leq \diam(G)\;.$$
Thus,
\begin{eqnarray*}
&&L_{2,GHP}^{surplus}(\Frag(\bfG,\calP_s),\Frag(\bfG^{(e,a,b)},\calP'_s))\\
  &\leq &\left[17(\tilde\gamma_G(b)\lor 3\gamma_G(b))+(16+\mu(G))\sqrt{\tilde\gamma_G(b)\lor 3\gamma_G(b)}\right]\lor 2\tilde \gamma_G(b)^2\;,\\
  &\leq & C(G)\sqrt{\tilde\gamma_G(b)\lor \gamma_G(b)}\;,
\end{eqnarray*}
where $C(G)$ is a positive constant depending only on $\diam(G)$ and $\mu(G)$. Finally, notice that $\calE$ has probability at least $\exp(-t(b-a))\geq 1-t(b-a)$. 
\end{proof}

Now, we shall prove Theorem~\ref{theo:FellerGraphs}. 

\begin{proof}{\it (of Theorem~\ref{theo:FellerGraphs})}

  The proofs of {\it (i)} and {\it (iii)} are completely analogous to the proofs of $(i)$ and $(iii)$ in Proposition~\ref{prop:FellerTrees}, so we leave them to the reader.
  
Let us prove  {\it (ii)}. First, we may suppose, thanks to Skorkohod representation theorem, that $\bfG^{n}$ and $\bfG$ are deterministic and that  $\bfG^{n}$ converges to $\bfG$ as $n$ goes to infinity. Thanks to Remarks~\ref{rem:deltagluing} $(iv)$ and \ref{rem:frag} $(iii)$, it is sufficient to prove Theorem~\ref{theo:FellerGraphs} when the components of $\bfG^{n}$ and $\bfG$ are genuine metric spaces, i.e $\RR$-graphs.

The argument at the beginning of the proof of Proposition~\ref{prop:FellerTrees} shows that it is sufficient to prove the result when $\bfG^{n}$ and $\bfG$ have a single component.  Let $\bfG=(G,d,\mu)$ be a measured $\RR$-graph and let $\eps>0$. We want to show that $\Frag(\bfG^{n},t)$ converges in distribution to $\Frag(\bfG^{n},t)$ when $\bfG^{n}$ is a sequence of $\RR$-graphs which converges to $\bfG$ while having the same surplus. 

Let $\delta<\delta(\eps,G)$ be given by Lemma~\ref{lemm:overlay} and let $\bfG'$ be such that $d_{GHP}(\bfG,\bfG')<\delta$ (we will take $\delta$ small enough later). Thus, there is a correspondence $\calR_0\in C(G,G')$ and a measure $\pi_0\in M(G,G')$ such that:
$$\dis(\calR_0)\lor \pi_0(\calR_0^c)\lor D(\pi_0;\mu,\mu')< \delta$$
Lemma~\ref{lemm:overlay} shows that there exists an $\eps$-overlay $\calR\in C(G,G')$ containing $\calR_0$. Let us denote by $\chi$ the multigraph isomorphism from $\ker(G)$ to $\ker(G')$ given by this overlay. For any edge $e\in \ker(G)$, $|\ell(e)-\ell'(\chi(e))|<\eps$. 

We define two graphs $\tilde \bfG$ and $\tilde \bfG'$ obtained from $\bfG$ and $\bfG'$ as follows. For each oriented edge $e=(u,v)\in\ker(G)$, denoting $(u',v')=\chi(e)$ and $\eta_e:=|\ell(e)-\ell'(\chi(e))|$, which is less than $\eps$,
\begin{itemize}
\item if $\ell(e)$ is smaller than $\ell'(e')$, we replace $G'$ by its $(6\eps-\eta_e,6\eps)$-shortening along $e'$ (cf. Definition \ref{defi:shortening}),
\item if $\ell'(e')$ is smaller than $\ell(e)$, we replace $G$ by its  $(6\eps-\eta_e,6\eps)$-shortening along $e$.
\end{itemize}
Let us denote by $(\tilde G,\tilde d)$ and $(\tilde G',\tilde d')$ the resulting $\RR$-graphs, let $\tilde\mu:=\mu|_{\tilde G}$, $\tilde\mu':=\mu'|_{\tilde G'}$ and define $\tilde\bfG:=(\tilde G,\tilde d,\tilde \mu)$, $\tilde\bfG':=(\tilde G',\tilde d',\tilde \mu')$.

Recalling the notation in Lemma~\ref{lem:gammaG}, let 
$$\kappa:=\gamma_G(11\eps)+12\eps$$
and define $\calR_1$ the $\kappa$-enlargement of $\calR$. We will show that 
\begin{equation}
\label{eq:corresGGtilde}
\tilde G\text{ and }\tilde G'\text{ are in correspondence through }\calR_1.
\end{equation} 
If $x\in \tilde G$ and $(x,x')\in\calR$ with $x'\not\in \tilde G'$, then, $x'\in\alpha_{G'}^{-1}(]v'-6\eps e',v'-(6\eps-\eta_e) e'[)$ for some edge $e'=(u',v')$ of $\ker(G')$. Lemma~\ref{lemm:corresalpha} shows that
\begin{equation}
\label{eq:x'x}
 0<6\eps-\eta_e -5\dis(\calR)\leq d(\alpha_G(x),v)\leq 6\eps+5\dis(\calR)\leq 11\eps\;.
\end{equation}
and thus
$$d(x,v)\leq \gamma_G(11\eps)+11\eps\;.$$
Thus,
\begin{equation}
\label{eq:gammaGgammaG'}
d'(x',v')\leq  \gamma_G(11\eps)+12\eps= \kappa\;.
\end{equation}
This shows that $(x,v')\in\calR_1$.

Now, suppose $x'\in\tilde G'$ and $(x,x')\in\calR$ with $x\not\in \tilde G$. Then, $x\in\alpha_{G}^{-1}(]v-6\eps e,v-(6\eps-\eta_e) '[)$ for some edge $e=(u,v)$ of $\ker(G)$ and $\eta<\eps$. Notice that:
$$d(x,v)\leq \gamma_G(6\eps)+6\eps\;,$$
and
$$d'(x',v')\leq d(x,v)+\dis(\calR)\leq \gamma_G(6\eps)+7\eps\leq \kappa\;.$$
Thus $(v,x')\in\calR_1$. This ends the proof of \eqref{eq:corresGGtilde}.

Notice that
$$\dis(\calR_1)\leq \eps+4\kappa\;.$$
Let $\calR_2:=\calR_1|_{\tilde G\times \tilde G'}\in C(\tilde G,\tilde G')$. Let $K$ be the number of edges in $\ker(G)$. Notice that
$$\forall (x,y)\in \tilde G,\;|d(x,y)-\tilde d(x,y)|< K\eps$$
and
$$\forall (x',y')\in \tilde G',\;|d'(x',y')-\tilde d'(x',y')|< K\eps\;.$$
Thus,
\begin{equation}
\label{eq:disR2}
\dis(\calR_2)\leq 2K\eps +\dis(\calR_1)<(2K+49)\eps+4\gamma_G(11\eps)\;.
\end{equation}

Clearly, there exists a homeomorphism $\psi$ from $\core(\tilde G)$ to $\core(\tilde G')$ which preserves the length-measure. For each oriented edge $e=(u,v)\in\ker(G)$, denoting $(u',v')=\chi(e)$, $\psi$ satisfies $\psi(v)=v'$. Furthermore, since $e$ and $e'$ are in correspondence through the overlay $\calR$, we have, for each $x\in [u,v]$, that there exists $x'\in [u',v']$ such that $(x,x')\in\calR$ and:
$$|d(x,u)-d'(x',u')|<\eps\;.$$
If furthermore $x\in \core(\tilde G)$, we know that $d(x,u)=d'(\psi(x),u')$, so
$$|d'(\psi(x),u')-d'(x',u')|<\eps\;.$$
Since $x'$ and $\psi(x)$ belong to $[u',v']$, 
$$d'(\psi(x),x')=|d'(\psi(x),u')-d'(x',u')|<\eps\;,$$
which shows that for every $x\in\core(\tilde G)$, 
\begin{equation}
\label{eq:psicorres}
(x,\psi(x))\text{ belongs to }\calR_2\;,
\end{equation}
the restriction to $\tilde G\times \tilde G'$ of the $\kappa$-enlargement of $\calR$.

Now, let $\pi:=\pi_0|_{\tilde G\times \tilde G'}\in M(\tilde G,\tilde G')$. First,
\begin{equation}
\label{eq:piR2c}
\pi(\calR_2^c)=\pi(\calR_1^c)\leq \pi_0(\calR_0^c)<\eps\;.
\end{equation}

Then, 
\begin{equation}
\label{eq:Dpimumu'}  D(\pi;\tilde \mu,\tilde \mu')\leq 2D(\pi_0; \mu, \mu')+\mu (G \setminus\tilde G )+\mu'(G'\setminus\tilde G')\;.
\end{equation}
Now, recall that
$$\tilde\gamma_G(\eps):=\sum_{e=(u,v)\in\ker(G)}\mu(\alpha_G^{-1}(]v-\eps e,v[))$$
which goes to zero as $\eps$ goes to zero. We have
$$\mu(G\setminus \tilde G)\leq \tilde\gamma_G(6\eps)\;.$$
Furthermore, recall inequality~\eqref{eq:x'x} which shows that if $x'\in G'\setminus \tilde G'$, then for every $x\in G$ such that $(x,x')\in\calR$,
$$x\in\bigcup_{e=(u,v)\in\ker(G)}\alpha_G^{-1}(]v-11\eps e,v[)\;.$$
Thus,
\begin{eqnarray}
\nonumber \mu'(G'\setminus\tilde G')&\leq &\pi_0(G\times (G'\setminus\tilde G'))+D(\pi_0;\mu,\mu')\\
\nonumber  &\leq &\pi_0((G\times (G'\setminus\tilde G'))\cap\calR)+\pi_0(\calR^c)+\eps\\
\nonumber  &\leq &\pi_0\left(\bigcup_{e=(u,v)\in\ker(G)}\alpha_G^{-1}(]v-11\eps e,v[)\times G'\right)+2\eps\\
\nonumber  &\leq &\mu\left(\bigcup_{e=(u,v)\in\ker(G)}\alpha_G^{-1}(]v-11\eps e,v[)\right)+D(\pi_0;\mu,\mu')+2\eps\\
\label{eq:tildegammaGG'}&\leq & \tilde\gamma_G(11\eps)+3\eps\;.
\end{eqnarray}
Thus, using \eqref{eq:Dpimumu'},
\begin{equation}
\label{eq:piD}
D(\pi;\tilde \mu,\tilde \mu')\leq 5\eps+2\tilde\gamma_G(11\eps)\;.
\end{equation}

Gathering \eqref{eq:psicorres}, \eqref{eq:disR2}, \eqref{eq:piR2c} and \eqref{eq:piD} shows that one may apply Lemma~\ref{lemm:FellerRootCore}, in the sense that there is a function $f_G(\eps)$ going to zero as $\eps$ goes to zero such that the Lévy-Prokhorov distance (for the Skorokhod topology associated to $L_{2,GHP}^{surplus}$) between the distributions of $(\Frag(\tilde G,s))_{s\in [0,t]}$ and $(\Frag(\tilde G',s))_{s\in[0,t]}$ is less than $f_G(\eps)$.

On the other hand, let $v'\in\ker(G')$ and $x'\in\alpha_{G'}^{-1}(]v-\eps,v[)$. Let $v\in\ker(G)$ (resp. $x\in G$) be such that $(v,v')\in\calR$ (resp. $(x,x')\in\calR$). Then,
\[d'(x',v') \leq  d(x,v)+\eps\;.\]
But, using Lemma~\ref{lemm:corresalpha},
\[d(\alpha_G(x),v)\leq d'(\alpha_{G'}(x'),v')+5\eps\leq 6\eps\;,\]
which implies
\[d(x,v)\leq \gamma_G(6\eps)\;.\]
Thus,
\[d'(x',v') \leq \gamma_G(6\eps)+\eps\;,\]
which shows that
$$\gamma_{G'}(\eps)\leq \gamma_G(6\eps)+\eps\;.$$
Inequality \eqref{eq:tildegammaGG'} shows that
$$\tilde\gamma_{G'}(\eps)\leq \tilde\gamma_G(11 \eps)+3\eps\;.$$
Then, Lemma~\ref{lemm:shortening} shows that there is a function $f_G(\eps)$ going to zero as $\eps$ goes to zero such that the Lévy-Prokhorov distance between the distributions of $(\Frag(G,s))_{s\in [0,t]}$ and $(\Frag(\tilde G,s))_{s\in [0,t]}$ is less than $f_G(\eps)$ and  the Lévy-Prokhorov distance  (for the Skorokhod topology associated to $L_{2,GHP}^{surplus}$) between the distributions of $(\Frag(G',s))_{s\in [0,t]}$ and $\Frag(\tilde G',s))_{s\in [0,t]}$ is less than $f_G(\eps)$. This ends the proof of {\it (ii)} (through  Lemma~\ref{lemm:compactcv} and inequality~\eqref{eq:dSkdc}).

\end{proof}

\subsection{Application to Erd\H{o}s-Rényi random graphs: proofs of Theorem~\ref{theo:cvfragGnp} and Proposition \ref{prop:retournement} }

Let us first compare the discrete fragmentation process and the continuous one.  Let $\calP^-$ be a Poisson process driving the discrete fragmentation on $G^{n}:=G(n,p(\lambda,n))$. Recall that $N^-(G^{n},\calP^-_t)$ is the state of this process at time $t$, seen as an element of $\calN^{graph}_2$. Let $\calQ^-$ be a Poisson process of intensity $\ell_n\otimes\leb_{\RR^+}$ on $K_n\times\RR^+$ where $K_n$ is the complete graph on $n$ vertices seen as an $\RR$-graph where the edge lengths are $\delta_n=n^{-1/3}$ and $\ell_n$ is its length measure. Then, one may suppose that $\calP^-$ is obtained as follows:
$$\calP^-=\{(e,t):\exists x\in e,\;(x,t)\in\calQ^-\}\;.$$
Then, for any $t$, $N^-(G^{n},\calP^-_t))$ is at $L_{2,GHP}$-distance at most $n^{-1/3}$ from $\Frag(G^{n},\calQ^-_t)$ (cf. for instance \cite[Propositions~3.4]{AdBrGoMi2017}). Recall that by Theorem~\ref{theo:cvGnp}, $\overline{\calG}_{n,\lambda}$ (which is $G^{n}$ with edge length $\delta_n$ and vertex weights $n^{-2/3}$) converges in distribution to $\calG_\lambda$ for $L^{surplus}_{2,GHP}$. Thus Theorem~\ref{theo:FellerGraphs} implies that $(\Frag(G^{n},\calQ^-_t))_{t\geq 0}$, and thus $(N^+(G^{n},\calP^-_t)))_{t\geq 0}$, converges to $(\Frag(\calG_\lambda,t))_{t\geq 0}$ as $n$ goes to infinity (in the Skorokhod topology associated to $L^{surplus}_{2,GHP}$). This shows Theorem~\ref{theo:cvfragGnp}.

We are now able to prove Proposition~\ref{prop:retournement}. Take $\calP_t^+$ of intensity $\gamma=n^{-4/3}$. Notice that the states of the edges are independent and identically distributed in $(\calG(n,p),N^+(\calG(n,p),\calP_t^+))$. Let $(X,Y)$ be the joint distribution of the state of one edge in $(\calG(n,p),N^+(\calG(n,p),\calP_t^+))$. Denoting by $0$ the state ``absent'' and $1$ the state ``present'', it is easy to compute this distribution:
$$\begin{array}{cc}
\PP((X,Y)=(0,0))=(1-p)e^{-\gamma t} & \PP((X,Y)=(0,1))=(1-p)(1-e^{-\gamma t})\\
\PP((X,Y)=(1,0))=0 & \PP((X,Y)=(1,1))=p
\end{array}$$
Now, take $\calP_t^-$ of intensity $\mu=n^{-1/3}$ and let $(X',Y')$ be the joint distribution of the state of one edge in $(N^-(G(n,p'),\calP_{t'}^-),G(n,p'))$. Then,
$$\begin{array}{cc}
\PP((X',Y')=(0,0))=(1-p')& \PP((X,Y)=(0,1))=p'(1-e^{-\mu t'})\\
\PP((X',Y')=(1,0))=0 & \PP((X',Y')=(1,1))=p'e^{-\mu t'}
\end{array}$$
Thus, if one chooses 
$$t=\frac{1}{\gamma}\ln\frac{1-p}{1-p'}\quad\text{ and }t'=\frac{1}{\mu}\ln\frac{p'}{p}\;,$$
then $(\calG(n,p),N^+(\calG(n,p),\calP_t^+))$ and $(N^-(G(n,p'),\calP_{t'}^-),G(n,p'))$ have the same distribution. Now, take $p=p(\lambda,n)$, $p'=p(\lambda+s,n)$. We have:
$$t=n^{4/3}\ln\left(1+\frac{s}{n^{4/3}(1-p')}\right)\xrightarrow[n\rightarrow \infty]{}s\;.$$
We consider that $\calG(n,p)$ is equipped with edge lengths $n^{-1/3}$ and vertex weight $n^{-2/3}$. Thus Theorem~\ref{theo:cvcoalGnp} shows that $(\calG(n,p),N^+(\calG(n,p),\calP_t^+))$ converges in distribution to $(\calG_\lambda,\Coal_0(\calG_\lambda,s))$. Also,
$$t'=n^{1/3}\ln\frac{1+\frac{\lambda+s}{n^{1/3}}}{1+\frac{\lambda}{n^{1/3}}}\xrightarrow[n\rightarrow \infty]{}s$$
thus Theorem~\ref{theo:FellerGraphs} shows that $(N^-(G(n,p'),\calP_{t'}^-),G(n,p'))$ converges in distribution to $(\Frag(\calG_{\lambda+s},s),\calG_{\lambda+s}))$. Thus $(\calG_\lambda,\Coal_0(\calG_\lambda,s))$ and $(\Frag(\calG_{\lambda+s},s),\calG_{\lambda+s}))$ have the same distribution. This ends the proof of Proposition~\ref{prop:retournement}.

Notice a curious fact: in \cite[Theorem~3]{AldousPitman98a}, it is shown that the sizes of the components of a fragmentation on the CRT are the time-reversal (after an exponential time-change) of the standard \emph{additive} coalescent. It would be intersting to make a direct link between additive and multiplicative coalescent in the context of fragmentation on $\calG_\lambda$.

\section{Combining fragmentation and coalescence: dynamical percolation}
\label{sec:dynperc}

\subsection{Almost Feller Property: proof of Theorem~\ref{theo:almostFellerCoalFrag}}
In this section, we prove Theorem~\ref{theo:almostFellerCoalFrag}. The following lemma is a simple variation on the proof of \eqref{eq:supdiamlimite}. 
\begin{lemm}
  \label{lemm:suplengthstable}
  Let $\bfX^{n}=(X^{n},d^{n},\mu^{n})$, $n\geq 0$ be a sequence of random m.s-m.s in $\calN_2^{graph}$ and $(\delta^{n})_{n\geq 0}$ be a sequence of non-negative real numbers. Suppose that: 
\begin{enumerate}[(i)]
\item $(\bfX^{n})$ converges in distribution (for $L_{2,GHP}$) to $\bfX^{\infty}=(X^{\infty},d^{\infty},\mu^{\infty})$ as $n$ goes to infinity
\item $\delta^n\xrightarrow[n\rightarrow\infty]{}0$
\item For any $\alpha>0$ and any $T>0$, $$\limsup_{n\in\NN} \PP(\suplength(\Coal_{\delta^{n}}(\bfX^{n}_{\leq \eps},T))>\alpha)\xrightarrow[\eps\rightarrow 0]{} 0$$
\end{enumerate}
Then, for any $\alpha>0$ and any $T>0$,
$$\PP(\suplength(\Coal_{0}(\bfX^{\infty}_{\leq \eps},T))>\alpha)\xrightarrow[\eps\rightarrow 0]{} 0\;.$$
\end{lemm}
\begin{proof}
The situation is simpler than in the proof of \eqref{eq:supdiamlimite}, since $\suplength$ is non-decreasing under coalesence. Using the notation of the proof of \eqref{eq:supdiamlimite}, 
\begin{eqnarray*}
  &&\PP(\suplength(\Coal_{0}(\bfX^{\infty}_{\leq \eps_m},T))>\alpha)\\
&=&\lim_{p\rightarrow \infty}\PP(\suplength(\Coal_{0}(\bfX^{\infty}_{m,p},T))>\alpha)
\end{eqnarray*}
Now, Proposition~\ref{prop:FellerN1} implies that $(\Coal_{\delta^{n}}(\bfX^{n}_{m,p},T)$ converges in distribution to $(\Coal_{0}(\bfX^{\infty}_{m,p},T)$ for any $m\leq p$.  Thus, for any $m\leq p$,
\begin{eqnarray*}
&&\PP(\suplength(\Coal_{0}(\bfX^{\infty}_{m,p},T))>\alpha)\\
&\leq & \limsup_{n\rightarrow\infty}\PP(\suplength(\Coal_{\delta^{n}}(\bfX^{n}_{m,p},T))>\alpha)\\
&\leq &\limsup_{n\rightarrow\infty}\PP(\suplength(\Coal_{\delta^{n}}(\bfX^{n}_{\leq \eps_m},T))>\alpha),
\end{eqnarray*}
which goes to zero when $m$ goes to infinity. 
\end{proof}
Now, we are able to prove Theorem~\ref{theo:almostFellerCoalFrag}.

Notice that the strong Markov property was already noticed (see the remark after Definition~\ref{defi:dynperc}). The fact that trajectories lie in $\calS^{length}$ is a consequence of Lemmas~\ref{lemm:calSdansN} and~\ref{lemm:calSstable}. Thus, to prove $(i)$, we only need to prove that the trajectories are càdlàg (almost surely). This will be done in the course of proving point $(ii)$.  

We will reduce the problem to $\calN_1^{graph}$ using a variation on the proof of Theorem~\ref{theo:AlmostFellercoal}. Let us study first $\Frag(\Coal_{\delta^n}(\bfX^{n},\calP^+_t),\calP^-_t)$ with $\calP^+$ and $\calP^-$ as in Definitions~\ref{defi:coal} and \ref{defi:frag}. Let us fix $\eps\in]0,1/2[$ and $0\leq t \leq T$. Any component of size at least $\eps$ in $\Frag(\Coal_{\delta^n}(\bfX^{n},\calP^+_t),\calP^-_t)$ has to belong to a component of size at least $\eps$ in $\Coal_{\delta^n}(\bfX^{n},\calP^+_t)$. Let $x^n:=\sizes(\bfX^{n})$ for $n\in\NNbar$. 

 As in the proof of Theorem~\ref{theo:AlmostFellercoal},  we obtain that there exists $ K(\eps)$, $\eps_1\in ]0,\eps[$ and  $\eps_2\in]0,\eps_1[$ such that for every $n\in\NNbar$, with probability larger than $1-\eps$ the event $\calA_n$ holds, where $\calA_n$ is the event that points (a), (b) and (c) of Corollary~\ref{coro:picture_structure} hold for any $t\in[0,T]$ and $S(x^n,T)\leq K(\eps)$.

Let us place ourselves on $\calA_n$. Then, for a significant component at time $t$, notice that fragmentation on the hanging trees of components does change neither the mass neither the distance in the heart of a component. Thus, the same proof as that of Theorem~\ref{theo:AlmostFellercoal} shows that on $\calA_n$, we have for every time $t\leq T$ and every $\eps'_2\leq \eps_2$:
\begin{eqnarray*}
&&L_{GHP}(\Frag(\Coal_{\delta^n}(\bfX^n,\calP^+_t),\calP^-_t),\Frag(\Coal_{\delta^n}(\bfX^n_{> \eps'_2},\calP^+_t),\calP^-_t))\\
  &\leq & 17(\delta^n+\supdiam(\Frag(\Coal_{\delta^n}(\bfX^n_{\leq \eps_1},\calP^+_T),\calP^-_t))+\eps_1)\left(1+\frac{8K(\eps)}{\eps^2}\right)\\
&&+16\eps\;.
\end{eqnarray*}
A slight difference occurs here: 
$$t\mapsto \supdiam(\Frag(\Coal_{\delta^n}(\bfX^n_{\leq \eps_1},\calP^+_T),\calP^-_t))$$
is not necessarily nonincreasing. However, the supremum of the lengths of injective paths clearly decreases (non-strictly) under fragmentation. Thus, on $\calA_n$,
\begin{eqnarray*}
&&L_{GHP}(\Frag(\Coal_{\delta^n}(\bfX^n,\calP^+_t),\calP^-_t),\Frag(\Coal_{\delta^n}(\bfX^n_{> \eps'_2},\calP^+_t),\calP^-_t))\\
  &\leq & 17(\delta^n+\suplength(\Coal_{\delta^n}(\bfX^n_{\leq \eps_1},\calP^+_T))+\eps_1)\left(1+\frac{8K(\eps)}{\eps^2}\right)+16\eps\;.
\end{eqnarray*}
Let $V=\comp(\Frag(X^n,\calP_t^-))$ and $W=\comp(\Frag(X^n_{>\eps'_2},\calP_t^-))\subset V$. Let $E'$ denote the set of pairs $(i,j)$ in $V^2$ such that $i\sim j$ if and only if $i$ and $j$ are at finite distance in $\Frag(\Coal_{\delta^n}(\bfX^n,\calP^+_t),\calP^-_t)$. Let $E$ denote the set of pairs $(i,j)$ in  $V^2$ such that $i\sim j$ if and only if $i$ and $j$ are at finite distance in $\Coal_{\delta^n}(\bfX^n,\calP^+_t)$. Define:
$$x^n(t):=\sizes(\Frag(\Coal_{\delta^n}(\bfX^n,\calP^+_t),\calP^-_t))$$
and
$$x^n_{>\eps}(t):=\sizes(\Frag(\Coal_{\delta^n}(\bfX^n_{> \eps},\calP^+_t),\calP^-_t))\;.$$
Lemma~\ref{lemm:Aldous17} shows that
\begin{eqnarray*}
&&\|x^{n}(t)-x^{n}_{>\eps'_2}(t)\|_2^2\\
&\leq& \|x^{n}(t)\|_2^2-\|x^{n}_{>\eps'_2}(t)\|_2^2
\end{eqnarray*}
then, we can use Lemma~\ref{lemm:pourSkorL2}, since \eqref{coro:nocycle} of Corollary~\ref{coro:picture_structure} holds on $\calA_n$:
\begin{eqnarray*}
&&\|x^{n}(t)-x^{n}_{>\eps'_2}(t)\|_2^2\\
&\leq& S(x^{n},t)-S(x^{n}_{>\eps_2},t)\\
&\leq &2\eps_1^2\leq \eps
\end{eqnarray*}
since \eqref{coro:S} of Corollary~\ref{coro:picture_structure} holds on $\calA_n$. Now, let us take $\delta_n=0$. Define:
$$\bfX^n(t):=\Frag(\Coal_{0}(\bfX^n,\calP^+_t),\calP^-_t)=\CoalFrag(\bfX^n,t)$$
and
$$\bfX^n_{>\eps_1}(t):=\Frag(\Coal_{0}(\bfX^n_{>\eps_1},\calP^+_t),\calP^-_t)=\CoalFrag(\bfX^{n}_{>\eps_1},t))\;.$$
Notice that we used the commutation relation guaranteed by Lemma~\ref{lemm:calSdansN}. Using the hypothesis~\eqref{eq:suplengthunif} on $\suplength$ and Lemma~\ref{lemm:suplengthstable}, we conclude that for any $\eps>0$,
\begin{equation}
\label{eq:finiteapproxCoalFrag}
\lim_{\eps_1\rightarrow 0}\sup_{n\in\NNbar}\PP^*[\sup_{t\in[0,T]}L_{2,GHP}(\bfX^n(t),\bfX^n_{>\eps_1}(t))>\eps]=0\;.
\end{equation}
Now, let us prove that the trajectories of $\CoalFrag(\bfX^\infty,\cdot)$ are càdlàg. Let $\bfY^n:=R_{\frac{1}{n}}(X^\infty_{>\frac{1}{n}})$. $\CoalFrag(\bfY^n,\cdot)$ is càdlàg (it  has a finite number of jumps on any bounded interval). Now, Lemma~\ref{lemm:approxReta} applied to $X^\infty_{>\frac{1}{n}}$ and equation~\eqref{eq:finiteapproxCoalFrag} applied to $X^n:=X^\infty$ show that the hypotheses of Lemma~\ref{lemm:cvcadlag} are satisfied for $\omega^n=\CoalFrag(\bfY^n,\cdot)$ and $\omega^\infty=\CoalFrag(\bfX^\infty,\cdot)$. This proves that the trajectories of $\CoalFrag(\bfX^\infty,\cdot)$ are càdlàg and ends the proof of $(i)$.

Now, let us prove point $(ii)$. Equation~\eqref{eq:finiteapproxCoalFrag} shows that it is sufficient to show the proposition for $\bfX^n$ converging to $\bfX$ in $L_{1,GHP}$ with $\bfX^n$ and $\bfX$ being  m.s-m.s with a finite number of components which are $\RR$-graphs. We shall only sketch the proof, since it is a variation on the arguments of the proofs of Propositions~\ref{prop:FellerN1} and Theorem~\ref{theo:FellerGraphs}. For any $n$ large enough, the proof of Theorem~\ref{theo:FellerGraphs} shows that one may couple a Poisson random set $\calP^{-,n}$ on $X^n\times \RR^+$ with intensity measure $\ell_{X^n}\otimes \leb_{\RR^+}$ with a Poisson random set $\calP^-$ on $X\times \RR^+$ with intensity $\ell_X\otimes \leb_{\RR^+}$ and one may find $\pi^n\in M(X,X^n)$ and $\calR^n\in C(X,X^n)$ such that there is an event $\calE_n$ in the $\sigma$-algebra of $(\calP^{-,n}_t,\calP^-_t)$ and a sequence $\eps_n$  such that:
\begin{enumerate}[(i)]
\item $\PP(\calE_n^c)\leq \eps_n$
\item $\eps_n\xrightarrow[n\rightarrow\infty]{}0$
\item on $\calE_n$, for any $s\leq t$, $\calR^n\cap(X\setminus\calP^-_s)\times (X^n\setminus \calP^{-,n}_s)\in C(X\setminus\calP^-_s,X^n\setminus \calP^{-,n}_s)$ and 
$$D(\pi|_{(X\setminus\calP^-_s)\times (X^n\setminus \calP^{-,n}_s)};\mu|_{X\setminus\calP^-_s},\mu_{X^n\setminus \calP^{-,n}_s})\lor \pi^n((\calR^n)^c)\lor \dis_s(\calR^n)\leq\eps_n$$
where $\dis_s(\calR^n)$ is the distortion of $\calR^n$ as a correspondence between the semi-metric spaces $\Frag(X,\calP^-_s)$ and $\Frag(X^n,\calP_s^{-,n})$
\item  on $\calE_n$, for any $s\leq t$, $\|\sizes(\Frag(\bfX,\calP_s^-))-\sizes(\Frag(\bfX^n,\calP_s^{-,n}))\|_1\leq \eps_n$.
\end{enumerate}
Then, one may use the proofs of Lemma~\ref{lemm:couplageN1} and Proposition~\ref{prop:FellerN1} to couple  a Poisson random set $\calP^{+,n}$ on $(X^n)^2\times \RR^+$ with intensity measure $\frac{1}{2}(\mu^n)^\otimes\otimes \leb_{\RR^+}$ with a Poisson ransom set $\calP^+$ on $(X)^2\times \RR^+$ with intensity $\frac{1}{2}(\mu)^\otimes\otimes \leb_{\RR^+}$ in such a way that there is an event $\calE'_n$, a sequence $\eps'_n$ such that:
\begin{enumerate}[(i)]
\item $\PP((\calE'_n)^c)\leq \eps'_n$
\item $\eps'_n\xrightarrow[n\rightarrow\infty]{}0$
\item on $\calE'_n$, for any $s\leq t$,
$$D(\pi|_{(X\setminus\calP^-_s)\times (X^n\setminus \calP^{-,n}_s)};\mu|_{X\setminus\calP^-_s},\mu^n|_{X^n\setminus \calP^{-,n}_s})\lor \pi^n((\calR^n)^c)\lor \dis'_s(\calR^n)\leq\eps'_n$$
where $\dis'_s(\calR^n)$ is the distortion of $\calR^n$ as a correspondence between the semi-metric spaces $\Coal_0(\Frag(X,\calP^-_s),\calP^+_s)$ and $\Coal_0(\Frag(X^n,\calP_s^{-,n}),\calP^{+,n}_s)$,
\item on $\calE'_n$, for any $s\leq t$, the multigraphs $\MG(\Frag(\bfX,\calP_s^{-}),\calP_s^+)$ and $\MG(\Frag(\bfX^n,\calP_s^{-,n}),\calP_s^{+,n})$ are the same.
\end{enumerate}
Thanks to the properties on the multigraphs above, on $\calE_n\cap\calE'_n$, we get that for any $s\leq t$,
\begin{eqnarray*}
  &&\|\sizes(\Coal(\Frag(\bfX,\calP_s^-),\calP_s^+))-\sizes(\Coal(\Frag(\bfX^n,\calP_s^{-,n}),\calP_s^+))\|_1\\
  &\leq& \|\sizes(\Frag(\bfX,\calP_s^-))-\sizes(\Frag(\bfX^n,\calP_s^{-,n}))\|_1\\
  &\leq &\eps_n\;.
\end{eqnarray*}
Using  Lemma~\ref{lemm:3GHP}, this ends the proof of the convergence in the sense of $L_{1,GHP}$, and thus the proof of point $(i)$.

Finally, $(iii)$ is a direct consequence of Theorem~\ref{theo:AlmostFellercoal} $(iii)$ and Theorem~\ref{theo:FellerGraphs}  $(iii)$. This ends the proof of Theorem~\ref{theo:almostFellerCoalFrag}.

\subsection{Application to Erd\H{o}s-Rényi random graphs}
Now, we want to prove Theorem~\ref{theo:cvpercodynGnp} and Theorem~\ref{theo:existencecoalfrag}. Intuitively, the dynamical percolation process on the complete graph $K_n$ should be very close to the process $\CoalFrag(K_n,.)$, but such a statement needs some care, essentially because $N^+$ and $N^-$ do not commute: some pairs of vertices might be affected by the two Poisson processes $\calP^+$ and $\calP^-$ in a time interval $[0,T]$. Furthermore, the typical number of such edges is of order $n^{1/3}$. It turns out that these edges will not be important for the $L_{GHP}$-metric, but this issue requires us to adapt the proof of Theorem~\ref{theo:almostFellerCoalFrag}.

\begin{proof} (of Theorem~\ref{theo:cvpercodynGnp}). Let $G^{\infty}=\Glambda$. Let $p=p(\lambda,n)$, let $\bfG^{n}$ be the graph $\calG(n,p)$ seen as a measured $\RR$-graph, with edge-lengths $\delta_n:=(1-p)n^{-1/3}\sim n^{-1/3}$ and measure the counting measure times $\sqrt{pn^{-1/3}}\sim n^{-2/3}$. Let  $\calP^+$ (of intensity $pn^{-1/3}$) and $\calP^-$ (of intensity $(1-p)n^{-1/3}$) be the two Poisson processes driving the dynamical percolation on $G^{n}$. Let us write 
$$G^{n}(t):=N(G^{n},(\calP^+,\calP^-)_t)$$
and
$$G^{n}_{>\eps_1}(t):=N(G^{n}_{>\eps_1},(\calP^+,\calP^-)_t)\;.$$
for the state of this process at time $t$, seen as a member of $\calN_2^{graph}$. Let us fix $\eps>0$ and $0\leq t \leq T$. Any component of size at least $\eps$ in $N(G^{n},(\calP^+,\calP^-)_t)$ has to belong to a component of size at least $\eps$ in $N(G^{n},(\calP^+,\emptyset)_t)$, which is nothing else but $\Coal_{\delta_n}(G^{n},\calP^+_t)$. Now, we claim that
\begin{equation}
\label{eq:hypsuplength}
\limsup_{n\in\NN} \PP(\suplength(\Coal_{\delta^{n}}(\bfG^{n}_{\leq \eps},T))>\alpha)\xrightarrow[\eps\rightarrow 0]{} 0\;.
\end{equation}
Indeed, if $G$ is a discrete graph with diameter $D$ and surplus bounded from above by $s$, 
$$\suplength(G)\leq 2D(1+s)\;.$$
Thus, \eqref{eq:hypsuplength} is a consequence of \eqref{eq:supdiamcalG_n} and the fact that the maximal surplus in $\bfG^{n}$ forms a tight sequence (see for instance \cite[sections 13 and 14]{JansonKnuthLuczakPittel93}). Then, \eqref{eq:hypsuplength} and Lemma \ref{lemm:suplengthstable} show that for any $\alpha>0$ and $T>0$,
$$\PP(\suplength(\Coal_{0}(\bfG^{\infty}_{\leq \eps},T))>\alpha)\xrightarrow[\eps\rightarrow 0]{} 0\;.$$
The arguments leading to \eqref{eq:finiteapproxCoalFrag} show that: 
$$
\lim_{\eps_1\rightarrow 0}\limsup_{n\in\NN}\PP[\sup_{t\in[0,T]}L_{2,GHP}(G^{n}(t),G^{n}_{>\eps_1}(t))>\eps]=0\;. 
$$
and
$$
\lim_{\eps_1\rightarrow 0}\PP[\sup_{t\in[0,T]}L_{2,GHP}(\CoalFrag(G^\infty,t),\CoalFrag(G^\infty_{>\eps_1},t))>\eps]=0\;.
$$
Thus, it is sufficient to show that for any $\eps_1>0$, $(G^{n}_{>\eps_1}(t))_{t\geq 0}$ converges to $(\CoalFrag(G^\infty_{>\eps_1},t))_{t\geq 0}$ in the Skorokhod topology associated to $L_{2,GHP}$. Let $Y_n$ denote the number of discrete coalescence events of $\calP^+_T$ occurring on $G^{n}_{>\eps_1}$. Since the masses of  $G^{n}_{>\eps_1}$ form a tight sequence, $(Y_n)$ is a tight sequence. Since $\delta_n$ goes to zero, the probability that $\calP^-_T$ touches an edge from $\calP^+_T$ in $G^{n}_{>\eps_1}$ goes to zero as $n$ goes to infinity. Thus, with probability going to one, for any $t\in[0,T]$ $G^{n}_{>\eps_1}(t)=\tilde G^{n}_{>\eps_1}(t)$ where
$$\tilde G^{n}_{>\eps_1}(t)=N(\Coal_{\delta_n}(G^{n}_{>\eps_1},\calP^+_t),(\emptyset,\calP^-)_t)\;.$$
Furthermore, since $Y_n$ is a tight sequence and $\delta_n$ goes to zero, 
$$\sup_{t\in[0,T]}L_{2,GHP}(\tilde G^{n}_{>\eps_1}(t),N(\Coal_{0}(G^{n}_{>\eps_1},\calP^+_t),(\emptyset,\calP^-)_t))\xrightarrow[n\rightarrow \infty]{\PP}0\;.$$
Let $\calQ^-$ be a Poisson process of intensity $\ell_n\otimes\leb_{\RR^+}$ on $K_n\times\RR^+$ where $K_n$ is the complete graph on $n$ vertices seen as an $\RR$-graph where the edge lengths are $\delta_n$ and $\ell_n$ is its length measure. Then, one may suppose that $\calP^-$ is obtained as follows:
$$\calP^-=\{(e,t):\exists x\in e,\;(x,t)\in\calQ^-\}\;.$$
Then, for any $t$, $N(\Coal_{0}(G^{n}_{>\eps_1},\calP^+_t),(\emptyset,\calP^-)_t)$ is at $L_{2,GHP}$-distance at most $\delta_n$ from $\Frag(\Coal_{0}(G^{n}_{>\eps_1},\calP^+_t),\calQ^-_t)$ (cf. for instance \cite[Proposition~3.4]{AdBrGoMi2017}). Altogether, we get:
$$\sup_{t\in[0,T]}L_{2,GHP}(G^{n}_{>\eps_1}(t),\Frag(\Coal_{0}(G^{n}_{>\eps_1},\calP^+_t),\calQ^-_t))\xrightarrow[n\rightarrow \infty]{\PP}0\;,$$
and $(\Frag(\Coal_{0}(G^{n}_{>\eps_1},\calP^+_t),\calQ^-_t))_{t\geq 0}$ is distributed as $\CoalFrag(G^{n}_{>\eps_1},t)_{t\geq 0}$.  Now Theorem \ref{theo:almostFellerCoalFrag} shows that the sequence of processes $\CoalFrag(G^{n}_{>\eps_1},\cdot)$ converges to $\CoalFrag(G^{\infty}_{>\eps_1},\cdot)$ for the Skorokhod topology associated to $L_{2,GHP}$, which finishes the proof.
\end{proof}
Finally, let us prove Theorem~\ref{theo:existencecoalfrag}. First, notice that $\Glambda\in\calS^{graph}$ (it is a consequence of Theorem~\ref{theo:AlmostFellercoal}). For $\Frag(\Glambda,\cdot)$, Theorem~\ref{theo:existencecoalfrag} is a consequence of Theorem~\ref{theo:FellerGraphs} and Lemma~\ref{lemm:calSdansN}. The fact that $\Coal(\Glambda,\cdot)$  and $\CoalFrag(\Glambda,\cdot)$  are strong Markov processes with càdlàg trajectories in $\calS^{length}$ is a consequence of the convergences already proven and of Theorems~\ref{theo:AlmostFellercoal}, \ref{theo:FellerGraphs} and \ref{theo:almostFellerCoalFrag}. It remains to prove that almost surely, for any $t\geq 0$, $\Coal(\Glambda,t)$, and $\CoalFrag(\Glambda,t)$ belong to $\calS^{graph}$.

Proposition~\ref{prop:retournement} shows that $\Coal(\Glambda,t)$ has the same distribution as $\calG_{\lambda+t}$. Thus, for a fixed $t\geq 0$, $\Coal(\Glambda,t)$ is in $\calS^{graph}$ almost surely. But if $\Coal(\Glambda,t)$ is in $\calS^{graph}$, then $\Coal(\Glambda,s)$ and $\CoalFrag(\Glambda,s)$ are in $\calS^{graph}$ for any $s\leq t$, cf. Remark~\ref{rem:coalScoalfragS}. This proves that $\Coal(\Glambda,t)$ and $\CoalFrag(\Glambda,\cdot)$ have trajectories in $\calS^{graph}$ and ends the proof of Theorem~\ref{theo:existencecoalfrag}.

\section{Perspectives}
\label{sec:perspectives}

Today, there is a lot of results concerning convergence in distribution of rescaled critical random graphs to $\Glambda$ (see \cite{AdBrGolimit,BhamidiSenWang2017,BhamidiBroutinSenWang2014arXiv}), or to other scaling limits (see \cite{BhamidiHofstadSen2018,BhamidiDharaHofstadSen2017arxiv,BroutinDuquesneWang2018arXiv}). For a lot of them, there is a notion of critical window parametrized by a real number $\lambda$, and the graph components merge approximately like a multiplicative coalescent when $\lambda$ crosses the critical window. This is particularly obvious for critical percolation on the configuration model and inhomogeneous random graphs, \cite{BhamidiBroutinSenWang2014arXiv,BhamidiDharaHofstadSen2017arxiv}, but also parametrized rank-$1$ inhomogeneous random graphs \cite{BhamidiSenWang2017,BhamidiHofstadSen2018}. It is natural to expect the results of the present paper to apply to those random graphs. More precisely, I expect that analogues of Theorems~\ref{theo:cvcoalGnp}, \ref{theo:cvfragGnp} and \ref{theo:cvpercodynGnp} hold. Let me describe quickly what I believe to be straightworward and what I believe no to be.

Concerning fragmentation, in order to apply Theorem~\ref{theo:FellerGraphs}, one needs, informally, joint convergence of the graphs, their size and  their surplus. For instance, \cite{DharaHofstadLeeuwaardenSen2016arXiv} and \cite{BhamidiDharaHofstadSen2017arxiv} should be enough to prove directly through Theorem~\ref{theo:FellerGraphs} that the discrete fragmentation process on the critically percolated heavy-tailed configuration models of \cite{BhamidiDharaHofstadSen2017arxiv} converge to a fragmentation process on the limit. Since discrete fragmentation is itself a coupled percolation process (with a decreasing percolation parameter), this in fact will show convergence of the critical percolation process (with decreasing percolation parameter) to the fragmentation process on the limit. For percolation on  light-tailed configuration models at criticality, the analogous result should hold by resorting to the results of \cite{DharaHofstadLeeuwaardenSen2017} and \cite{BhamidiBroutinSenWang2014arXiv}. For rank-one inhomogeneous random graphs, whether light or heavy-tailed, \cite[Theorem~2.14]{BroutinDuquesneWang2018arXiv} should be enough to prove through Theorem~\ref{theo:FellerGraphs} that the discrete fragmentation process on critical rank-one inhomogeneous random graphs converge to a fragmentation process on the limit. 

Concerning coalescence, it seems to me that applying Theorem~\ref{theo:AlmostFellercoal} is not as straightworward, since it could require substantial work to prove the additional hypothesis \eqref{eq:supdiamunif}. For some models where the convergence of a height process is available, as rank-one inhomogeneous random graphs \cite{BroutinDuquesneWang2018arXiv}, there is some hope to mimick the proof of section~\ref{subsec:CoalER}, but not without substantial additional work, essentially to prove condition $(ii)$ of Lemma~\ref{lem:supdiamunif}. It could also be feasible for the configuration model with i.i.d degrees leading to stable graphs, whose convergence due to Conchon-Kerjan and Golschmidt is announced in \cite{GoldschmidtHaasSenizergues2018arXiv}.

Concerning dynamical percolation, once the difficulties explained above concerning fragmentation and coalescence will be resolved, it should be straightword to obtain a result analogous to Theorem~\ref{theo:cvpercodynGnp}. Also, I believe that a duality result analogous to Proposition~\ref{prop:retournement} should hold also in the heavy-tailed setting.

Finally, let us mention that a version of Theorem~\ref{theo:almostFellerCoalFrag} where the convergence of initial data and of the process would be with the same topology should be true by using a stronger topology. One should be able to do this with a distance compatible with the Feller property of the augmented coalescent proved in \cite{BhamidiBudhirajaWang2014}.

\section*{Acknowledgements}
I want to thank Christophe Leuridan for helpful discussions and three very consciencious referees for their careful reading of the manuscript, which led to a substantial improvement of this article.

\appendix
\section{Topologies for processes}
\label{sec:compactcv}

Let $(M,d)$ denote a separable, complete metric space. Let $\calF([0,\infty),M)$ (respectively $D([0,\infty),M)$) denote the set of functions from $[0,\infty)$ to $M$ (respectively c\`adl\`ag functions from $[0,\infty)$ to $M$). For $\omega_1$ and $\omega_2$ in $\calF([0,\infty),M)$, let us define
\[d_{c,k}(\omega_1,\omega_2):=\sup_{t\in[0,k]}d(\omega_1(t),\omega_2(t)))\]
and 
\[d_c(\omega_1,\omega_2):=\sum_{k\geq 1}2^{-k}(1\land d_{c,k}(\omega_1,\omega_2))\;.\]
  It is easy to see that $(\calF([0,\infty),M),d_c)$ and $(D([0,\infty),M),d_c)$ are complete metric spaces (not separable in general), and that a sequence $\omega^n=(\omega^{n}(t))_{t\in\RR^+}$ converges in this metric space to $\omega^\infty=(\omega^{\infty}(t))_{t\in\RR^+}$ if and only if for every $T>0$, 
$$\sup_{t\in[0,T]}|\omega^{n}(t)-\omega^\infty(t)|\xrightarrow[n\rightarrow +\infty]{}0\;,$$
whence the term ``topology of compact convergence''. A reference when $M=\RR$, with applications, is \cite[section V.5]{PollardConvergence}, but we shall not use it here.

Now, let us define $d_{S}$ the Skorokhod metric as in \cite[section~VI.1]{PollardConvergence}. For each $k\in\NN^*$,  $\omega_1$ and $\omega_2$ in $\calF([0,\infty),M)$, let $d_{S,k}(\omega_1,\omega_2)$ be the infimum of those values $\delta$ for which there exists grids $0=t_0<t_1<\ldots <t_r$ with $t_r\geq k$ and $0=s_0<s_1<\ldots <s_r$, with $s_r\geq k$ such that $|t_i-s_i|\leq \delta$ for $i=1,\ldots,r$ and
\[d(\omega_1(t),\omega_2(s))\leq \delta\quad \text{ if }\quad t_i\leq t< t_{i+1}\text{ and }s_i\leq s< s_{i+1}\;.\]
Then, let 
\[d_{S}(\omega_1,\omega_2):=\sum_{k\geq 1}2^{-k}(1\land d_{S,k}(\omega_1,\omega_2))\;.\]
Then,  $(D([0,\infty),M),d_S)$ is a separable metric space, cf. \cite[Theorem VI.6]{PollardConvergence}. Notice that for any $k$,
\begin{equation}
  \label{eq:dSkdc}
  d_{S,k}(\omega_1,\omega_2)\leq d_{c,k}(\omega_1,\omega_2)
\end{equation}
Notably, $d_{S}\leq d_c$, and thus the topology induced by $d_c$ on $D([0,\infty),M)$ is finer than Skorokhod's topology induced by $d_S$. Notice also that
\begin{equation}
  \label{eq:dcdck}
  d_{S}(\omega_1,\omega_2)\leq d_{S,k}(\omega_1,\omega_2)+2^{-k}\;.
\end{equation}
  Concerning measurability, $(\calF([0,\infty),M),d_c)$ (resp. $(D([0,\infty),M),d_c)$) will always be equipped with its projection $\sigma$-field $\calP_\calF$ (resp. $\calP_D$), the smallest $\sigma$-field making the projections $\pi_t$ measurable, where $\pi_t$ maps $\omega$ to $\omega(t)\in M$. Those $\sigma$-fields are included in the Borel $\sigma$-fields for $d_c$. But $\calP_D$ coincides with the Borel $\sigma$-field induced by the Skorokhod topology on $D([0,\infty),M)$, cf. \cite[Theorem VI.6]{PollardConvergence}.

In this article, we prove convergence in distribution of a sequence of processes $((\bfX_n(t))_{t\geq 0})_{n\geq 1}$ towards $(\bfX(t))_{t\geq 0}$  by exhibiting couplings showing essentially that the Lévy-Prokhorov distance for $d_c$ between the distributions of $\bfX_n$ and  $\bfX$ goes to zero as $n$ goes to infinity. This implies convergence in distribution for the Skorokhod metric. Everything needed is gathered in the following lemmas. Recall that $\PP^*$ denotes the outer measure associated to $\PP$.
\begin{lemm}
  \label{lemm:cvcadlag}
  Let $\omega^n$, $n\in\NNbar$ be random variables with values in $\calF([0,\infty),M)$, defined on the same complete probability space $(\Omega,\calF,\PP)$. Suppose that for every $n\in\NN$, $\omega^n\in D([0,\infty),M)$ almost surely, and that
  \[d_c(\omega^n,\omega^\infty)\xrightarrow[n\rightarrow +\infty]{\PP^*}0\;.\]
  Then, $\omega^\infty$ belongs to $ D([0,\infty),M)$ almost surely.
\end{lemm}
\begin{dem}
  Since $d_c(\omega^n,\omega^\infty)$ converges in probability to zero, one may extract a subsequence $\omega^{n_k}$, $k\geq 0$ such that $d_c(\omega^{n_k},\omega^\infty)$ converges to zero $\PP^*$-a.s, and hence (since $\PP$ is complete), $\PP$-a.s. But since almost surely, $\omega^n$ is càdlàg for any $n$, and any limit of a càdlàg sequence for $d_c$ is càdlàg, we obtain that $\omega^\infty$ is càdlàg $\PP$-a.s. 
\end{dem}

\begin{lemm}
  \label{lemm:compactcv}
  Let $\omega^n$, $n\in\NNbar$ be random variables with values in $D([0,\infty),M)$ such that for any $\eps>0$, any $k\in\NN^*$, there exists $N\in\NN$ such that for every $n\geq N$, there exists a coupling of $\omega^n|_{[0,k]}$ and $\omega^\infty|_{[0,k]}$ such that:
\[\PP[d_{S,k}(\omega^n,\omega^\infty)>\eps]\leq \eps\;.\]
Then, $\omega^n$ converges in distribution to $\omega^\infty$, for the Skorokhod topology.
\end{lemm}
\begin{dem}
  Let $F$ be a closed set in $(D([0,\infty),M),d_{S})$. It is sufficient to prove that
  \[\limsup_{n\rightarrow +\infty}\PP(\omega^n\in F)\leq\PP(\omega^\infty\in F)\;.\]
  Let $\eps>0$. Let us define:
  \[F^\eps:=\{x\in D([0,\infty),M)\,:\,\exists y\in F,\;d_{S}(x,y)\leq \eps\}\;,\]
  which is a closed set. We have:
  \[\PP(\omega^n\in F)\leq \PP(\omega^\infty\in F^{2\eps})+\PP(d_{S}(\omega^n,\omega^{\infty})>\eps)\;.\]
  Now, choose $k\in\NN$ such that $2^{-k}\leq \frac{\eps}{2}$. Then, using \eqref{eq:dSkdc},
  \[\PP(\omega^n\in F)\leq \PP(\omega^\infty\in F^{2\eps})+\PP(d_{S,k}(\omega^n,\omega^{\infty})>\frac{\eps}{2})\;.\]
  By hypothesis, the last term above is less than $\eps$ if $n$ is choosen large enough. Letting $\eps$ go to zero and using the fact that $\bigcap_{\eps>0}F^\eps=F$, we get the result.  
\end{dem}
\section{The set of isometry classes of finite measured semi-metric spaces}
\label{sec:calM}
Even though there is no \emph{set} of finite measured semi-metric spaces in the sense of Zermelo-Frankel set theory\footnote{See for instance \cite[Remark 7.2.5]{BuragoBuragoIvanov}: the problem is that if this class $\calC$ was a set, one can assign to it a finite semi-metric, $d_{GHP}$ and a measure $\mu$, so that $(\calC,d_{GHP},\mu)$ is a member of $\calC$, leading to Russell's paradox.}, $\calC/\calR$ can be considered as a set in the sense that there exists a true set of representatives of elements of $\calC$. The idea is to consider the set of all finite measured semi-metric spaces of a sufficiently large set $\UU$ such that $\UU$ contains a copy of any separable metric space, in the following (quite weak) sense:
\begin{hyp}
  \label{hyp:RN}
  For any separable metric space $M$, there exists a map $\phi$ from $M$ to $\UU$ such that $\phi$ is a bijection from $M$ to $\phi(M)$.
\end{hyp}
In Definition~\ref{def:setclass}, we make the (quite standard) choice $\UU=(\RR_+)^\NN$. Indeed, for any metric space $(M,d)$ with a dense sequence $(x_i)_{i\in\NN}$, the following function shows that Hypothesis~\ref{hyp:RN} is satisfied:
\[ \phi:\left\lbrace\begin{array}{lcl}M&\rightarrow& \UU\\ x&\mapsto &(d(x,x_i))_{i\in\NN}\end{array}\right.\]
\begin{defi}
\label{def:setclass}
Let $\UU=(\RR_+)^\NN$. Let $\calP$ denote the set of measured semi-metric spaces $\bfX=(X,d,\mu)$ such that:
\begin{itemize}
\item $X$ is a subset of $\UU$,
\item $d$ is a finite semi-metric on $X$.
\end{itemize}
We denote by $\calM$ the quotient $\calP/\calR$ of $\calP$ by the equivalence relation $\calR$, where:
$$\bfX\calR \bfX'\Leftrightarrow d_{GHP}(\bfX,\bfX')=0\;.$$
By abuse of language, we may call $\calM$ the ``set of equivalence classes of finite measured semi-metric spaces, equipped with the Gromov-Hausdorff-Prokhorov distance $d_{GHP}$''. 
\end{defi}
$\calM$ is a set of representatives of elements of $\calC$. Indeed, since every separable metric space is isometric to a subspace of $\UU$ equipped with a suitable metric through the function $\phi$ above, every member of $\calC$ will be at zero $d_{GHP}$-distance from some element of $\calP$, and even at zero $d_{GHP}$-distance from some compact element of $\calP$. Thus, for every member $\bfX $ of the class $\calC$, there is an element $[\bfX']$ of $\calM$ such that for any $\bfX''\in[\bfX]$, $d_{GHP}(\bfX,\bfX'')=0$.
Abusing notation, we shall denote by $[\bfX]$ the member of $\calM$ whose elements are at zero $d_{GHP}$-distance from $\bfX$. 

For our purpose, it is in fact not crucial to have Definition~\ref{def:setclass}, and one could reformulate all the results in this article in terms of sequences of random variables, at the expense of much more heavy statements.

Finally, one may define a set of equivalent classes of $\RR$-graphs as follows.
\begin{defi}
  Let $\calP^{graph}$ denote the subset of $\calP$ composed of measured semi-metric spaces which are semi-metric $\RR$-graphs. We denote by $\calM^{graph}$ the quotient $\calP^{graph}/\calR$ of $\calP^{graph}$ by the equivalence relation $\calR$, where:
$$\bfX\calR \bfX'\Leftrightarrow d_{GHP}(\bfX,\bfX')=0\;.$$
\end{defi}

\section{Commutation relations for coalescence and fragmentation}
\label{sec:commutcoalfrag}
\begin{lemm}
  \label{lemm:commutcoal0}
  Let $(X,d)$ be a semi-metric space. For any $A,B\subset X^2$, $\Coal_{0}(\bfX,A\cup B)=\Coal_{0}(\Coal_{0}(\bfX,A),B)$. 
\end{lemm}
\begin{proof}
  Let $Eq(A)$ denote the equivalence relation generated by $A$. Since $Eq(A\cup B)=Eq(Eq(A)\cup Eq(B))$, it is sufficient to prove the lemma when $A$ and $B$ are equivalence relations. So suppose that $A$ and $B$ are equivalence relations.  Let $x$ and $y$ be in $X$. Then,
  \[d_{A\cup B,0}(x,y)=\inf\sum_{i=1}^kd(p_i,q_i)\]
  where the infimum is over all $k\in\NN^*$, $p_1,q_1,\ldots,p_k,q_k$ such that $p_0=x$, $q_k=y$ and $(q_i,p_{i+1})\in Eq(A\cup B)$ for any $i$ in $\{1,\ldots,k-1\}$. Now, if $(q_i,p_{i+1})\in Eq(A\cup B)$, there exists $r_1,s_1,\ldots,r_l,s_l,r_{l+1}$ such that $r_1=q_i$, $s_l=p_{i+1}$ and
  \[\forall j\in\{1,\ldots,l\},\;(r_j,s_j)\in A\text{ and }(s_j,r_{j+1})\in B\]
  Then,
  \[d_X(p_i,q_i)=d(p_i,q_i)+\sum_{j=1}^ld_{A,0}(r_j,s_j)\;.\]
  And we obtain, denoting by $\tilde d$ the distance of $\Coal_0(\Coal_0(X,A),B)$,
  \[d_{A\cup B,0}(x,y)\geq \tilde d(x,y)\;.\]
  In the other direction, let $k\in\NN^*$, $p_1,q_1,\ldots,p_k,q_k$ and $\eps>0$ be such that:
  \[\tilde d(x,y)\leq \sum_{i=1}^kd_{A,0}(p_i,q_i)+\eps\]
  with $p_0=x$, $q_k=y$ and $(q_i,p_{i+1})\in B$ for any $i$ in $\{1,\ldots,k-1\}$. Then, let $h_1,\ldots,h_k$ and $p_{i,j}$, $q_{i,j}$, $i=1,\ldots,k$, $j=1,\ldots,h_i$ be such that for any $i\in\{1,\ldots,k\}$,
  \[d_{A,0}(p_i,q_i)\leq \sum_{j=1}^{h_i}d(p_{i,j},q_{i,j})+\frac{\eps}{k}\;,\]
  with $p_{i,1}=p_i$, $q_{i,h_i}=q_i$ and $(q_{i,j},p_{i,j+1})\in A$. Then we get
  \[\tilde d(x,y)\leq\sum_{i=1}^k\sum_{j=1}^{h_i}d(p_{i,j},q_{i,j})+2\eps\;.\]
  with $p_{1,1}=x$, $q_{k,h_k}=y$ and $(q_{i,j},p_{i,j+1})\in A\cup B$ with the convention $p_{i,h_i+1}=p_{i+1,1}$. Minimizing on $p$ and $q$, we get, for any $\eps>0$:
  \[\tilde d(x,y)\leq d_{A\cup B,0}(x,y) +2\eps\]
  which gives the result.  
\end{proof}
\begin{lemm}
  Let $(X,d)$ be a semi-metric space. For any multisets $A$, $B$ of elements of $X^2$ and any $\delta>0$, $\Coal_{\delta}(X,A\sqcup B)=\Coal_{\delta}(\Coal_{\delta}(X,A),B)$. 
\end{lemm}
\begin{proof}
  For a multiset $A$ of elements of $X^2$, let  us denote by $I_A$ the multiset of intervals added when $\delta>0$:
  \[I_A:=\bigsqcup_{(x,x')\in A}[a_{x,x'},b_{x,x'}]\;,\]
  and let $\tilde A$ denote the set of points to be identified:
  \[\tilde A:=\{(x,a_{x,x'})\,:\,(x,x')\in A\}\cup\{(x',b_{x,x'})\,:\,(x,x')\in A\}\]
  Then,
  \[\Coal_\delta(X,A)=\Coal_0(X\sqcup I_A,\tilde A)\;.\]
  Notice also that $I_{A\sqcup B}=I_A\sqcup I_B$ and $\widetilde{A\sqcup B}=\tilde A\sqcup \tilde B$. Now, 
  \begin{eqnarray*}
    \Coal_{\delta}(X,A\sqcup B)
    &=&\Coal_{0}(X\sqcup I_{A\sqcup B},\widetilde{A\sqcup B})\\
    &=&\Coal_{0}(X\sqcup I_{A}\sqcup I_B,\tilde A\sqcup \tilde B)\\
    &=&\Coal_{0}(\Coal_0(X\sqcup I_{A}\sqcup I_B,\tilde A),\tilde B)\\
    &=&\Coal_{0}(\Coal_0(X\sqcup I_{A},\tilde A)\sqcup I_B,\tilde B)\\
    &=&\Coal_{\delta}(\Coal_0(X\sqcup I_{A},\tilde A),B)\\
    &=&\Coal_{\delta}(\Coal_\delta(X,A),B),\\
  \end{eqnarray*}
  where we used Lemma~\ref{lemm:commutcoal0} at the third line.    
\end{proof}

\begin{lemm}
  \label{lemm:commutFrag}
  Let $X$ be a length-space. Then, for any $A$ and $B\subset X$, $ \Frag(\Frag(X,A),B)=\Frag(X,A\cup B)$.
\end{lemm}
\begin{proof}
  Notice first that if $x\in X\setminus A^d$ and $y\in B^d\setminus A^d$,
  \[d^{\Frag}_A(x,y)=0\Leftrightarrow d(x,y)=0\;.\]
  Thus,
  \[(A\cup B)^d=A^d\cup (B\setminus A)^{d^{\Frag}_A}\;.\]
Now, let $x$ and $y$ belong to $X$ and $\gamma$  a path from $x$ to $y$, indexed by $[0,1]$, such that
\[\gamma\cap A^d=\emptyset\;,\]
then we claim that
\begin{equation}
  \label{eq:pathfrag}
  \ell_{\Frag(X,A)}(\gamma)= \ell_X(\gamma)
\end{equation}
Indeed, trivially,
\[\ell_{\Frag(X,A)}(\gamma)\geq \ell_X(\gamma)\]
On the other hand, if $0=t_1\leq \ldots \leq t_n=1$ is a subdivision of $[0,1]$, $\gamma|_{[t_i,t_{i+1}]}$ is a path from $\gamma(t_i)$ to $\gamma(t_{i+1})$ in $X\setminus A^d$. Thus, 
\begin{eqnarray*}
  \sum_{i=1}^{n-1}d^{\Frag}_A(\gamma(t_i),\gamma(t_{i+1}))
  &\leq& \sum_{i=1}^{n-1}\ell_X(\gamma|_{[t_i,t_{i+1}]})\\
  &=& \ell_X(\gamma)
\end{eqnarray*}
which shows~\eqref{eq:pathfrag}. Finally, denoting by $\tilde d$ the distance on $\Frag(\Frag(X,A),B)$,
\begin{eqnarray*}
  d^{\Frag}_{A\cup B}(x,y)
  &=&\inf_{\substack{\gamma:x\rightarrow y\\ \gamma \cap (A\cup B)^d=\emptyset}}\ell_X(\gamma)\\
  &=&\inf_{\substack{\gamma:x\rightarrow y\\ \gamma \cap (A\cup B)^d=\emptyset}}\ell_{\Frag(X,A)}(\gamma)\\
  &=&\inf_{\substack{\gamma:x\rightarrow y \text{ in }\Frag(X,A)\\ \gamma \cap (B\setminus A)^{d^{\Frag}_A}=\emptyset}}\ell_{\Frag(X,A)}(\gamma)\\
  &=&\tilde d(x,y)\;.
\end{eqnarray*}
\end{proof}
Hereafter, we say that a path $\gamma$ in $\Coal_0(X,A)$ takes a shortcut $(a,b)$ in $A\subset X^2$ if $(a,b)\in A$ and $\gamma\cap \{a,b\}\not=\emptyset$.
\begin{lemm}
\label{lemm:commutdeter}
Let $(X,d)$ be a semi-metric length space and $A\subset X^2$ an equivalence relation. Suppose that for any $(x,y)\in X^2$, every simple rectifiable path in $\Coal_0(X,A)$ from $x$ to $y$ takes only a finite number of shortcuts in $A$. Let $B^d$ denote the set
\[\{x\in X\,:\,\exists y\in B,\;d(x,y)=0\}\;.\]
Then, for any $B\subset X$ such that $B^d\cap\{x\in X:\exists y\in X,\; (x,y)\text{ or }(y,x)\in A\}=\emptyset$, 
$$\Coal_0(\Frag(X,B),A)=\Frag(\Coal_0(X,A),B)\;.$$
\end{lemm}
\begin{dem}
Let $\ell_X(\gamma)$ denote the length of a path $\gamma$ in $X$. Let $d^{fragcoal}$ (resp. $d^{coalfrag}$, resp. $d^{frag}$) denote the distance of $\Frag(\Coal_0(X,A),B)$ (resp. $\Coal_0(\Frag(X,B),A)$, resp. $\Frag(X,B)$) on $X\setminus B^{d}$. We want to show that $d^{fragcoal}=d^{coalfrag}$. First, it is always true that $ d^{fragcoal}\leq d^{coalfrag}$. Indeed, let  $\{p_i\}$ and $\{q_i\}$, $i=0,\ldots ,k$ be such that $(q_i,p_{i+1})\in A$ for all $i= 0,\ldots,k-1$ and $ p_0 = x$, $q_k = y$. Then, the concatenation of $(k+1)$ paths $\gamma_i$ in $X$, $i=0,\ldots,k$ such that $\gamma_i$ goes from $p_i$ to $q_i$ and each path avoids $B^d$ gives a path in $\Coal_0(X,A)$ from $x$ to $y$ avoiding $B^d$. Thus, for any $x$ and $y$ in $X\setminus B^d$,
\begin{eqnarray*}
d^{fragcoal}(x,y)&\leq &\inf_{\substack{k,\{p_i\},\{q_i\}\\ \gamma_i:p_i\rightarrow q_i\\ \gamma_i\cap B^d=\emptyset}}\left(\sum_{i=1}^k\ell_X(\gamma_i)\right)\\
&=&\inf_{k,\{p_i\},\{q_i\}}\left(\sum_{i=1}^kd^{frag}(p_i,q_i)\right)\\
&=&d^{coalfrag}(x,y)\;.
\end{eqnarray*}
Let us show now that $d^{coalfrag}\leq  d^{fragcoal}$. Let $x$ and $y$ be in $X\setminus B^d$ and let $\gamma$ be a rectifiable simple path from $x$ to $y$ in $\Coal_0(X,A)$ such that $\gamma\cap B^d=\emptyset$. Then, $\gamma$ takes only a finite number of shortcuts in $A$. Thus, there exists $\{p_i\}$ and $\{q_i\}$, $i=0,\ldots ,k$ with $p_0=x$ and $q_k=y$ and paths $\gamma_i$, $i= 0,\ldots,k$ such that $(q_i,p_{i+1})\in A$ and $\gamma_i$ is a path from $p_i$ to $q_i$ in $X$ and $\gamma$ is the concatenation of $\gamma_0,\ldots,\gamma_k$. Thus, $\gamma_i\cap B^d=\emptyset$ for any $i$ and \begin{eqnarray*}
\ell_{\Coal_0(X,A)}(\gamma)&=&\sum_{i=1}^k\ell_X(\gamma_i)\\
&\geq &\sum_{i=1}^kd^{frag}(p_i,q_i)\\
&\geq &d^{coalfrag}(x,y)\\
\end{eqnarray*}
Taking the infimum over rectifiable simple paths $\gamma$ from $x$ to $y$ in $\Coal_0(X,A)$ such that $\gamma\cap B^d=\emptyset$ gives that $d^{coalfrag}\leq  d^{fragcoal}$.
\end{dem}
\section{$\RR$-graphs as $\RR$-trees with shortcuts}
\label{sec:shortcut}
 Here, we sketch the proof of Lemma~\ref{lemm:shortcut}. Let us start with the ``if'' direction. By induction, it is sufficient to show that if $(X,d)$ is an $\RR$-graph and $(x,y)\in X^2$, then the quotient metric space obtained from $\Coal_0(X,\{(x,y)\})$ is isomorphic to an $\RR$-graph. Notice that $\Coal_0(X,\{(x,y)\})$ is obviously a totally bounded and finite semi-metric space. Let $R>0$ be such that for any $x\in X$, $(B_R(x),d|_{B_R(x)})$ is an $\RR$-tree, where $B_R(x)$ is the  ball of radius $R$ and center $x$ in $X$.  Then, let $R':=\min\{R/2,d(x,y)/5\}$, $d'$ denote the distance on $\Coal_0(X,\{(x,y)\})$ and $B'_{R'}(z)$ be the  ball of radius $R'$ and center $z$ in $(X,d')$. If $R'=0$, then clearly $\Coal_0(X,\{(x,y)\})$ is equal to $X$, so let us suppose that $R'\neq 0$. The reader can check that for any $z\in X$, $(B'_{R'}(z),d'|_{B'_{R'}(z)})$ is a totally bounded acyclic geodesic finite semi-metric space, and the only pair of points at $d'$-distance zero is $(x,y)$. Its quotient metric space is thus an $\RR$-graph.

 Now, let us look at the ``only if'' direction. Let $(X,d)$ be an $\RR$-graph which is not an $\RR$-tree. The set of branchpoints in an $\RR$-tree is at most countable. Thus, one may find a point $x\in\core(X)$ such that $x$ belongs to a cycle, and $x$ is of degree $2$ in $X$. Let $Y_1$ and $Y_2$ denote the two components of $B_R(x)\setminus\{x\}$ (with $R>0$ small enough). Then, let $(X',d')$ denote the completion of $\Frag(X,\{x\})$. It is shown in \cite[section~7.1]{AdBrGoMi2017} that $X'$ adds exactly two points $x_{(1)}$ and $x_{(2)}$ to $\Frag(X,\{x\})$ and $d'$ can be described as follows:
 \begin{itemize}
 \item If $y,z\not\in \{x_{(1)},x_{(2)}\}$, then $d'(y,z)$ is the minimal length of a path from $y$ to $z$ in $X$ not visiting $x$.
 \item If $y\neq x_{(2)}$, then $d'(x_{(1)},y)$ is the minimal length of an injective path from $x$ to $y$ in $X$ which takes its values in the component $Y_1$ on some small initial interval, and similarly for $d'(x_{(2)},z)$ with $z\neq x_{(1)}$.
 \item $d'(x_{(1)},x_{(2))})$ is the minimal length of a cycle passing through $x$
 \end{itemize}
 The reader can check that $X'$ is an $\RR$-graph, with a surplus strictly smaller than $X$, and that $X$ is isomorphic to the quotient metric space associated to $\Coal(X',\{x_{(1)},x_{(2)}\})$. Thus, one may conclude by induction on the surplus of $X$.
\section{Total variation distance between Poisson random measures}
\begin{lemm}
  \label{lemm:dVTPoiss}
  Let $\calP(\mu)$ denote the distribution of a Poisson random measure with a finite intensity measure $\mu$ on a measurable space $(E,\calE)$. If $\nu$ is another finite measure on $E$, then
  \[\|\calP(\mu)-\calP(\nu)\|\leq 2\|\mu-\nu\|\]
\end{lemm}
\begin{dem}
  Suppose first that $\mu(E)=0$. Then,
  \[\|\calP(\mu)-\calP(\nu)\|=\|\delta_\emptyset-\calP(\nu)\|=1-e^{-\nu(E)}\leq \nu(E)\le 2\|\mu-\nu\|\;,\]
  and the lemma is proved. So suppose now that $\mu(E)$ and $\nu(E)$ are non-zero. Without loss of generality, suppose that $\mu(E)\leq \nu(E)$. Let $\pi$ denote an optimal coupling between $\frac{\mu}{\mu(E)}$ and $\frac{\nu}{\nu(E)}$. Let $N$ (resp. $N'$) be a Poisson random variable with parameter $\mu(E)$ (resp. $\nu(E)$) coupled in an optimal way, i.e:
  \[\PP(N\neq N')=\|\calP(\mu(E))-\calP(\nu(E))\|\]
 It is easy to see, for instance approximating the Poisson distribution by the binomial distribution, that 
 \[\|\calP(\mu(E))-\calP(\nu(E))\|\leq |\mu(E)-\nu(E)|\]
 Next, let $(X_i,Y_i)_{i\geq 1}$ be a sequence of i.i.d random variables of distribution $\pi$, independent from $(N,N')$. Then, $(\sum_{i=1}^N\delta_{X_i},\sum_{j=1}^{N'}\delta_{Y_i})$ is a coupling of $\calP(\mu)$ and $\calP(\nu)$. Notice also that $\sum_{i=1}^N\II_{X_i\not=Y_i}$ has Poisson distribution with parameter
 \[\mu(E)\PP(X_1\not=Y_1)=\mu(E)\|\frac{\mu}{\mu(E)}-\frac{\nu}{\nu(E)}\|\;.\]
  Thus,
 \begin{eqnarray*}
   \|\calP(\mu)-\calP(\nu)\|&\leq & \PP(\sum_{i=1}^N\delta_{X_i}\neq \sum_{j=1}^{N'}\delta_{Y_i})\\
   &\leq& \PP(N\not=N')+\PP(\exists i\in\{1,\ldots,N\},\; X_i\not=Y_i)\\
   &\leq & |\mu(E)-\nu(E)|+ 1-e^{-\mu(E)\|\frac{\mu}{\mu(E)}-\frac{\nu}{\nu(E)}\|}\\
   &\leq & \|\mu-\nu\|+ \mu(E)\|\frac{\mu}{\mu(E)}-\frac{\nu}{\nu(E)}\|
 \end{eqnarray*}
 Now, for any  $A\in\calE$, since $\mu(E)\leq \nu(E)$,
 \[\frac{\nu(A)}{\nu(E)}-\frac{\mu(A)}{\mu(E)}\leq \frac{\nu(A)-\mu(A)}{\mu(E)}\leq \frac{\|\mu -\nu\|}{\mu(E)}\]
 Thus,
 \[\|\frac{\mu}{\mu(E)}-\frac{\nu}{\nu(E)}\|\leq \frac{\|\mu -\nu\|}{\mu(E)}\]
 which gives the result. 
\end{dem}

\bibliographystyle{plain}


\end{document}